\documentclass[10pt,reqno,fleqn]{amsart}
\usepackage[foot]{amsaddr}

\author{Federico Cornalba$^*$}
\thanks{$^*$University of Bath, Claverton Down, BA2 7AY, Bath, United Kingdom. \\\emph{E-mail address:} \href{mailto:fc402@bath.ac.uk}{fc402@bath.ac.uk}}
\author{Julian Fischer$^\dag$}
\author{Jonas Ingmanns$^\ddag$}
\thanks{$^{\dag,\ddag}$Institute of Science and Technology Austria (ISTA), Am~Campus~1, 
3400 Klosterneuburg, Austria. 
\\ \emph{E-mail addresses:} \href{mailto:julian.fischer@ista.ac.at}{julian.fischer@ista.ac.at},
 \href{mailto:jonas.ingmanns@ista.ac.at}{jonas.ingmanns@ista.ac.at}}
\author{Claudia Raithel$^\S$}
\thanks{$^\S$Technische Universit\"at Dresden, 01217 Dresden, Germany. \\ \emph{E-mail address:}~\href{mailto:claudia.raithel@tu-dresden.de}{claudia.raithel@tu-dresden.de}}

\usepackage[hmarginratio=1:1,top=32mm,columnsep=20pt]{geometry} 
\usepackage[utf8]{inputenc}
\usepackage{amsmath}
\usepackage{graphicx}
\usepackage{amssymb}
\usepackage{esint}
\usepackage{color}
\usepackage{amsthm}
\usepackage{epsfig}
\usepackage[]{todonotes}
\usepackage{mathrsfs} 
\usepackage{xxcolor}%
\usepackage[pagebackref=false,linktocpage=true,colorlinks=true,linkcolor=Blue,citecolor=Green]{hyperref}
\usepackage{doi}
\usepackage{enumitem}
\usepackage{mathtools}
\usepackage{tikz}
\usetikzlibrary{arrows}
\usepackage[english]{babel}
\usepackage[square,numbers]{natbib}

\iffalse 
\usepackage[autostyle]{csquotes}
\usepackage[
    backend=biber,
    style=numeric,
    firstinits = true,
    isnb=false,
    issn=false,
    url=false, 
    doi=true,
    eprint=true
]{biblatex}

\renewbibmacro{in:}{
  \ifentrytype{article}{}{
  \printtext{\bibstring{in}\intitlepunct}}}

\DeclareFieldFormat[article,unpublished,book]{pages}{#1}  

\DeclareFieldFormat[article]{title}{#1}
\DeclareFieldFormat[unpublished]{title}{#1}
\DeclareFieldFormat[misc]{title}{#1}
\fi 

\newtheorem{theorem}{Theorem}
\newtheorem{proposition}[theorem]{Proposition}
\newtheorem{lemma}[theorem]{Lemma}
\newtheorem{corollary}[theorem]{Corollary}
\newtheorem{definition}[theorem]{Definition}

\newtheorem*{theorem*}{Theorem}
\newtheorem*{maintheorem*}{Main Result}
\newtheorem*{theorem**}{{``Theorem''}}
\theoremstyle{definition}

\theoremstyle{definition}

\newtheorem{rem}[theorem]{Remark}

\reversemarginpar

\newtheorem{innercustomgeneric}{\customgenericname}
\providecommand{\customgenericname}{}
\newcommand{\newcustomtheorem}[2]{%
  \newenvironment{#1}[1]
  {%
   \renewcommand\customgenericname{#2}%
   \renewcommand\theinnercustomgeneric{##1}%
   \innercustomgeneric
  }
  {\endinnercustomgeneric}
}

\newcustomtheorem{customthm}{Assumption}
\newcustomtheorem{customthm2}{Definition}

\def\XXint#1#2#3{{\setbox0=\hbox{$#1{#2#3}{\int}$ }
\vcenter{\hbox{$#2#3$ }}\kern-.6\wd0}}

\definecolor{Yellow}{rgb}{0.95,0.9,0.0} 
\definecolor{Red}{rgb}{0.8,0.1,0.1}
\definecolor{Green}{rgb}{0.1,0.65,0.2}
\definecolor{Blue}{rgb}{0.1,0.1,0.8}
\definecolor{Purple}{rgb}{0.7,0.1,0.7}
\definecolor{Grey}{rgb}{0.6,0.6,0.6}

\newcommand{\ov}{\overline}

\newcommand{\eps}{\varepsilon}
\newcommand{\ep}{\epsilon}
\newcommand{\vphi}{\varphi}
\newcommand{\vtheta}{\vartheta}
\newcommand{\delflex}{\kappa}
\newcommand{\delZero}{{\delta_0}}
\newcommand{\et}{\theta}

\newcommand{\N}{\mathbb{N}}
\newcommand{\Nz}{\mathbb{N}_0}
\newcommand{\Z}{\mathbb{Z}}

\newcommand{\R}{\mathbb{R}}

\newcommand{\domain}{{\mathbb{T}^d}}
\newcommand{\compemb}{\subset\subset}

\newcommand{\divv}{\nabla\cdot}
 
\newcommand{\dx}{\,\mathrm{d}x}
\newcommand{\dt}{\,\mathrm{d}t}

\newcommand{\f}{\boldsymbol}
\newcommand{\m}{\,\mathrm{d}}				

\newcommand{\Ev}{\mathbb{E}}
\newcommand{\Pm}{{\mathbb{P}}}

\newcommand{\Ghd}{G_{h,d}} 				
\newcommand{\LzGhd}{{L^2(\Ghd)}}
\newcommand{\hbase}[2][ ]{{f^{#1}_{#2}}} 
\newcommand{\HGhd}[1][ ]{{H^{#1}(\Ghd)}}
\newcommand{\Ih}{\mathcal{I}_{h}}		

\newcommand{\dOne}[2][h]{{\partial_{#1,1}^{#2}}}
\newcommand{\gradhR}{{\nabla_h^R}}
\newcommand{\dhRx}[1][\ell]{{\partial_{h,x_{#1}}^R}}

\newcommand{\nS}{{n_S}}
\newcommand{\numtest}{{K}}
\newcommand{\rI}{{r_I}}
\newcommand{\sI}{{s_I}}
\newcommand{\V}[1][ ]{{V^\rI_{#1}}}				
\newcommand{\rhor}[1][ ]{{\rho^\rI_{#1}}}			
\newcommand{\rhob}[1][ ]{{\ov{\rho}_{#1}}}		
\newcommand{\rhobr}[1][ ]{{\ov{\rho}^\rI_{#1}}}	
\newcommand{\rhoz}[1][ ]{{\rho^0_{#1}}}			
\newcommand{\rhobz}[1][ ]{{\ov{\rho}^0_{#1}}}	
\newcommand{\U}[1][ ]{{U^\rI_{#1}}}
\newcommand{\etat}[1][ ]{{\eta^t_{#1}}}
\newcommand{\phit}[1][ ]{{\phi^t_{#1}}}
\newcommand{\phith}[1][_]{{\phi^t_{\if_#1 h \else h,#1 \fi}}}
\newcommand{\empmeas}[1][ ]{{\mu_{{#1}}^{\rI,N}}}
\newcommand{\bT}{{\f{T}}}
\newcommand{\bTlin}[1][ ]{{\breve{\f{T}}_{#1}}}
\newcommand{\bvphi}{{\f{\vphi}}}
\newcommand{\bphi}[1][ ]{{\f{\phi}_{#1}}}
\newcommand{\bphit}[1][ ]{{\f{\phi}^t_{#1}}}
\newcommand{\bphith}[1][_]{{\f{\phi}^t_{\if_#1 h \else h,#1 \fi}}}
\newcommand{\bphittil}[1][ ]{{\widetilde{\f{\phi}}{}^t_{#1}}}
\newcommand{\bphitlin}[1][ ]{{\breve{\f{\phi}}{}^t_{#1}}}
\newcommand{\bphilin}[1][ ]{{\breve{\f{\phi}}{}_{#1}}}
\newcommand{\psit}[1][ ]{{\psi^t_{#1}}}

\newcommand{\psitlin}[1][ ]{{\breve{\psi}{}^t_{#1}}}
\newcommand{\psilin}[1][ ]{{\breve{\psi}{}_{#1}}}
\newcommand{\bzeta}{{\f{\zeta}}}
\newcommand{\bzetat}[1][ ]{{\f{\zeta}^t_{#1}}}
\newcommand{\bzetath}[1][_]{{\f{\zeta}^t_{\if_#1 h \else h,#1 \fi}}}
\newcommand{\bzetatTsh}[1][_]{{\f{\zeta}^{t\wedge\Ts}_{\if_#1 h \else h,#1 \fi}}}
\newcommand{\zetat}[1][ ]{{\zeta^t_{#1}}}

\newcommand{\zetatTsh}[1][_]{{\zeta^{t\wedge\Ts}_{\if_#1  h \else h,#1 \fi}}}
\newcommand{\tilQ}[1][ ]{{Q^{\rI,N}_{#1}}}
\newcommand{\Tbh}{{\overline{T}_h}}
\newcommand{\Ts}{{T_\oslash}}
\newcommand{\tTs}[1][ ]{{\widetilde{T}{}_\oslash^{#1}}}

\newcommand{\Errstop}[1][_]{{{\mathrm{Err}}_{\if_#1  \oslash \else \oslash,#1 \fi}}}

\newcommand{\M}{{\mathcal{M}}}
\newcommand{\G}{{\mathcal{G}}}

\allowdisplaybreaks[2] 		

\numberwithin{equation}{section}	
\numberwithin{theorem}{section}

\allowdisplaybreaks[2]

\usepackage{delimdelim}

\begin{document}

\title[Density fluctuations in interacting particle systems]{Density fluctuations in weakly interacting particle systems via the Dean--Kawasaki equation}


\begin{abstract}
The Dean--Kawasaki equation -- one of the most fundamental SPDEs of fluctuating hydrodynamics -- has been proposed as a model for density fluctuations in weakly interacting particle systems. In its original form it is highly singular and fails to be renormalizable even by approaches such as regularity structures and paracontrolled distributions, hindering mathematical approaches to its rigorous justification. It has been understood recently that it is natural to introduce a suitable regularization, e.\,g.,\ by applying a formal spatial discretization or by truncating high-frequency noise.

In the present work, we prove that a regularization in form of a formal discretization of the Dean--Kawasaki equation indeed accurately describes density fluctuations in systems of weakly interacting diffusing particles: We show that in suitable weak metrics, the law of fluctuations as predicted by the discretized Dean--Kawasaki SPDE approximates the law of fluctuations of the original particle system, up to an error that is of arbitrarily high order in the inverse particle number and a discretization error.
In particular, the Dean--Kawasaki equation provides a means for efficient and accurate simulations of density fluctuations in weakly interacting particle systems.

\end{abstract}

\maketitle



\tableofcontents

\section{Introduction}
The theory of \emph{Fluctuating Hydrodynamics} \cite{SpohnBook} describes the dynamics of large, finite-size particle systems subject to fluctuations. In the framework of this theory, the particle system being investigated is described via a suitable stochastic PDE (SPDE), which captures the fluctuations of the system on top of its deterministic limiting dynamics.
The physics applications of this theory are numerous and diversified, and touch upon several different fields, see for instance \cite{bertini2015macroscopic,marconi1999dynamic,goddard2012general,velenich2008brownian,
delfau2016pattern,donev2014reversible,lutsko2012dynamical,
dejardin2018calculation,duran2019instability,
cates2015motility,thompson2011lattice}. 


This work is concerned with giving a fully quantitative justification to a pivotal SPDE from fluctuating hydrodynamics, the so-called \emph{Dean--Kawasaki equation} \cite{Dean,Kawasaki}
\begin{align}\label{dk_self_int}
	 &	\partial_t \rho = 
	 			\sigma\Delta\rho
	 			+\divv\left(\rho\nabla V\ast\rho\right)+ N^{-1/2}\divv (\sqrt{2\sigma \rho}\,\xi).
\end{align}
Here, $\rho$ denotes the density of particles, $N$ is the number of particles, $V$ is an interaction potential, $\sigma>0$ is the diffusion coefficient, and $\xi$ denotes space-time vector-valued white noise. 
The model \eqref{dk_self_int} is proposed as a mesoscopic description for the law of the empirical density
\begin{align*}
	\mu^N(\cdot,t) := \frac{1}{N} \sum_{i=1}^N \delta_{X_i^N(t)}
\end{align*}
of a system of $N\gg 1$ particles with positions $X_i^N(t)\in \domain$ driven by i.i.d.\ Brownian motions $(B_i)_{i=1}^N$ and interacting weakly via a smooth potential $V$:
\begin{equation}\label{eq_intro_PartSys}
	\m X^{N}_{i}(t) =
	- N^{-1} \sum_{j=1}^N\nabla V\left(X^{N}_{i}(t)-X^{N}_{j}(t)\right)\m t + \sqrt{2\sigma}\m B_{i}(t).
\end{equation}
The purpose of the Dean--Kawasaki equation \eqref{dk_self_int} is to correctly describe the law of particle density fluctuations, going beyond the (deterministic) mean-field description of the particle density
\begin{align*}
\partial_t \bar \rho = \sigma \Delta \bar \rho + \divv\left(\bar \rho \nabla V \ast \bar \rho\right).
\end{align*}
As we shall see, the Dean--Kawasaki equation \eqref{dk_self_int} turns out to be even substantially superior in accuracy as compared to the leading-order description of fluctuations by the process
\begin{align}\label{lin_dk}
\partial_t \hat \rho = \sigma \Delta \hat \rho + \divv\left(\hat \rho \nabla V \ast \hat \rho\right)+ N^{-1/2}\divv (\sqrt{2\sigma \bar\rho}\,\xi).
\end{align}

The Dean--Kawasaki equation \eqref{dk_self_int} itself is a highly singular SPDE; it is not even renormalizable by approaches like regularity structures or paracontrolled distributions.
As shown in the seminal work \cite{konarovskyi2020dean}, the singular SPDE \eqref{eq_intro_PartSys} turns out to be in fact a formal -- mathematically equivalent -- rewriting of the associated microscopic particle system \eqref{eq_CDPartSys}: All of its martingale solutions are precisely given as the empirical density of an interacting particle system of the form \eqref{eq_intro_PartSys}, i.e., the only solutions to \eqref{dk_self_int} are of the form
\begin{align}\label{vonRenesseEtAl}
\rho(\cdot,t) \equiv \mu^N(\cdot,t) := N^{-1}\sum_{i=1}^{N}\delta_{X^N_i(t)}
\end{align}
for some $X^N_i$ satisfying the system of SDEs \eqref{eq_intro_PartSys}.
While this might -- at first glance -- appear to imply that the representation \eqref{dk_self_int} brings no additional insight over the associated particle dynamics \eqref{eq_intro_PartSys}, it has been observed in recent years that this is in fact not the case: Natural regularizations of the formal Dean--Kawasaki SPDE \eqref{dk_self_int} are better-behaved objects and may be expected to provide meaningful approximations of density fluctuations. Note that in any practical use for physics simulations, the Dean--Kawasaki equation \eqref{dk_self_int} is necessarily subjected to a regularization, either by applying a spatial numerical discretization (which implicitly truncates the high-frequency noise modes) or by introducing an explicit frequency cutoff in the noise term. While regularised models may differ from one another in terms of specific features and applicability, they are all usually much more tractable and versatile versions of \eqref{dk_self_int}, and capable of describing -- at least some -- features of the underlying microscopic systems up to a quantifiable, small error. 

In the present work, we prove that upon regularizing the Dean--Kawasaki equation \eqref{dk} by applying a formal spatial discretization, it is capable of approximating the law of density fluctuations of the weakly interacting particle system \eqref{eq_CDPartSys} up to \emph{arbitrary} precision in $N^{-1}$ (plus numerical errors). Note that any meaningful comparison of the empirical density $\mu^N$ -- a sum of Dirac measures -- to the continuous solution $\rho_h$ of a regularized variant of the Dean--Kawasaki equation \eqref{dk} must be formulated in terms of weak spatial norms, i.e.,\ in terms of testing the densities against a test function with sufficient regularity. Furthermore, as $\mu^N$ and $\rho_h$ live on different probability spaces, any comparison of $\rho_h$ to $\mu^N$ can only be phrased in terms of their law (and not as a pathwise statement).  The informal statement of our result is as follows.
\begin{maintheorem*}[Informal statement for high-order approximation of density fluctuations in the single-species case of Theorem \ref{MainThm} below] 
Let $T>0$ be a fixed time.
Let $\rho_h$ denote the solution to a suitable finite-difference discretisation of \eqref{dk_self_int} on $\domain \times [0,T]$ with spatial discretisation parameter $h$ and of spatial order $p+1$. 
Let $\mu^N$ denote the empirical density of the particle system \eqref{eq_intro_PartSys} consisting of $N$ particles. 
Suppose that the initial density $\rho_h(\cdot,0)\geq \inf_x \rho_h(x,0)>0$ is strictly positive and approximates (in law) the empirical density $\empmeas(\cdot,0)$. 
Furthermore, let the overline symbol $(^-)$ denote suitable ``mean-field'' analogues of the densities. 
For the scaling regime, assume that 
\begin{equation}\label{InformalScaling_N_h}
		N^{1-\delZero}h^d \geq 1 \quad\text{ for some }\delZero>0. 
\end{equation}
Then, for every $0<\ep<\delZero/4$, there exists a stopping time $\Ts$ with
\begin{align}\label{ExpDecayStopTime}
	\mathbb{P}(\Ts < T) 
	\lesssim
	\exp\left( - C N^{\ep/2} \right),
\end{align} 
where for every $\delflex>0$ there exists $N_0\in\N$, such that for every $t\in[0,T]$ and $N\geq N_0$,
\begin{align}\label{InformalResult}
	& d_{weak,2j+1}\left[	N^{1/2}(\rho_{h} - \rhob[h])(t \wedge \Ts),\,		N^{1/2}(\mu^N - \rhob)(t )\right] \\
	& \quad \lesssim 		N^{2\ep+\delflex}\left(h^{p+1} + \exp\big(-CN^{\ep/2}\big) + N^{-j(1/2-2\ep)}\right) \nonumber\\
	& \quad \eqqcolon			Err_{num} + Err_{\oslash} + Err_{fluct,rel} \nonumber
\end{align}
holds for all $j\in\mathbb{N}$. Here, $d_{weak,2j+1}$ is a negative Sobolev-type distance of order $-2j-1$.
\end{maintheorem*}
This result is desirable for three reasons:
\begin{itemize}
\item	The approximate scaling regime $Nh^d\gg 1$ (cfr. \eqref{InformalScaling_N_h}) is extremely relevant, as it corresponds to an -- on average -- large amount of particles per grid cell.
		Hence, it is the regime in which the direct particle simulation is more expensive than the numerical SPDE model.
\item 	The estimate \eqref{InformalResult} shows that the (discrete) Dean--Kawasaki equation provides an accurate description of the underlying particle system, in the sense that the error due to the finiteness of the number of particles is -- in suitable weak norms -- of arbitrarily high order in $N^{-1}$ and thus typically dominated by the numerical error $N^{2\ep+\delflex}h^{p+1}$.
\item 	By \eqref{ExpDecayStopTime}, the stopping time 
			(whose role in the proof is to guarantee that the fluctuations $\rho_{h}-\rhob[h]$ roughly stay in the regime of classical mean field fluctuations, and that $\rho_h$ stays positive) 
		runs short of the final time $T>0$ only with exponentially small probability.
\end{itemize}

The key novelty of our result is the derivation of an error estimate for an \emph{interacting} particle system that is \emph{of arbitrary order} in the inverse particle number $N^{-1}$ (plus numerical error $Err_{num}$ and modelling error $Err_{\oslash}$): In particular, our result is the first to show the Dean--Kawasaki equation is a far more accurate descriptor of fluctuations than the leading-order equation \eqref{lin_dk} in the case of weakly interacting particle systems. 

In the case of independent Brownian particles, i.e., the case $V\equiv 0$ in \eqref{dk_self_int}, the corresponding result has been proven in \cite{cornalba2021dean}.
Again in the non-interacting particle case, a short proof of a quantitative error bound of the order $O(N^{-1/(d+2)})$ has recently been given in \cite{djurdjevac2022weak}: Notably, the result in \cite{djurdjevac2022weak} does not require a positivity lower bound on the initial density profile.
Previously, a low-order error estimate $O(N^{-\beta})$ for some $\beta<\tfrac{1}{2}$ had been established in \cite{FehrmanGess}, covering also the case of weakly interacting particles.

Let us mention that, having in mind a future application to cross-diffusion limits, throughout the present work we in fact consider a somewhat more general interacting particle system: We allow for multiple species of diffusing particles that interact with each other via possibly mildly (singularly) rescaled potentials $\V[\alpha\beta](\cdot)=\rI^{-d}V(\cdot/\rI)$ with interaction length scale $\rI=\rI(\log N)$. The associated particle dynamics is given by the system of SDEs \eqref{eq_CDPartSys} below, while the corresponding analogue of the Dean--Kawasaki equation is given by the system
\begin{align}\label{dk}
	 &	\partial_t \rhor[\alpha] = 
	 			\sigma_\alpha\Delta\rhor[\alpha]
	 			+\divv\left(\sum_{\beta=1}^\nS\rhor[\alpha]\nabla\V[\alpha\beta]\ast\rhor[\beta]\right)+ N^{-1/2}\divv (\sqrt{2\sigma_\alpha \rhor[\alpha]}\,\xi_\alpha)
\end{align}
for $\alpha,\beta\in \{1,\cdots,\nS\}$, $\nS$ being the number of species. However, we emphasize that our results are already new and relevant in the single-species case $n_S=1$ and for $\rI\equiv 1$.

The most remarkable feature of \eqref{dk} is the linear cross-variation structure of the noise (with respect to the density $\rho_\alpha$). More precisely, as we will discuss thoroughly, the It\^o differential of the process $\int_{\domain}{(\rho_\alpha - \overline{\rho}_{\alpha})\varphi_1}\int_{\domain}{(\rho_\alpha - \overline{\rho}_{\alpha})\varphi_2}$, for smooth test functions $\varphi_1,\varphi_2$, has martingale component with cross-variation given by
\begin{align}\label{CovStructure}
 N^{-1}\int_{\domain}{\rho_\alpha\nabla\varphi_1 \cdot \nabla\varphi_2}.
\end{align}
One of the key steps of this contribution is -- in a nutshell -- to suitably reduce the analysis of the (nonlinear) convolutional term $\divv\big(\sum_{\beta=1}^\nS\rho_\alpha\nabla\V[\alpha\beta]\ast\rho_\beta\big)$ in \eqref{dk} to the one of the (linear) cross-variation structure \eqref{CovStructure}.

%

\subsection{Related literature} 

\subsubsection{Dean--Kawasaki model}

As discussed around \eqref{vonRenesseEtAl}, the Dean--Kawasaki equation in its original form is a rather rigid mathematical object, as it only allows for the empirical particle system as solution: this result is given in \cite{konarovskyi2020dean} and, also, in the (earlier) analogue version for non-interacting particles \cite{konarovskyi2019dean}. The papers \cite{konarovskyi2019dean,konarovskyi2020dean} are related to a series of works \cite{andres2010particle,von2009entropic,konarovskyi2017reversible,konarovskyi2019modified,schiavo2022dirichlet,konarovskyi2020conditioning,marx2021infinite} which -- among others -- shed light on the rigid interplay which deterministic and stochastic components of the Dean--Kawasaki model have to abide to (in the context of a suitable stochastic Wasserstein gradient flow).

The aforementioned rigidity of the Dean--Kawasaki model can be broken once suitable regularisations (such as, for instance, noise smoothing or truncation or a discretization) are introduced: 
A substantial number of works belong to this ever-growing framework. 

In \cite{dirr2020conservative}, a rigorous justification of the SPDE of fluctuating hydrodynamics for the simple exclusion process is provided (with leading order dynamics, although the noise term in this case is nonlinear and therefore much more challenging), together with convergence results concerning the rate functions for large deviation principles.
Rigorous links between Dean--Kawasaki type models and large-deviation principles for zero-range processes (and associated thermodynamic setting) are given in \cite{dirr2016entropic}.

In \cite{fehrman2019well}, the well-posedness theory (in a suitable kinetic formulation) of stochastic porous media and fast diffusion equations driven by nonlinear, conservative noise is provided. The generation of a random dynamical system is also discussed. The latter topic is expanded and enriched with uniqueness of invariant measures and mixing for the associated Markov process in \cite{fehrman2022ergodicity}. Similar results are also derived in the case of correlated noise \cite{fehrman2021well}. For the same noise, derivation of underlying microscopic dynamics is given in \cite{ding2022new}.

Rates of convergence of the discretized Dean--Kawasaki dynamics towards the particle system in the case of independent Brownian particles are discussed in \cite{cornalba2021dean}. 
In the recent paper \cite{djurdjevac2022weak}, the authors prove weak error estimates in the non-interacting particle case, but they use a suitable SPDE approximation of the Dean--Kawasaki model rather than a discrete numerical approximation of the same. Their mathematical approach (which is centered around Laplace duality arguments and Kolmogorov backwards equation techniques) is somewhat complementary to ours, and it is safe to say that the current work and  \cite{djurdjevac2022weak} have different points of strength. 
Unlike the results in this work, those in \cite{djurdjevac2022weak} allow for general initial particle profiles, and provide non-negativity of the solution (in addition to other  well-posedness properties, including a comparison principle and entropy estimates): 
However, the weak error accuracy in \cite{djurdjevac2022weak} in terms of $N$ is capped (in terms of relative error) by $N^{-1/(d/2+1)}\log N$ (which gets worse with the spatial dimension $d$), while our fluctuation rate \eqref{InformalResult} can be arbitrarily high.

Recently, conservative stochastic PDEs sharing strong similarities with Dean--Kawasaki models have been proven to be limit of stochastic interacting particle systems in the mean-field limit (e.g., the case of stochastic gradient descent dynamics in overparametrised, shallow neural networks is covered in \cite{gess2022conservative} with optimal convergence rates provided for both convergence and associated Central Limit Theorem).

For regularised Dean--Kawasaki models of inertial type (i.e., models capturing in both density and momentum density), high-probability well-posedness for both independent and weakly interacting particle systems is discussed in \cite{Cornalba2019a, Cornalba2021al, cornalba2021well}. 

The Dean--Kawasaki model is becoming more and more widespread in physics applications (see, for instance, \cite{cates2015motility,thompson2011lattice,lutsko2012dynamical,marconi1999dynamic,velenich2008brownian,goddard2012general,delfau2016pattern,dejardin2018calculation,donev2014reversible,duran2019instability}). 
Consequently, works devoted to numerical approximations of the Dean--Kawasaki model are on the rise. 
Among such contributions on the numerical side, we mention structure-preserving finite difference and finite element schemes for high-order fluctuation bounds of non-interacting diffusing particles \cite{cornalba2021dean}, analysis of finite element discretisations in the context of reaction-diffusion (agent-based) systems models \cite{helfmann2021interacting, kim2017stochastic}, 
analysis of finite differences discretisations of agent-based models describing co-evolving opinion and social dynamics under the influence of multiplicative noise \cite{djurdjevac2022feedback},
finite-volume schemes for stochastic gradient flow equations \cite{russo2021finite, donev2010accuracy}, 
full reconstruction of dissipative operators in gradient flow equations from particle fluctuations \cite{li2018find}, convergence of finite element schemes for a weak formulation 
of suitable smoothed Dean--Kawasaki model \cite{bavnas2020numerical}, 
convergence analysis of discountinuous Galerkin scheme -- and modelling -- for the regularised inertial Dean--Kawasaki model \cite{cornalba2022regularised}.

\subsubsection{The mean-field limit of \eqref{eq_CDPartSys} and cross-diffusion systems} 
The analysis of the Dean--Kawasaki model \eqref{dk} is naturally built on top of the mean-field dynamics of the particle system \eqref{eq_CDPartSys}.
The study of the mean-field dynamics of systems of SDEs goes back to the 80s (see, e.g., the reviews \cite{G, JW}). In the late 80s Oelschl\"ager obtained a deterministic nonlinear diffusion process as the mean-field limit of a weakly interacting particle system \cite{O2}. In his subsequent work \cite{O1}, reaction-diffusion systems are derived from moderately interacting particle systems -- in fact, the case of cross-diffusion is included, however, with a positive-definiteness assumption on the diffusion coefficients. Quadratic porous-medium-type equations are derived from moderately interacting particle systems in \cite{O3}.  The methods of Oelschl\"ager were significantly extended by Stevens in \cite{Stevens} to derive a chemotaxis system. In \cite{JM_1998} the mean-field limit and fluctuations of a moderately interacting particle system with nonlinear diffusion coefficients is studied.
Further contributions include, e.\,g., the derivation of a two-phase Stefan problem as the mean-field limit of a master equation \cite{IRS, KS}, the 2-species Maxwell-Stefan model for the diffusion of gaseous mixtures as the hydrodynamic limit of two (singularly) interacting Brownian motions \cite{I}, nonlocal Lotka-Volterra systems with cross-diffusion as limits of a suitable Markov process \cite{FM}, as well as Shigesada-Kawasaki-Teramoto type cross-diffusion systems \cite{CDHJ}.

Daus, Chen, and J\"ungel have shown in \cite{chen2019rigorous} that the limiting behavior of our SDE system \eqref{eq_CD-IMFL} under the simultaneous limit $N\rightarrow\infty$ and $\rI \sim (\log N)^{-\beta} \rightarrow 0$ is captured by a cross-diffusion system \cite[(1)]{chen2019rigorous}; we refer to \cite{Galiano_Selgas} for an earlier more restrictive result and to \cite{jungel2022nonlocal} for a recent extension.
In the present paper we will only concern ourself with the first limit $N \rightarrow \infty$, yielding a mean-field limit of the form \eqref{eq_CD-IMFL}; however, many intermediate results are formulated to be of use also in a future work focusing on this cross-diffusion limit, by being established, e.g., uniformly in $\rI$ for a suitable range $(\log N)^c \lesssim \rI \lesssim 1$.

\section{Main Result}\label{SecMainRes}

In our main result, we rigorously quantify the distance between (a) the law of the density fluctuations arising in the interacting particle system \eqref{eq_CDPartSys}, and (b) the law of density fluctuations in a suitable numerical discretisation of the Dean--Kawasaki SPDE \eqref{dk}. In order to compare these laws, we make use of the family of distances between $\mathbb{R}^K$-valued  random variables $X,Y$ given by
\begin{align}\label{WeakDistance}
d_{-j}[X,Y] := \sup_{\psi\colon \max_{0\leq \tilde{j}\leq j}{\|D^{\tilde{j}}\psi\|_{L^{\infty}(\mathbb{R}^K)}} \leq 1}{\left|\mean{\psi(X)} - \mean{\psi(Y)}\right|}, \qquad j\in\mathbb{N}.
\end{align}
In other words, the distance $d_{-j}$ is a negative Sobolev distance acting on the probability distributions of the arguments. Note that the distance $d_{-1}$ essentially corresponds to the $1$-Wasserstein distance (up to the zero-th order bound $\|\psi\|_{L^{\infty}}\leq 1$), stated in dual formulation.

While the full notation is given below in Section \ref{SecSettNotAss}, we briefly define the remaining minimal ingredients needed to state our main theorem, namely: $\Ghd$ (the uniformly spaced grid on $\domain$ with grid-size parameter $h>0$); $(\cdot,\cdot)_h$ (the standard inner product in $L^2(\Ghd)$); $\langle\cdot,\cdot\rangle$ (the standard measure/function duality), and $\Ih$ (the operator interpolating continuous functions on the grid points of $\Ghd$). 

Our main theorem reads as follows.

\begin{theorem}[High-order approximation of density fluctuations in weakly interacting particle systems]\label{MainThm}
Let $T>0$ and $N\in \mathbb{N}$.
On $[0,T]$ let
\begin{itemize}[noitemsep, topsep=0pt, left= 10pt]
	\item $\empmeas=(\empmeas[\alpha])_{\alpha=1}^\nS$ as in \eqref{vonRenesseEtAl} be the empirical measures of the cross-diffusing particle system \eqref{eq_CDPartSys} satisfying Assumption \ref{ass_CDsystem+MFL} (particle system),
	\item $\rhobr$ be the intermediate mean-field limit as given in \eqref{eq_CD-IMFL}, and satisfying Assumption \ref{ass_RegularityMFL} (existence and regularity of continuous mean-field limit),
	\item $\rhor[h]$ be a solution to the discretised Dean--Kawasaki model \eqref{eq_discretizedDK} below in the context of a standard finite-difference discretisation of order $p+1$ and spatial spacing $h>0$  satisfying Assumption \ref{ass_DiscrDefOps} (discrete finite-difference operators),
	\item $\rhobr[h]$ be the corresponding finite difference mean-field limit, as defined in  \eqref{eq_discrIntMeanFiLim}.
\end{itemize}
Let the initial conditions satisfy Assumption \ref{ass_InitCond} and the parameters $\rI,N,h$ satisfy Assumption \ref{ass_Scaling} with $0<\delZero<1$ (scaling regime for $N$ and $h$). 

Let $0<\ep<\delZero/4$.
Then there exists a stopping time $\Ts\in[0,T]$ with
\begin{align}\label{NegBoundRig}
	& \mathbb{P} \big[  \Ts  < T \big] 
	\lesssim  
	\exp\big( - C N^{\ep/2} \big).
\end{align}
such that the discrete Dean-Kawasaki solutions $\{\rho_{h,\alpha}(t\wedge\Ts)\}_{\alpha=1}^{n_S}$ capture the fluctuations of the the empirical measures $\mu_{\alpha,t}^{r_I, N}$ in the following sense.

Let $j\in\N$. 
Assume that with
$
	s(d,p,j)\coloneqq p+\frac{3d}{2}+4+j
$
the continuous mean field limit $\rhobr$ is in $L^{\infty}\big(0,T; C^{s(d,p,j)}\big)$,
and that the interaction potentials satisfy $V\in \big[W^{\sI,1}(\domain)\big]^{\nS\times\nS}$ with 
$
	\sI > 2^{j-1}3j^2 s(d,p,j) +2d +1
$.
Then, abbreviating $data:=\{V,\rhobr,\rho_{min},\rho_{max},d, T, \nS, p, j\}$, for each $\delflex>0$, there exists $N_0=N_0(\delflex,\ep,K,data)$
such that if $N>N_0$,
\begin{align}
&d_{-(2j+1)}\left[
N^{1/2}
\begin{pmatrix}
\sum_{\alpha=1}^\nS\big(\Ih[\vphi_{1,\alpha}],
	(\rhor[h,\alpha]-\rhobr[h,\alpha])(T_1 \wedge \Ts)\big)_h
\\
\vdots
\\
\sum_{\alpha=1}^\nS\big(\Ih[\vphi_{K,\alpha}],
	(\rhor[h,\alpha]-\rhobr[h,\alpha])(T_K \wedge \Ts)\big)_h
\end{pmatrix},
\right.\nonumber\\
& 
\quad \quad \quad \quad \quad \quad \quad \quad\qquad \quad\left.
N^{1/2}
\begin{pmatrix}
\sum_{\alpha=1}^\nS\big\langle\vphi_{1,\alpha},
	(\empmeas[\alpha]-\rhobr[\alpha])(T_1)\big\rangle
\\
\vdots
\\
\sum_{\alpha=1}^\nS\big\langle\vphi_{K,\alpha},
	(\empmeas[\alpha]-\rhobr[\alpha])(T_K)\big\rangle\end{pmatrix}
\right]
\label{DistanceRandomVariables}\\
& \quad 
\leq C(K,\|\bvphi\|_{H^{s(d,p,j)}},data)N^{2\ep+\delflex}\big(h^{p+1} +\exp\big(-CN^{\ep/2}\big) +N^{-j(1/2-2\ep)} \big) \nonumber\\
& \quad =: Err_{num} + Err_{\oslash} + Err_{fluct,rel}\label{FluctPrecise}
\end{align}
holds for any $\f{\varphi} = (\varphi_{k,\alpha})_{\alpha=1,\dots,\nS}^{k=1,\dots,K} \in \big[H^{s(d,p,j)}(\domain,\R^\nS)\big]^K$ and any $(T_1,\dots, T_K) \in [0,T]^K$,
where the distance $d_{-j}[X,Y]$ has been introduced in \eqref{WeakDistance}.
\end{theorem}

\begin{rem}
Throughout the paper, the generic expression \emph{data} -- which may change from use to use -- denotes a relevant subset of parameters $\{V,\rhobr,\rho_{min},\rho_{max},d, T, \nS, p, j\}$.
More specifically, dependencies on $V,\rhobr$ come in terms of Sobolev norms. 
Note that $\rI$ is not included. 
We assume that there is a uniform bound in $\rI$ for the respective Sobolev norms of $\rhobr$. 
This is a reasonable assumption due to the convergence to the solution of a suitable cross-diffusion system given regular enough initial data for $\rI\rightarrow 0$, see e.g. \cite{chen2019rigorous}.
\end{rem}

\begin{rem}
Theorem \ref{MainThm} captures the relative error in fluctuations, $Err_{fluct,err}$: this is due to the prefactor $N^{1/2}$ in the arguments of the metric $d_{-j}$, which balances out the natural order of fluctuations $N^{-1/2}$ of the inner products 
	$\sum_{\alpha=1}^\nS(\Ih[\vphi_{k,\alpha}],
	(\rhor[h,\alpha]-\rhobr[h,\alpha])(T_k \wedge \Ts))_h$
	and 
	$\sum_{\alpha=1}^\nS\langle\vphi_{k,\alpha},
	(\empmeas[\alpha, T_k]-\rhobr[\alpha])(T_k)\rangle
	$. 
The term $Err_{num}$ accounts for the intrinsic numerical error of the scheme. 
The term $Err_{\oslash}$ accounts for the cases where the stopping time runs short of $T$.
The definition of the stopping time (see Section \ref{SubsecExpSmallProb}) ensures positivity of $\rhor[h]$: This is due to fact that $\Ts$ incorporates an $L^\infty$-bound for $\rhor[h]-\rhobr[h]$, and the fact that the mean-field limit is $\rhobr[h]$ is strictly positive by Assumption \ref{ass_RegularityMFL} below.
Additionally, $\Ts$ ensures that the fluctuations $\rhor[h]-\rhobr[h]$ roughly stay within the natural regime $N^{-1/2}$. 
Hence, control of the stopping time is the discrete equivalent of quantifying the mean-field limit convergence.

The additional factors of $N^{\ep}$ appear since we are only able to control the stopping time if we relax the fluctuation bound from $N^{-1/2}$ to $N^{-1/2+\ep}$.
The factor $N^\delflex$ stems from the logarithmic scaling of the interaction radius $\rI$ with respect to $N$, see Assumption \ref{ass_Scaling}.
If $\rI$ is constant and does not scale, $N^\delflex$ can be replaced with a constant depending on \emph{data}.
\end{rem}

\begin{rem}\label{AbbreviationXY}
We will often abbreviate the $\mathbb{R}^K$-valued random variables in \eqref{DistanceRandomVariables} as
\begin{align}\label{RandomVariablesXY}
N^{1/2}\big\langle\bvphi,\empmeas[\bT]-\rhobr(\bT)\big\rangle, \qquad N^{1/2}\big(\Ih[\bvphi],(\rhor[h]-\rhobr[h])(\bT \wedge \Ts)\big)_h,
\end{align}
where $\bT := [T_1,\dots, T_K]\in [0,T]^K$.
\end{rem}

\subsection{Structure of the paper}
The details of the weakly interacting particle systems we consider, as well as the relevant discretised Dean--Kawasaki model, are given in Section \ref{SecSettNotAss}. 
Section \ref{SecInformalProof} gives an informal -- yet exhaustive -- summary of the most important results needed to prove Theorem \ref{MainThm}: 
In particular, Theorem \ref{MainThm} is of inductive type,
and the small fluctuation error $Err_{fluct,rel} \propto N^{2\ep+\delflex-j(1/2 -2\ep)}$ is obtained after $j$ induction steps.
Section \ref{TheKeyStepSection} spells out the structure of one of these induction steps and provides all necessary building blocks (above all, Proposition \ref{prop_Iterative Structure for Errlin} indicates how to quantitatively include the convolutional nonlinearity in the iteration, thus resolving the mismatch with the linear noise covariance). 
The proof of Theorem \ref{MainThm} (i.e., the quantitative performance of all $j$ steps) is finalised in Section \ref{sec_ProofMT}. 
The core technical lemmas are deferred to subsequent sections, namely: 
Quantitative convergence to the mean-field limit (Section \ref{SecQuantConvMeanF}); 
Exponentially decaying bound for the probability of the stopping time $\Ts$ coming short of the final time horizon $T>0$ (Section \ref{SubsecExpSmallProb}).

Finally, the appendix contains the following: Regularity estimates for the continuous test functions (Appendix \ref{AppContReg}); Regularity estimates for the discretised mean-field limit and the discretised test functions, as well as error bounds with respect to their continuous counterparts (Appendix \ref{AppDiscrete}); Explicit construction of a set of admissible initial conditions for the discrete mean-field limit (Appendix \ref{AppConstructionInitialData}).

\section{Setting, Notation, and Assumptions}\label{SecSettNotAss}

Throughout the paper, we use $C$ to denote a generic constant whose value may change from line to line. Relevant dependencies on specific parameters are highlighted whenever needed. Moreover, for generic functions $f,g\colon \domain \rightarrow \mathbb{R}^m$ ($m\in\mathbb{N}$) and $s\colon \domain \rightarrow \mathbb{R}$, we denote 
$
(f \ast s) (x) := \left[\int_{\domain}{f_\ell(y)s(x-y)\m y}\right]_{\ell=1}^{m}$
and
$
(f\ast_c g)(x) := \sum_{\ell=1}^{m}{(f_\ell \ast g_{\ell})(x)}.
$

We now give specific notation and relevant assumptions for the weakly interacting particle system we consider, and its Dean--Kawasaki approximation.

\subsection{The continuous setting - the particle system}
The weakly interacting particle system we are interested in is given by
\begin{equation}\label{eq_CDPartSys}
	\left\{\begin{aligned}
		\m X^{\rI,N}_{\alpha,i}(t) &=
				- \sum_{\beta=1}^\nS N^{-1} \sum_{j=1}^N\nabla\V[\alpha\beta]\left(X^{\rI,N}_{\alpha,i}(t)-X^{\rI,N}_{\beta,j}(t)\right)\m t 
				+ \sqrt{2\sigma_\alpha}\m B_{\alpha,i}(t)\\
		X^{\rI,N}_{\alpha,i}(0)	&= \eta_{\alpha,i},
		\qquad \alpha=1,\ldots,\nS,\quad i=1,\ldots,N.
	\end{aligned}\right.
\end{equation}
In \eqref{eq_CDPartSys}, $\{X^{\rI,N}_{\alpha,i}\}_{i=1}^{N}\subset \domain$ denote positions of particles of species $\alpha$, $\{\sigma_\alpha\}_{\alpha=1}^{\nS} > 0$ are diffusion constants, $B_{\alpha,i}$ are independent Brownian motions (also independent of $\{X^{\rI,N}_{\alpha,i}(0)\}_{\alpha,i}$), $\{\eta_{\alpha,i}\}_{i=1}^N$ are the particles' initial positions, and the potentials $\{\V[\alpha\beta]\}_{\alpha,\beta=1}^{\nS}$ are defined as standard mass-preserving rescaling of smooth potentials $V_{\alpha\beta}$, namely
\begin{align}\label{DefPotentialsV_alpha_beta}
\V[\alpha\beta](\cdot)=\rI^{-d}V_{\alpha\beta}\left(\cdot/\rI\right).
\end{align}
We refer to Assumption \ref{ass_CDsystem+MFL} below for the regularity of the potentials $\{\V[\alpha\beta]\}_{\alpha,\beta=1}^{\nS}$ and the law of the particles' initial positions $\{\eta_{\alpha,i}\}_{i=1}^N$.

For each species $\alpha$ we define the empirical measure 
\begin{equation}\label{eq_empirical measure}
	\empmeas[\alpha,t] := N^{-1}\sum_{i=1}^N\delta_{X^{\rI,N}_{\alpha,i}(t)}.
\end{equation}
In the limit $N\rightarrow \infty$, the empirical densities \eqref{eq_empirical measure} converge almost surely to the deterministic limit $\rhobr[\alpha]$ satisfying the PDE 
\begin{equation}\label{eq_CD-IMFL}
	\left\{\begin{aligned}
	 	\partial_t \rhobr[\alpha] &= 
	 			\sigma_\alpha\Delta\rhobr[\alpha]
	 			+\divv\left(\sum_{\beta=1}^\nS\rhobr[\alpha]\nabla\V[\alpha\beta]\ast\rhobr[\beta]\right)\\
		\rhobr[\alpha](0)&=\rhobz[\alpha],\\
	\end{aligned}\right.
\end{equation}
where $\rhobz[\alpha]$ is a suitable deterministic approximation of the initial particle distribution, see Assumption \ref{ass_CDsystem+MFL}. 

Finally, the fluctuating hydrodynamics Dean--Kawasaki equation capturing the fluctuations of the particle system \eqref{eq_CDPartSys} on top of the mean-field limit \eqref{eq_CD-IMFL} is precisely \eqref{dk}, the SPDE of interest for this work.

\begin{rem}
As already mentioned, we view the current work in the weakly interacting particle setting (which translates in the mollified potentials $\V[\alpha\beta]$ via the parameter $r_I>0$) as laying the ground for future applications to the purely local cross-diffusion case (i.e., considering $r_I \rightarrow 0$).
\end{rem} 

\subsection{Discretisation of the Dean--Kawasaki model}\label{FiniteDiffDiscr}

We work with the uniformly spaced grid on the $d$-dimensional torus $\mathbb{T}^d := [-\pi,\pi)^d$. Specifically, for $L\in 2\mathbb{N}$, we define the spatial discretisation parameter $h:=2\pi/L$, and set $\Ghd := h\mathbb{Z}^d \cap \mathbb{T}^d = \{-\pi, -\pi + h, \dots, \pi - h\}^d$. 
For $m\in\mathbb{N}$, we endow the space $[L^2(\Ghd)]^m$ with the standard inner product 
\begin{align*}
	(u_h,v_h)_h := \sum_{x\in \Ghd}{h^du_h(x) \cdot v_h(x)},
\end{align*}
and the orthonormal basis $\hbase[m]{x,\ell}(y) := h^{-d/2}\delta_{x,y}e_\ell$ for $(x,\ell)\in (\Ghd,\{1,\dots,m\})$, where $e_\ell$ is the $\ell$-th vector of the $\mathbb{R}^d$ canonical basis.
If there is no ambiguity, the notation is simplified as $\hbase[m]{x,\ell} \equiv \hbase{x}$. 
The natural discrete analogue of the continuous convolution operator $\ast$ (respectively, $\ast_c$, see beginning of Section \ref{SecSettNotAss}) with respect to the $L^2(\Ghd)$-inner product is denoted by $\ast_h$ (respectively, by $\ast_{c,h}$).
Furthermore, we denote by $\mathcal{I}_h$ the interpolation operator of continuous functions onto $[L^2(\Ghd)]^m$, meaning that $\mathcal{I}_hf(x) = f(x)$ for every $x\in \Ghd$.

As for the discrete differential operators, we use: i) a discrete gradient $\nabla_h$ and divergence $\nabla_h \cdot $ based on suitable first-order discrete partial derivatives $[\partial_{h,x_1},\dots,\partial_{h,x_d}]$, and ii) second-order discrete derivates $D^2_{h,x_\ell}$ satisfying a standard integration by parts rule
\begin{align}\label{IntByParts}
(D^2_{h,x_\ell} u_h, v_h)_h = -(D_{h,x_\ell}u_h, D_{h,x_\ell}v_h)_h
\end{align}
for some other first-order operators $D_{h,x_\ell}$. 
Furthermore, we denote $\Delta_h := \sum_{\ell=1}^{d}{D^2_{h,x_\ell}}$ and $\nabla_{h,D} := [D_{h,x_1},\dots,D_{h,x_d}]$. 
In general, $D_{h,x_\ell}$ may differ from $\partial_{h,x_\ell}$.

We can now define the discretized Dean--Kawasaki model.

\begin{definition}[Finite difference Dean--Kawasaki model of order $p+1$]\label{fddk}
We say that the $\LzGhd$-valued processes $(\rhor[h,\alpha])_{\alpha=1,\ldots\nS}$ solve the finite difference Dean--Kawasaki model if they solve the system of stochastic differential equations
\begin{equation}\label{eq_discretizedDK}
\tag{h-DK}
	\left\{\begin{aligned}
		\mbox{d}(\rhor[h,\alpha],\hbase{x})_h =&\,
				 \bigg[\sigma_\alpha(\Delta_h\rhor[h,\alpha],\hbase{x})_h 
				 -\sum_{\beta=1}^\nS\big(\rhor[h,\alpha](\Ih[\nabla\V[\alpha\beta]]\ast_h\rhor[h,\beta]),\nabla_h \hbase{x}\big)_h\bigg] \m t\\
				&\,\,-\sqrt{2\sigma_\alpha} N^{-1/2}\!\!\!\!\sum_{\substack{y\in\Ghd \\ l\in\lbrace 1,\ldots,d\rbrace}}\!\!\!
					\left(\sqrt{(\rhor[h,\alpha])^+}\hbase[d]{y,l},\nabla_h \hbase{x}\right)_h \m  W_{(y,l)}^\alpha,
				\quad\forall x,
				\\
		\rhor[h,\alpha](0) =&\, \rho_{h,\alpha}^0
	\end{aligned}\right.
\end{equation}
on a finite-time horizon $T>0$, where $(f_x)_{x\in\Ghd}$ is the basis from above, and where $$\lbrace  W_{(y,l)}^\alpha \rbrace_{(y,l)\in\Ghd\times\lbrace 1,\ldots,d\rbrace}^{\alpha=1,\ldots,\nS}$$ are independent Brownian motions.
The assumptions on the random initial datum $\{\rho_{h,\alpha}^0\}_{\alpha=1,\dots,\nS}$ will be given in Assumption \ref{ass_InitCond} below. 
Moreover, the families $\lbrace  W_{(y,l)}^\alpha \rbrace_{(y,l)\in\Ghd\times\lbrace 1,\ldots,d\rbrace}^{\alpha=1,\ldots,\nS}$ and $\{\rho_{h,\alpha}^0\}_{\alpha=1,\dots,\nS}$ are independent.
\end{definition}

Analogously to the continuous case, the model \eqref{eq_discretizedDK} captures the fluctuations around the following discretised mean-field limit.
\begin{definition}[Mean-field limit for \eqref{eq_discretizedDK}]\label{fddkmf}
We say that the $\LzGhd$-valued functions $(\rhobr[h,\alpha])_{\alpha=1,\ldots\nS}$ solve the discrete version of the mean-field limit of \eqref{eq_CDPartSys} if they solve the system of differential equations\begin{equation}
	\tag{h-MFL}
	\left\{\begin{aligned}\label{eq_discrIntMeanFiLim}
		\partial_t\rhobr[h,\alpha] &=
				\sigma_\alpha\Delta_h\rhobr[h,\alpha]
				+ \nabla_h\cdot\bigg(\rhobr[h,\alpha]\sum_{\beta=1}^\nS\Ih[\nabla\V[\alpha\beta]]\ast_h\rhobr[h,\beta]\bigg)
				\,\,\,\,\,\,\text{on }\Ghd\times(0,T),\\
		\rhobr[h,\alpha](0)	&= \rhobz[h,\alpha],
				\qquad \alpha=1,\ldots,\nS.
	\end{aligned}\right.
\end{equation}
for given deterministic initial datum $\rhobz[h,\alpha]$, see Assumption \ref{ass_InitCond} and Remark \ref{rem_AlernativeInitCond}.
\end{definition}

\subsection{Relevant functions spaces}\label{subsec_spaces}
\subsubsection{The spaces $\mathcal{L}_{pow,r}^{q}$ of functions with polynomially growing derivatives}
While the definition of the metric $d_{-j}$ takes the supremum over `generalised moment functions' $\psi\in W^{j,\infty}\cap C^j(\R^K,\R)$, the natural spaces for such functions will turn out to be
\begin{align*}
	\mathcal{L}_{pow,r}^{q}(\R^K) 
	\coloneqq 
	\left\lbrace \psi\in C^q(\R^K) ;~ 
	\|\psi\|_{\mathcal{L}_{pow,r}^{q}} \coloneqq \max_{0\leq\tilde{q}\leq q} \|(1+|\cdot|^2)^{-r/2} D^{\tilde{q}}\psi\|_{L^\infty}
	<\infty \right\rbrace
\end{align*}
for $K\in\N$, $q,r\in\Nz$. In particular, it holds $W^{j,\infty}\cap C^j= \mathcal{L}_{pow,0}^{j}$.

\subsubsection{The discrete Sobolev norms}
We denote the $L^2$-norm induced by the discrete inner product $(\cdot,\cdot)_h$ on $\LzGhd$ by either $\|\cdot\|_{\LzGhd}$ or, more succinctly, $\|\cdot\|_{L^2_h}$.
Analogously to the continuous setting, we also use the notation
$\|g\|_{L^\infty(\Ghd)}=\|g\|_{L^\infty_h} \coloneqq \max_{x\in\Ghd}|g(x)|$.

We will also need a discrete version of (also negative) Sobolev norms: This version uses the first order one-sided finite differences given by
\begin{equation*}
	\dOne{e_\ell}g(x) = \frac{g(x+he_\ell)-g(x)}{h} \qquad\text{for all } g\in\LzGhd, \qquad \ell=1,\ldots,d.
\end{equation*}

\begin{definition}[The discrete Sobolev norms]\label{def_Hsh}
For $s\in\Nz$ and $g\in\LzGhd$ set
\begin{equation}\label{eq_MFLhO_altDef_Hs}
	\|g\|_{\HGhd[s]}=\|g\|_{H^s_h} \coloneqq \sup_{|\nu|\leq s}\|\dOne{\nu} g\|_{\LzGhd}
\end{equation}
with the supremum over multi-indices $\nu\in\Nz^d$.
\end{definition}

\begin{rem}\label{rem_DefHsh}
The discrete $H^s$-norms can be equivalently characterized via the discrete Fourier basis:
Let  $(\vtheta_m)_{m\in\Z^d\cap\left[-\frac{\pi}{h},\frac{\pi}{h}\right)^d}\subset\LzGhd$ with $\vtheta_m(x)\coloneqq (2\pi)^{-d/2}e^{im\cdot x}$. 
Then
\begin{equation*}
	\|g\|_{\HGhd[s]}^2
	\approx 
	\sum_{m\in\Z^d\cap\left[-\frac{\pi}{h},\frac{\pi}{h}\right)^d}(1+|m|_2^2)^s\big(g,\vtheta_m\big)_h^2.
\end{equation*}
This stems from the observation that due to the discrete version of the Plancherel theorem
\begin{equation*}
	\|\dOne{e_\ell}g\|_{\LzGhd}
	=
	\sum_{m\in\Z^d\cap\left[-\frac{\pi}{h},\frac{\pi}{h}\right)^d}\frac{1}{h^2}|e^{im_\ell h}-1|^2\big(g,\vtheta_m\big)_h^2.
\end{equation*}
\end{rem}

This remark leads to a natural definition of the discrete negative Sobolev norms.
\begin{definition}[Discrete negative Sobolev norms]\label{def_HshNeg}  
For $\Z \ni s < 0$ and $g\in\Ghd$, we set
\begin{equation*}
	\|g\|_{\HGhd[s]}^2
	\coloneqq
	\sum_{m\in\Z^d\cap\left[-\frac{\pi}{h},\frac{\pi}{h}\right)^d}(1+|m|_2^2)^s\big(g,\vtheta_m\big)_h^2
\end{equation*}
where $(\vtheta_m)_{m\in\Z^d\cap\left[-\frac{\pi}{h},\frac{\pi}{h}\right)^d}\subset\LzGhd$ with $\vtheta_m(x)\coloneqq e^{im\cdot x}$ is the Fourier basis of $\LzGhd$.
\end{definition}

\subsection{Assumptions}\label{SecAssumptions}

\begin{customthm}{A1}[Weakly interacting particle system and associated mean-field limit]\label{ass_CDsystem+MFL}
We consider the weakly interacting particle system $\{X^{\rI,N}_{\alpha,i}\}_{\alpha=1,\dots,\nS}^{i=1,\dots,N}$ as given in \eqref{eq_CDPartSys}, and its associated mean-field limit \eqref{eq_CD-IMFL}. 
In terms of regularity, we assume that
\begin{align*}
	V_{\alpha\beta} & \in C^{p+3}(\domain), \qquad \forall \alpha,\beta=1,\ldots,\nS,
\end{align*}
and, furthermore, that $V_{\alpha\beta}$ is symmetric (i.e., $V_{\alpha\beta}(x) = V_{\alpha\beta}(-x)$). This regularity is passed on to the rescaled potentials $\V[\alpha\beta]$. 

For the initial values of the particles, we assume that either they are i.i.d.\ according to some probability distribution or that they satisfy a spectral gap inequality in the sense that for any $F\in C^1((\R^d)^N)$ and all $\alpha=1,\ldots,\nS$ there holds
\begin{equation}\label{eq_ass_sepctralGap}
	\Ev\left[\left|F\left(\big(X^{\rI,N}_{\alpha,i}(0)\big)_i\right)-\mathbb{E}\left[F\left(\big(X^{\rI,N}_{\alpha,i}(0)\big)_i\right)\right]\right|^2\right]
	\leq C \Ev\left[\left|\nabla F\left(\big(X^{\rI,N}_{\alpha,i}(0)\big)_i\right)\right|^2\right].
\end{equation} 
For the initial value $\rhobr(0)=\rhobz$ of the mean field limit we assume that 
\begin{equation}\label{eq_ass_rhobz vs Ev}
	\big\|\Ev[\empmeas[0]]-\rhobz\big\|_{H^{-d/2-2}}\leq CN^{-1/2}.
\end{equation}
\end{customthm}

\begin{customthm}{A2}[Continuous mean field limit]\label{ass_RegularityMFL}
The solution $\{\rhobr[\alpha]\}_{\alpha=1}^{\nS}$ to \eqref{eq_CD-IMFL} exists and belongs to $L^\infty(0,T;C^{p+3}(\domain,\R^\nS))$.
Furthermore, we assume that there exist $\rho_{min},\rho_{max}\in\R$ such that on $[0,T]$ it holds that $0<\rho_{min}\leq\rhobr[\alpha]\leq\rho_{max}$ for all $\alpha=1,\ldots,\nS$.
\end{customthm}

\begin{customthm}{A3}[Discrete differential operators]\label{ass_DiscrDefOps}
Let $p\in\N$ be fixed.
The discrete operators $\partial_{h,x_\ell}, D_{h,x_\ell},D^2_{h,x_\ell}$ introduced in Subsection \ref{FiniteDiffDiscr} are standard finite difference operators of order $p+1$. 
In particular, being of finite difference type, these operators commute.
Finally, the two first order differential operators $\partial_{h,x_\ell}, D_{h,x_\ell}$ satisfy the inequality
\begin{align}\label{1stDerBound}
\|\dhRx[\ell]g\|_{L^2(\Ghd)}^2 \leq \frac{1}{C_D} \|D_{h,x_\ell}g\|_{L^2(\Ghd)}^2,\qquad \forall g\in L^2(\Ghd),
\end{align}
for some $C_D>0$, where $\dhRx[\ell]$ is the reflected version of $\partial_{h,x_\ell}$ (i.e., its adjoint) which appears when integrating by parts in the discrete setting.
\end{customthm}


\begin{customthm}{A4}[Discrete initial conditions]\label{ass_InitCond}
The initial conditions $\{\rho_{h,\alpha}^0\}_\alpha$ for \eqref{eq_discretizedDK} and $\{\rhobz[h,\alpha]\}_{\alpha}$ for \eqref{eq_discrIntMeanFiLim} are chosen such that the following properties hold.
\begin{itemize} 
\item 	\emph{Initialization via interpolation:} 
		We set $\rhobz[h]\coloneqq \Ih[\rhobr(0)]=\Ih[\rhobz]$.
\item	\emph{Positivity and mass restriction:}
		We assume that the random discrete initial data satisfies $\rhoz[h,\alpha]\geq 0$ and $\|\rhoz[h,\alpha]\|_{L^1_h}=h^d\sum_{x\in\Ghd}|\rhoz[h,\alpha](x)|\leq 2$ for all $\alpha=1,\ldots,\nS$.
\item 	\emph{High-order fluctuation bound:} 
		For any $K,q\in\N$, $r\in\Nz$ and $\psi\in\mathcal{L}_{pow,r}^q(\R^K)$, the bound
		\begin{multline}\label{InitDataFluctInequality}
			\left| \Ev\bigg[\psi\bigg(N^{1/2}\big\langle\bvphi,\empmeas[0]-\rhobr(0)\big\rangle\bigg)\bigg] \right.\\
			\left. - \Ev\bigg[\psi\bigg(N^{1/2}(\Ih[\bvphi],(\rhor[h]-\rhobr[h]))_h(0)\bigg)\bigg]\right| 
			\leq 
			C\|\psi\|_{\mathcal{L}_{pow,r}^q}h^{p+1}\|\bvphi\|_{C^{p+1}}^{r+1}
		\end{multline}
		holds for any $\bvphi \in \big[C^{p+1}(\domain,\R^\nS)\big]^K$. 
		In \eqref{InitDataFluctInequality}, we have used the vectorial notation convention specified in Remark \ref{AbbreviationXY}.
\item	\emph{Exponentially decaying bound for probability of observing `large fluctuations':}
		We assume that for any $1\leq R \leq N^{1/2}h^{d/2}$ it holds that 
		\begin{align}
			\Pm\big[\|\rhoz[h]-\rhobz[h]\|_{L^\infty_h}\geq N^{-1/2}h^{-d/2}R\big]
			&\leq
			C\exp(-CR),\label{eq_ass_InitialLinftyh}\\
			\Pm\big[\|\rhoz[h]-\rhobz[h]\|_{H^{-\floor{d/2+1}}_h}\geq N^{-1/2}R\big]
			&\leq
			C\exp(-CR).\label{eq_ass_InitialHlh}
		\end{align}		
\end{itemize}
\end{customthm}
		
\begin{customthm}{A5}[Parameter scaling]\label{ass_Scaling}
The parameters $(N,h,\rI)$ scale with respect to the following relations:
For a chosen, arbitrarily small $0<\delZero<1$, we assume that
\begin{align}
	N^{1-\delZero}h^d 	&\geq 	1,	\label{ScalingRegime}\\
	N^{\delZero(T+1)}h	&\leq	1, \label{ScalingRegimeNhReverse}\\
	r_I^{-2(d+2)} 		&\leq 	\log N.\label{ScalingRegimeRadius}
\end{align}
The scaling regime \eqref{ScalingRegimeRadius} is analogous to that of \cite[Theorem 3]{CDHJ}.
\end{customthm}

\begin{rem}[Ad Assumptions \ref{ass_CDsystem+MFL}--\ref{ass_RegularityMFL}]\label{RemA1A2}
The regularity of $\{\rhobr[\alpha]\}_{\alpha=1}^{\nS}$ prescribed in Assumption \ref{ass_RegularityMFL} can be met on any time interval $[0,T]$, for instance, if the initial value for the mean field limit $\rhobz[\alpha]$ belongs to $H^{s}$, $s > \max\{d/2+1; p+3 + (d/2+1)\}$ and satisfies a smallness assumption $\|\rhobz[\alpha]\|_{H^s} \leq \{\min_{\alpha}{\sigma_{\alpha}}\}/[C(s,d)\sum_{\alpha,\beta}{\|V_{\alpha\beta}\|_{L^1}}]$ \cite{chen2019rigorous}. See also \cite{CDHJ} for an alternative setting.
\end{rem}


\begin{rem}[Ad Assumption \ref{ass_InitCond}]\label{rem_AlernativeInitCond}
Since $\rhobr[h](0)=\Ih[\rhobr(0)]$, for order $p+1$ finite difference operators $(\partial_h^{e_\ell})_{\ell=1,\ldots,d}$ we have
\begin{align}\label{InitDataInterpInequality}
	\sup_{|\nu|\leq s}\|\Ih[\partial^\nu\rhobr(0)] - \partial_h^\nu\rhobr[h](0)\|_{L^2_h} 
	\leq 
	C\|\rhobr(0)\|_{C^{s+p+2}}h^{p+1},
\end{align}
where the supremum is over multiindices $\nu\in\Nz^d$ provided $\rhobr[h](0)$ is regular enough.
\end{rem}

\begin{rem}[Positivity and Boundedness of the discrete mean field limit]\label{rem_rhominh_rhomaxh}
Analogously to the continuous Sobolev inequality, in the discrete setting there exists $C_S>0$ such that $\|g\|_{L^\infty_h}\leq C_S \|g\|_{H^{s}}$ for all $s>{d/2}$, $g\in\LzGhd$.
In particular, in light of Proposition \ref{prop_MFLhigherOrder}, as long as $\rhobr\in L^\infty(0,T;C^{p+3+\floor{d/2}+1})$, there exist $\rho_{min,h},\rho_{max,h}$ such that for all $N>N_0$ and $h<h_0$ it holds that on $[0,T]$
\begin{equation*}
	0<\rho_{min,h}\leq\rhobr[h,\alpha]\leq\rho_{max,h}
	\quad\text{for all } \alpha=1,\ldots,\nS,
\end{equation*}
for some $N_0=N_0(p, data)$, where $h_0=h_0(\|\rhobr\|_{C^{\floor{d/2}+2}},\rho_{min})$.
\end{rem}

\begin{rem}[Constructing the discrete initial data]\label{rem_DiscrInitData}
The assumptions for the initial random distribution $\rhoz[h]$ (Assumption \ref{ass_InitCond}) may seem extensive at first glance, but such discrete initial data can be rather naturally constructed from the continuous initial data. 
The most straightforward example is a pathwise approach based on an interpolation scheme -- each realisation of $\rhoz[h]$ is derived from the initial particle realisation $\empmeas[0]$.
Essentially, the mass of a particle starting at $X^{\rI,N}_{\alpha,i}(0)$ is split across surrounding grid points according to interpolation weights.
These weights stem from an interpolation scheme of sufficiently high order that is used to approximate functions at $X^{\rI,N}_{\alpha,i}(0)$ based on the function values at the grid points.
\end{rem}

\section{Strategy of the proof of Theorem \ref{MainThm}: an informal view}\label{SecInformalProof}

In this section, we spell out the main ideas behind the proof of Theorem \ref{MainThm} in an informal way. All arguments will be made precise and rigorous later on.

\subsection{The induction step}\label{SecInductiveStep}
The proof of Theorem \ref{MainThm} is of inductive type, as we now detail.

In order the compare the fluctuations of the discrete Dean--Kawasaki solution $\rho_h$ (see \eqref{eq_discretizedDK}) and of the particle empirical density $\empmeas$ (see \eqref{eq_empirical measure}), we choose a set of times $\f{T}= (T_k)_{k=1}^{\numtest}$, regular enough test functions $\f{\varphi} = (\vphi_{k,\alpha})_{k=1,\ldots,\numtest}^{\alpha=1,\ldots,\nS}$ and set
\begin{align}\label{RandomVariablesXY}
\bzeta^T 	& := N^{1/2}\big\langle\bvphi,\empmeas[\bT]-\rhobr(\bT)\big\rangle \in \mathbb{R}^K,\nonumber\\
\bzeta^T_h 	& := N^{1/2}\big(\Ih[\bvphi],(\rhor[h]-\rhobr[h])(\bT \wedge \Ts)\big)_h \in \mathbb{R}^K,
\end{align}
which is a shorthand vectorial notation (over the indexes $k=1,\dots,\numtest$ and $\alpha=1,\dots,\nS$) for the random variables in \eqref{DistanceRandomVariables}, as anticipated in Remark \ref{AbbreviationXY}. 
The role and definition of the stopping time $\Ts$ will be discussed in due course.

We measure the distance between $\bzeta^T$ and $\bzeta^T_h$ using the Wasserstein-type metric $d_{-j}$-metric \eqref{WeakDistance}, namely
\begin{align*}
d_{-j}[\bzeta^T,\bzeta^T_h] := \sup_{\psi\colon \max_{1\leq \tilde{j}\leq j}{\|D^{\tilde{j}}\psi\|_{L^{\infty}}} \leq 1}{\left|\mean{\psi(\bzeta^T)} - \mean{\psi(\bzeta^T_h)}\right|}.
\end{align*}
Furthermore, we define the shorthand notation\footnote{we use a slight abuse of notation, as $\psi(\bzeta^T)$ and $\psi(\bzeta^T_h)$ live in different probability spaces.}  
\begin{align}\label{InductiveQuantity}
\mathcal{M}(\psi,\bvphi) := \mean{\psi(\bzeta^T)}-\mean{\psi(\bzeta^T_h)}.
\end{align}
The proof of Theorem \ref{MainThm} crucially revolves around essentially obtaining the following relation via the It\^o calculus
\begin{align}\label{IterationBound}
\mean{d\mathcal{M}(\psi,\bvphi)} & \propto N^{-1/2+2\ep}\cdot\mean{\mathcal{M}(\tilde{\psi},\tilde{\bvphi})}dt + Err_{\oslash} + Err_{num}, 
\end{align}
where the parameter $\ep>0$ is small (so that $2\ep<1/2$).
The bound \eqref{IterationBound} has an iterative component, as the right-hand-side contains -- yet -- another object of kind $\mathcal{M}$, as well as a numerical error $Err_{num}$, which will show to be of the type
\begin{align}
 Err_{num} & \propto N^{2 \ep + \delflex}h^{p+1},  \label{ResidualNum}
\end{align}
and a modelling error 
\begin{align}
 Err_{\oslash} & \propto N^{2\ep+\delflex}\exp\big(-CN^{\ep/2}\big),  \label{Stop}
\end{align}
where, again, the parameter $\delflex>0$ is small enough.
The role of the small parameters $\ep,\delflex$ will be clarified throughout the proofs, and we need not discuss it in this summary.

By iterating \eqref{IterationBound} over $\mathcal{M}$, one cumulates as many (small) prefactors $N^{-1/2 +2\ep}$ as steps performed, together with numerical and modelling errors \eqref{ResidualNum}--\eqref{Stop}. The number of iterative steps one performs is only capped by the regularity of the initial test functions $\psi,\bvphi$ (such regularity usually deteriorates from step to step). Once the regularity of the test functions is exhausted and \eqref{IterationBound} is inapplicable, one closes off the argument by obtaining the -- final -- fluctuation contribution in \eqref{FluctPrecise}. 
The precise details concerning \eqref{IterationBound}--\eqref{Stop} are given in Theorem \ref{thm_iterative structure}, which, in turn, relies on several other ingredients which we now list.

\subsubsection{Choice of dynamics for test functions $\f{\varphi}$ and $\mathcal{I}_h\f{\varphi}$}\label{subsubsec_OVBackEvoTest} The particle system \eqref{eq_CDPartSys} satisfies a crucial property: as the particles are driven by independent Brownian motions, it is easy to see that the cross variation (denoted by square brackets) of the quantities $\langle \varphi_1, \mu^{r_I,N} -\rhobr \rangle$ and $\langle \varphi_2, \mu^{r_I,N} -\rhobr \rangle$ (these are the `building blocks' for $\bzeta^T$ in \eqref{RandomVariablesXY}) is, for sufficiently regular $\varphi_1,\varphi_2$, given by
\begin{align}\label{ParticleCovStructure}
& \left[ \langle \varphi_1, \mu^{r_I,N} -\rhobr \rangle, \langle \varphi_2, \mu^{r_I,N} -\rhobr \rangle\right] \nonumber \\
& \quad = N^{-1} \langle \nabla\varphi_1\cdot\nabla\varphi_2, \mu^{r_I,N} \rangle \nonumber \\
& \quad = N^{-1} \langle \nabla\varphi_1\cdot\nabla\varphi_2,  \mu^{r_I,N} -\rhobr \rangle + N^{-1} \langle \nabla\varphi_1\cdot\nabla\varphi_2, \rhobr \rangle =: P_1 + P_2.
\end{align}
With the exception of term $P_2$ (which we will deal with at a later stage), \eqref{ParticleCovStructure} amounts to saying that the cross-variation preserves linear functionals of $\mu^{r_I,N} -\rhobr$: this fact is crucial, as it plays directly into the iterative structure of $\mathcal{M}$ in \eqref{InductiveQuantity}.

In light of \eqref{ParticleCovStructure}, it is convenient to define time-dependent test functions $\bphi$ which: 
\begin{itemize}
\item coincide with the original test functions $\bvphi$ at the evaluation times $\bT$ (i.e., $\f{\phi}^{\bT} = \f{\varphi}$), and
\item reduce the deterministic drift of the It\^o differential for $\bzeta^T$ in \eqref{RandomVariablesXY} `as much as possible', as such deterministic drift cannot easily be treated using \eqref{ParticleCovStructure}. 
\end{itemize}
The latter requirement leads to the derivation of the backward evolution for $\f{\phi}$ in Lemma \ref{lem_backEvoTest}. Noticeably, the only deterministic drift term which survives this cancellation effort is 
\begin{equation}\label{quad_term}
	\tilQ[t](\phit)
	\coloneqq
	\sum_{\alpha,\beta=1}^\nS
	\big\langle\nabla\phit[\alpha] \cdot \big(\nabla\V[\alpha\beta]\ast\big(\empmeas[\beta,t]-\rhobr[\beta](t)\big)\big),
	\empmeas[\alpha,t]-\rhobr[\alpha](t)\big\rangle.
\end{equation}
The structure of $\tilQ$, which appears as a compensation term after linearising the convolution nonlinearity, is -- at least not yet -- compatible with the structure of $\mathcal{M}$, and will be dealt with in Subsection \ref{subsubsec_OVErrlin} below.

A totally analogous discussion also applies for suitable discretisations of \eqref{dk}: Being statistically equivalent to the particle system \eqref{eq_CDPartSys}, the Dean--Kawasaki model \eqref{dk} enjoys the `mesoscopic analogue' of \eqref{ParticleCovStructure}.
Namely, testing \eqref{dk} with smooth enough $\varphi_1,\varphi_2$ and integrating by parts, one finds the noise cross-variation to be \eqref{CovStructure}, which we recall here:
\begin{align*}
\left[ N^{-1/2}\int_{\domain}{\sqrt{\rho_\alpha}\xi\cdot \nabla\varphi_1}, N^{-1/2}\int_{\domain}{\sqrt{\rho_\alpha}\xi\cdot\nabla\varphi_2} \right] = N^{-1}\int_{\domain}{\rho_\alpha\nabla\varphi_1 \cdot \nabla\varphi_2}.
\end{align*}
Our chosen discretisation for the Dean--Kawasaki model \eqref{eq_discretizedDK} preserves \eqref{CovStructure} on the discrete level: 
Namely, 
for any $\alpha=1,\cdots,\nS$ and $\varphi_{1,h},\varphi_{2,h}\in L^2(\Ghd)$, we have
\begin{align}\label{DiscreteCovStructure}
& \left[ \big(\varphi_{1,h},\rhor[h,\alpha]-\rhobr[h,\alpha]\big)_h , \big(\varphi_{2,h},\rhor[h,\alpha]-\rhobr[h,\alpha]\big)_h\right] \nonumber\\
& = \left[ N^{-1/2}\!\!\!\!\sum_{y\in\Ghd, \,\,l\in\lbrace 1,\ldots,d\rbrace}
					\left(\sqrt{(\rhor[h,\alpha])^+}\hbase[d]{y,l},\nabla_h \varphi_{1,h}\right)_h \m  W_{(y,l)}^\alpha, \right. \nonumber\\
& \quad \quad \quad \left. N^{-1/2}\!\!\!\!\sum_{y\in\Ghd, \,\,l\in\lbrace 1,\ldots,d\rbrace}
					\left(\sqrt{(\rhor[h,\alpha])^+}\hbase[d]{y,l},\nabla_h \varphi_{2,h}\right)_h \m  W_{(y,l)}^\alpha \right] \nonumber\\
& \quad = N^{-1}(\nabla_h \varphi_{1,h}\cdot \nabla_h \varphi_{2,h}, \rhor[h,\alpha]-\rhobr[h,\alpha])_h + N^{-1}(\nabla_h \varphi_{1,h}\cdot \nabla_h \varphi_{2,h}, \rhobr[h,\alpha])_h \nonumber\\
& \quad \quad + N^{-1}(\nabla_h \varphi_{1,h}\cdot \nabla_h \varphi_{2,h}, (\rhor[h,\alpha])^{-})_h =: M_1 + M_2 + M_3. 
\end{align}
With the exception of the terms $M_2$ and $M_3$ (which will be treated in Subsections \ref{subsubsec_OVBackEvoMom} and \ref{subsubsec_OVMomStruc}), \eqref{DiscreteCovStructure} shows that the cross-variation preserves linear functionals of $\rhor[h,\alpha]-\rhobr[h,\alpha]$. In the same way as in the continuous setting, we define a suitable backwards evolution $\f{\phi}_h$ for $\f{\phi}_h^{\bT} = \mathcal{I}_h\f{\varphi}$, see Lemma \ref{lem_discrBackEvoTest}. Taking all species and test functions into account in the definition of $\bzeta^T_h$ in \eqref{RandomVariablesXY}, the only deterministic drift term which survives is given by
\begin{equation}\label{quad_term_discrete}
	\tilQ[h,t](\phith)
	\!\coloneqq\!
	\sum_{\alpha,\beta=1}^\nS
	\!\!\big(\nabla_h \phith[\alpha]\cdot (\Ih[\nabla\V[\alpha\beta]]\ast_h(\rhor[h,\beta]-\rhobr[h,\beta])(t)) ,(\rhor[h,\alpha]-\rhobr[h,\alpha])(t)\big)_h,
\end{equation}
This term, which is the discrete analogous of $\tilQ[t](\phit)$ above, and which also does not yet fit the structure of $\mathcal{M}$, is treated in Subsection \ref{subsubsec_OVErrlin} below.

\subsubsection{The generalised moment structure $\psi$}\label{subsubsec_OVBackEvoMom} The It\^o analysis conducted in Subsection \ref{subsubsec_OVBackEvoTest} shows that -- aside from the mean-field contributions $P_2,M_2$ and the mismatch $\rho_{h,\alpha}^+ \leftrightarrow \rho_{h,\alpha}$ in $M_3$ -- the iterative structure in $\mathcal{M}$ is preserved at the level of noise cross-variations. In order to include the mean-field contributions $P_2,M_2$, we also need to define a suitable backwards equation for the generalised moment function $\psi$: This analysis (see Lemma \ref{lem_pre34a equivalent}) is somewhat complementary to the derivation of the dynamics for the test functions $\bphi$ discussed earlier. The discussion is tailored to the continuous case associated with the particle system, within the adjustments needed for the discrete case treated in Lemma \ref{lem_pre34b equivalent}.

\subsubsection{Comparing the generalised moments, the iterative structure}\label{subsubsec_OVMomStruc} In this step we take the difference of the contributions of the It\^o differential on microscopic and SPDE level (discussed in the Subsections \ref{subsubsec_OVBackEvoTest} and \ref{subsubsec_OVBackEvoMom}), and we derive the bound
\begin{align}\label{IterationBoundPreQuadratic}
\mean{d\mathcal{M}(\psi,\bvphi)} & \propto N^{-1/2 + 2\ep}\cdot\mean{\mathcal{M}(\tilde{\psi},\tilde{\bvphi})}dt + Err_{num} +  Err_{\oslash} \nonumber\\
& \quad + Err_{lin,a} - Err_{lin,b}.
\end{align}
The terms $Err_{lin,a}, Err_{lin,b}$, given by
\begin{align}\label{ErrLin}
Err_{lin,a} & \coloneqq - N^{1/2}\sum_{k=1}^K\int_0^{T_k}\Ev\big[\partial_k\psit(N^{1/2}\big\langle\bphit,\empmeas[t\wedge\bT]-\rhobr(t\wedge\bT)\big\rangle)\tilQ[t](\phit[k])\big]\m t. \\
Err_{lin,b} & \coloneqq -N^{1/2}\sum_{k=1}^K \int_0^{T_k}\Ev\big[\partial_k\psi^t\big(N^{1/2}\big(\bphith,(\rhor[h]-\rhobr[h])(t\wedge\bT\wedge\Ts)\big)_h\big)\\
	& \qquad \qquad \qquad \qquad \qquad \qquad \qquad \qquad \qquad  \qquad 
			\times\tilQ[h,t\wedge\Ts](\phi_{h,k}(t))\big]\m t\nonumber
\end{align}
compensate for linearising the convolutional nonlinearity: in this form, they are not yet suitable for the purposes of iterating. 
Crucially, in \eqref{IterationBoundPreQuadratic}, the mismatch $\rho_h \leftrightarrow \rho^+_{h}$ (cfr. $M_3$ in \eqref{DiscreteCovStructure}) has been resolved using the stopping time $\Ts$: the definition of such a stopping time and Remark \ref{rem_rhominh_rhomaxh} entail the non-negativity of $\rho_h(t)$ for all $t\leq \Ts$. A suitable estimate concerning the smallness of $P[\Ts < T]$ (i.e., \eqref{NegBoundRig}) is proved separately in Proposition \ref{prop_StopTBound}, and justifies the bound \eqref{Stop}. 
Additionally, the error term $Err_{num}$ in \eqref{IterationBoundPreQuadratic} keeps track of several numerical approximations, including: 
\begin{itemize}
\item difference of continuous test functions $\f{\phi}$ and discrete counterparts $\f{\phi}_h$, 
\item difference of continuous mean-field limit $\rhobr[]$ and discrete counterpart $\rhobr[h]$, and 
\item difference of initial conditions between the particle system \eqref{eq_CDPartSys} and the discrete Dean--Kawasaki model \eqref{eq_discretizedDK}, see Assumption \ref{ass_InitCond} and Appendixes \ref{AppContReg}--\ref{AppConstructionInitialData}.
\end{itemize}

The argument for this subsection is spelled out in Proposition \ref{prop_ContAndDiscrMomStructure}.

\subsubsection{Linearising the convolution contributions}\label{subsubsec_OVErrlin} 
Here, we linearise the terms $Err_{lin,a}$, $Err_{lin,b}$ introduced above in order to make them suitable for the iteration. The linearisation is carried our using a $2d$-Fourier argument, which is proved in a separate result (Proposition \ref{prop_Iterative Structure for Errlin}), and discussed in more detail in Subsection \ref{subsubsec_TCErrlin} below.

\subsubsection{Closing the estimate}\label{subsubsec_OVClosingArgument} All arguments carried out so far are used to cumulate as many iteration steps as possible (depending on the initial regularity of the initial test functions $\f{\varphi},\psi$). When the iteration can no longer be performed, the remaining terms (other than $Err_{num}$) are bounded using quantitative convergence bounds to the mean-field limit (proven in Section \ref{SecQuantConvMeanF}), and the properties of the stopping time $\Ts$ (proven in Section \ref{SubsecExpSmallProb}).

\subsection{Technical challenges}\label{subsec_TechChall}

We highlight the four main technical challenges which we address in order to achieve the proof's building blocks sketched in Subsections \ref{subsubsec_OVBackEvoTest}--\ref{subsubsec_OVClosingArgument}.

\subsubsection{Quantitative convergence to the mean-field limit}\label{subsubsec_TCConvergenceMFL}
In several points of the argument, we need a quantitative bound on the convergence to the mean-field limit in both continuous and discrete setting. Specifically, in Proposition \ref{PropErrorEstimateAuxiliarySystem} we prove bounds of the type
\begin{align*}
& \big\| {\mu}_{\alpha,t}^{\rI,N} - \rhobr[\alpha](t) \big\|_{H^{-d/2-2}} \\
& \quad \leq\label{ErrorEstimateMeanFieldLimit_2}
	C(T+1) \exp\left(C T r_I^{-(d+2)}\right)\left(\big\| {\mu}_{0}^{\rI,N} - \rhobr(0) \big\|_{H^{-d/2-2}}	+ \mathcal{C} N^{-1/2}\right) 
\end{align*} 
for a random variable $\mathcal{C}$ with Gaussian moments $\Ev\big[\exp(C\mathcal{C}^2)\big]\leq 3$.
The corresponding estimate on the discrete level is a by-product of Section \ref{SubsecExpSmallProb} below.

\subsubsection{Linearisation argument for convolutional nonlinearity}\label{subsubsec_TCErrlin}

Linearising the quadratic terms $\tilQ[t](\phit), \tilQ[h,t](\phit[h])$ (given in \eqref{quad_term} and \eqref{quad_term_discrete}) essentially revolves around rewriting them into an infinite sum of suitable objects. 
If the interaction potential is regular enough, the Fourier expansion allows to rewrite the function $(x,y)\mapsto \partial_{\ell}\V[\alpha\beta](x-y)\colon \mathbb{T}^{2d}\rightarrow \mathbb{R}$ as 
	$\partial_{\ell}\V[\alpha\beta](x-y)= \sum_{m,n\in\mathbb{Z}^d}{\hat{F}_{k,\alpha\beta}^\ell[m,n]e^{im\cdot x}e^{in\cdot y}}$, 
effectively separating the variables $x,y$ in each addend of the sum.
We use this to rewrite $\tilQ[t](\phit)$ as 
\begin{multline*}
	\tilQ[t](\phit) \propto \sum_{\alpha,\beta=1}^\nS\sum_{\ell=1}^d\sum_{n,m\in\Z^d} \hat{F}_{k,\alpha\beta}^\ell[m,n] 
	\left(N^{1/2}\sum_{\gamma=1}^\nS\big\langle\delta^{\beta\gamma}\vtheta_m,\empmeas[\gamma,t]-\rhobr[\gamma](t)\big\rangle\right)\\
	\times \left( N^{1/2}\sum_{\gamma=1}^\nS\big\langle\delta^{\alpha\gamma}\partial_\ell\phit[\alpha]\vtheta_n,\empmeas[\gamma,t]-\rhobr[\gamma](t)\big\rangle\right),
\end{multline*}
where we have set $\vtheta_n(x) := \cos(n\cdot x)$, and where $\delta^{\cdot,\cdot}$ is the usual Kronecker delta. The above expression is now compatible with the structure of $\mathcal{M}$, as the nonlinearity has been `split' -- at the expense of having an infinite sum -- in two suitable stand-alone contributions (in the round brackets). Analogous discussions are applicable for $\tilQ[h,t\wedge\Ts](\phi^t_h)$.

\subsubsection{Estimate for probability of large fluctuations}\label{subsubsec_TCErrneg}
The exponentially decaying estimate \eqref{NegBoundRig} for $P[\Ts < T]$ 
is centered around proving Gaussian type moment bounds for the quantities $\rhor[h,\alpha]-\rhobr[h,\alpha]$. 
For this purpose, two different stopping times arguments are needed: i) an $L^{\infty}$-type stopping time, which controls the size of the stochastic noise, and ii) a $H^{-2\floor{d/2}-2}_h$-type stopping time (for which the estimates in Subsection \ref{subsubsec_TCConvergenceMFL} are needed), which allows to linearise the contribution of the quadratic convolutional nonlinearity.

\subsubsection{Compatible discrete setting}\label{subsubsec_DiscrDiff}
A number of auxiliary results is needed in order to substantiate our numerical approximations. The two main difficulties here are related to:
\begin{itemize}
\item
 quantifying the difference of relevant continuous functions (namely, the nonlinear mean-field limit and test functions) and bounding their discretised counterparts in higher-order Sobolev norms, see Appendixes \ref{AppContReg}--\ref{AppDiscrete}, and;
 \item constructing a compatible initial profile of fluctuations in the case of high order operators (i.e., $p>1$): we use arguments from polynomial interpolation theory, see Appendix \ref{AppConstructionInitialData}.
 \end{itemize}


\section{The key step - setting up the iterative structure}\label{TheKeyStepSection}
In this section we formalize the arguments outlined in Subsection \ref{SecInductiveStep}.
That is, we compute and compare the It\^o formulas (given by It\^o's rule) for
\begin{align*}
	\Ev[\psi(\bzeta^T)] &= \Ev\big[\psi\big(N^{1/2}\big\langle\bvphi,\empmeas[\bT]-\rhobr(\bT)\big\rangle\big)\big],\\
	\Ev[\psi(\bzeta^T_h)] &= \Ev\big[\psi\big(N^{1/2}\big(\Ih[\bvphi],(\rhor[h]-\rhobr[h])(\bT\wedge\Ts)\big)_h\big)\big].
\end{align*}
Theorem \ref{thm_iterative structure} below is the integrated, more precise version of \eqref{IterationBound}, that also provides bounds on $\widetilde{\psi}$, $\widetilde{\bphi}$, $\mathrm{Err}_{num}$, and $\mathrm{Err}_{neg}$.
On the first read of this section, the reader may wish to simply focus on the structure 
of the proofs and the comparison of the various blocks therein, and skip the quantitative estimates (such estimates rely on the results from the subsequent Section \ref{SecQuantConvMeanF}, Section \ref{SubsecExpSmallProb}, Appendix \ref{AppContReg} and Appendix \ref{AppDiscrete}). 

\begin{theorem}[Iterative Structure]\label{thm_iterative structure} 
Let $0<T$.
On $[0,T]$ let
\begin{itemize}[noitemsep, topsep=0pt, left= 10pt]
	\item $\empmeas=(\empmeas[\alpha])_{\alpha=1}^\nS$ as in \eqref{eq_empirical measure} be the empirical measures of the  particle system \eqref{eq_CDPartSys} satisfying Assumption \ref{ass_CDsystem+MFL},
	\item $\rhobr$ as in \eqref{eq_CD-IMFL} be the intermediate mean-field limit satisfying Assumption \ref{ass_RegularityMFL},
	\item $\rhor[h]$ solve the discretised Dean--Kawasaki model \eqref{eq_discretizedDK} satisfying Assumption \ref{ass_DiscrDefOps},
	\item $\rhobr[h]$ be the finite difference mean-field limit given in \eqref{eq_discrIntMeanFiLim}
\end{itemize}
Let Assumption \ref{ass_InitCond} (concerning the initial conditions) hold.
Let the parameters $\rI,N,h$ obey the scaling in Assumption \ref{ass_Scaling}.
Now let $K\in\N$, $q,r\in\Nz$, $q\geq 3$, and $s\in\N$, satisfying
\begin{equation}\label{eq_IS_sbounds}
	s> p+3d/2+5.
\end{equation}
Let $\psi\in\mathcal{L}_{pow,r}^{q}(\R^K)$,
$\bvphi=(\vphi_{k,\alpha})_{k=1,\ldots,K}^{\alpha=1,\ldots,\nS}\in[H^{s}(\domain,\R^{\nS})]^K$,
$\bT=(T_1,\ldots,T_K)\in[0,T]^K$.
Additionally, assume that $\rhobr\in L^\infty\left(0,T;[C^s(\domain)]^\nS\right)$ and $V\in W^{\sI,1}(\domain)$ for $\sI\in\N$ with
\begin{equation}\label{eq_IS_sIbounds}
	\sI > 2d+1+(r+2)\ceil{d/2+1}.
\end{equation}
Let $0<\ep<\delZero/4$, and let $\Ts=\Ts(\ep/2,\ep)$ be the associated stopping time as defined in \eqref{defn_stopping} at the start of Section \ref{SubsecExpSmallProb} for $\delta=\ep/2$.

Then the following facts hold. 
First, for all $t\in(0,T)$, $k,\tilde{k}\in\lbrace 1,\ldots,K\rbrace$, $m,n\in\Z^d$, and $\alpha,\beta\in\lbrace 1,\ldots,\nS\rbrace$ as well as $\ell\in\lbrace 1,\ldots,d\rbrace$ there exist generalised moment functions
\begin{equation*}
	\psi^0\in\mathcal{L}_{pow,r}^{q}(\R^K),  \quad
	\widetilde{\psi}_{k\tilde{k}}^t\in\mathcal{L}_{pow,r+1}^{q-2}(\R^{K+1}), \quad
	\psitlin[k]\in\mathcal{L}_{pow,r+2}^{q-1}(\R^{K+2}),
\end{equation*}
and coefficients 
	$\hat{F}_{k,\alpha\beta}^\ell[m,n]\in\R$
with sets of test functions
\begin{equation*}
	\bphi^0\in[H^s(\domain,\R^{\nS})]^K, \;\;,
	\bphittil[k\tilde{k}]\in [H^{s-1}(\domain,\R^{\nS})]^{K+1}, \;\;
	\bphitlin[k,mn,\alpha\beta,\ell]\in [H^{s}(\domain,\R^{\nS})]^{K+2},
\end{equation*}
and test times 
	$\tilde{\bT}_{k\tilde{k}}\in[0,T]^{K+1}$,
	$\bTlin[k]\in[0,T]^{K+2}$
such that
\begin{align}\label{eq_34a with adjusted Errlin}
	& \Ev\bigg[\psi\bigg(N^{1/2}\big\langle\bvphi,\empmeas[\bT]-\rhobr(\bT)\big\rangle\bigg)\bigg] 
	\\
	& \quad =
	\Ev\bigg[\psi^0\bigg(N^{1/2}\big\langle\bphi^0,\empmeas[0]-\rhobr(0)\big\rangle\bigg)\bigg]
	\nonumber\\
	& \quad \quad
	+ N^{-1/2}\sum_{k,\tilde{k}=1}^K \int_0^{T_k\wedge T_{\tilde{k}}} 
		\Ev\bigg[\widetilde{\psi}_{k\tilde{k}}^t\bigg(N^{1/2}\big\langle\bphittil[k\tilde{k}],\empmeas[t\wedge\bT_{k\tilde{k}}]-\rhobr(t\wedge\f{\tilde{T}}_{k\tilde{k}})\big\rangle\bigg)\bigg]\m t
	\nonumber\\
	& \quad \quad 
	+N^{-1/2}\sum_{k=1}^K\sum_{\alpha,\beta=1}^\nS\sum_{\ell=1}^d\int_0^{T_k}\sum_{n,m\in\Z^d} \hat{F}_{k,\alpha\beta}^\ell[m,n]
	\nonumber\\
	& \quad  \quad \quad \quad \times\Ev\bigg[\psitlin[k]\bigg(N^{1/2}\big\langle\bphitlin[k,mn,\alpha\beta,\ell], \empmeas[t\wedge{\bTlin[k]}]-\rhobr(t\wedge\bTlin[k]) \big\rangle\bigg)\bigg]\m t,
	\nonumber
\end{align}
as well as
\begin{align}\label{eq_34b with adjusted Errlin}
	& \Ev\bigg[\psi\bigg(N^{1/2}\big(\Ih[\bvphi],(\rhor[h]-\rhobr[h])(\bT\wedge\Ts)\big)_h\bigg)\bigg]\\
	& \quad =	
	\Ev\bigg[\psi^0\bigg(N^{1/2}\big(\Ih[\bphi^0],(\rhor[h]-\rhobr[h])(0)\big)_h\bigg)\bigg]
	\nonumber\\
	& \quad \quad 
	+ N^{-\frac{1}{2}}\!\!\sum_{k,\tilde{k}=1}^K \! \int_0^{T_k\wedge T_{\tilde{k}}} \!
	\mean{\widetilde{\psi}_{k\tilde{k}}^t\big(N^{1/2}(\Ih\big[\bphittil[k\tilde{k}]\big],(\rhor[h]-\rhobr[h])(t\wedge\f{\tilde{T}}_{k\tilde{k}}\wedge\Ts))_h\big)}\m t
	\nonumber\\
	& \quad \quad 
	+ N^{-\frac{1}{2}}\sum_{k=1}^K\sum_{\alpha,\beta=1}^\nS\sum_{\ell=1}^d\int_0^{T_k}\sum_{n,m\in\Z^d} \hat{F}_{k,\alpha\beta}^\ell[m,n]
	\nonumber\\
	& \quad \quad \quad 
		\times\Ev\bigg[\psitlin[k]\bigg(N^{1/2}\big(\Ih[\bphitlin[k,mn,\alpha\beta,\ell]], (\rhor[h]-\rhobr[h])(t\wedge\bTlin[k]\wedge\Ts) \big)_h\bigg)\bigg]\m t
	\nonumber\\
	& \quad \quad +\mathrm{Err}_{neg}+\mathrm{Err}_{num}+\Errstop.
	\nonumber
\end{align}
Second, for any $\delflex>0$ there exists $N_0=N_0(\delflex,\ep,r,s,K,data)$, such that for all $N>N_0$ the following estimates hold.
The stopping time satisfies
\begin{align}\label{eq_IS_TsEst}
	\Pm\left(\Ts<T\right) 
	\leq
	C \exp\big( - C N^{\ep/2} \big).
\end{align}
For $t\in(0,T)$, $k,\tilde{k}\in\lbrace 1,\ldots,K\rbrace$, $m,n\in\Z^d$, the bounds 
\begin{align}
	\|\psi^0\|_{{\mathcal{L}}_{pow,r}^{q}}
	&\leq
	C N^{\delflex}\left(1+\|\bvphi\|_{H^{s}}^r\right) \|\psi\|_{{\mathcal{L}}_{pow,r}^{q}},
	\label{eq_IS_psi0Est}\\
	\|\widetilde{\psi}_{k\tilde{k}}^t\|_{\mathcal{L}_{pow,r+1}^{q-2}}
	&\leq
	C N^{\delflex}\left(1+\|\bvphi\|_{H^{s}}^r\right) \|\psi\|_{{\mathcal{L}}_{pow,r}^{q}},
	\label{eq_IS_psiRegEst}\\
	\|\psitlin[k]\|_{\mathcal{L}_{pow,r+2}^{q-1}} 
	&\leq 
	C N^{\delflex}\left(1+\|\bvphi\|_{H^{s}}^r\right) \|\psi\|_{{\mathcal{L}}_{pow,r}^{q}}	
	\label{eq_IS_psiCompEst}
\end{align}
hold.
For all $t\in(0,T)$, $k\in\lbrace 1,\ldots,K\rbrace$, $m,n\in\Z^d$, $\alpha,\beta\in\lbrace 1,\ldots,\nS\rbrace$ and $\ell\in\lbrace 1,\ldots,d\rbrace$ the coefficients $\hat{F}_{k,\alpha\beta}^t[m,n]$ are subject to
\begin{align}\label{eq_IS_FhatEst}
	|\hat{F}_{k,\alpha\beta}^t[m,n]|
	\leq 
	C N^{\delflex}(|m|_2+|n|_2)^{-\sI+1}.
\end{align}
The test functions are subject to
\begin{align}	
	\|\bphi^0\|_{H^s}
	&\leq 
	C N^{\delflex} \|\bvphi\|_{H^s},
	\label{eq_IS_phi0Est}\\
	\|\bphittil[k\tilde{k}]\|_{H^{s-1}}
	&\leq 
	C N^{\delflex}\left(1+\|\bvphi\|_{H^s}\right)\|\bvphi\|_{H^s},
	\label{eq_IS_phiRegEst}\\
	\|\bphitlin[k,mn,\alpha\beta,\ell]\|_{H^{s-1}}
	&\leq 
	C N^{\delflex}\left(1+\|\bvphi\|_{H^{s}}\right)\left(1+|m|_\infty^{s-1}\vee|n|_\infty^{s-1}\right).
	\label{eq_IS_phiCompEst}
\end{align}
Furthermore, $\Errstop$, $\mathrm{Err}_{neg}$ and $\mathrm{Err}_{num}$ can be estimated via
\begin{align}
	|\Errstop| &\leq C\|\psi\|_{{\mathcal{L}}_{pow,r}^{2}}\big(1+ \|\bvphi\|_{H^s}^{2r+2}\big)N^{r\ep+\delflex}\exp\left(-CN^{\ep/2}\right),
	\label{eq_IS_ErrstopEst}\\
	|\mathrm{Err}_{neg}|&= 0,
	\label{eq_IS_ErrnegEst}\\
	|\mathrm{Err}_{num}|&\leq C\|\psi\|_{\mathcal{L}_{pow,r}^3} \big(1 + \|\bvphi\|_{H^s}^{2r+3}\big)N^{(r+1)\ep+\delflex}h^{p+1}.
	\label{eq_IS_ErrnumEst}
\end{align}
\end{theorem}

\begin{proof}
The result will follow from Proposition \ref{prop_ContAndDiscrMomStructure} and Proposition \ref{prop_Iterative Structure for Errlin}. 
\end{proof}


\subsection{The backwards evolution equations for the test functions}\label{subsec_BackEvos}
We first look at the dynamics of the empirical measure and then the solution of the discretized Dean--Kawasaki equation. 
This will lead to a choice for the backwards evolution for the test functions in both the continuous and discrete setting, summarized in Lemma \ref{lem_backEvoTest} and Lemma \ref{lem_discrBackEvoTest} respectively.

\subsubsection{Dynamics of the empirical measure}
Taking a set of generic time-dependent test functions $\eta_\alpha\colon[0,T]\times\domain\rightarrow\R$, writing $\etat[\alpha]=\eta_\alpha(t,\cdot)$, we apply the It\^o rule and calculate
\begin{align*}
	\mbox{d}\big\langle\etat[\alpha],\empmeas[\alpha,t]\big\rangle 
	& = \bigg\langle \partial_t\etat[\alpha] 
	+ \sigma_\alpha\Delta\etat[\alpha] 
	- \nabla\etat[\alpha] \cdot \bigg(\sum_{\beta=1}^\nS\nabla\V[\alpha\beta]\ast\empmeas[\beta,t]\bigg),
	\empmeas[\alpha,t]\bigg\rangle\m t\\
	& \quad + \frac{\sqrt{2\sigma_\alpha}}{N}\sum_{i=1}^N\nabla\etat[\alpha](X^{\rI,N}_{\alpha,i}(t))\m B_{\alpha,i}(t).
\end{align*}
We linearise the interaction induced term around the mean field limit, that is
\begin{align*}
	& \bigg\langle\nabla\etat[\alpha] \cdot \bigg(\sum_{\beta=1}^\nS\nabla\V[\alpha\beta]\ast\empmeas[\beta,t]\bigg),
	\empmeas[\alpha,t]\bigg\rangle \\
	& \quad
	= \bigg\langle\nabla\etat[\alpha] \cdot \bigg(\sum_{\beta=1}^\nS\nabla\V[\alpha\beta]\ast\rhobr[\beta](t)\bigg),
	\empmeas[\alpha,t]\bigg\rangle\\
	& \quad \quad 
	+ \sum_{\beta=1}^\nS \big\langle\nabla\etat[\alpha] \cdot \big(\nabla\V[\alpha\beta]\ast\big(\empmeas[\beta,t]-\rhobr[\beta](t)\big)\big),
	\rhobr[\alpha](t)\big\rangle
	+\tilQ[\alpha,t](\etat[\alpha]),
\end{align*}
where $\tilQ[\alpha,t](\etat[\alpha])$ is the quadratic linearisation compensation given by
\begin{equation*}
	\tilQ[\alpha,t](\etat[\alpha])
	\coloneqq \sum_{\beta=1}^\nS\big\langle\nabla\etat[\alpha] \cdot \big(\nabla\V[\alpha\beta]\ast\big(\empmeas[\beta,t]-\rhobr[\beta](t)\big)\big),
	\empmeas[\alpha,t]-\rhobr[\alpha](t)\big\rangle.
\end{equation*}
Introducing the notation 
$
	\U[\alpha](t)\coloneqq\sum_{\beta=1}^\nS\nabla\V[\alpha\beta]\ast\rhobr[\beta](t),
$
and by subtracting the mean field limit equation \eqref{eq_CD-IMFL}, we obtain
\begin{align}
	 \mbox{d}\big\langle\etat[\alpha],\empmeas[\alpha,t]-\rhobr[\alpha](t)\big\rangle &= \big\langle \partial_t\etat[\alpha] 
	+ \sigma_\alpha\Delta\etat[\alpha] 
	- \U[\alpha](t) \cdot \nabla\etat[\alpha],
	\empmeas[\alpha,t]-\rhobr[\alpha](t)\big\rangle\m t \label{eq_FlucEvoSingleSpec}\\
	& \quad  
	- \sum_{\beta=1}^\nS\big\langle\nabla\etat[\alpha] \cdot \big(\nabla\V[\alpha\beta]\ast\big(\empmeas[\beta,t]-\rhobr[\beta](t)\big)\big),	\rhobr[\alpha](t)\big\rangle\m t
	 \nonumber\\
	& \quad  -\tilQ[\alpha,t](\etat[\alpha])\m t + \frac{\sqrt{2\sigma_\alpha}}{N}\sum_{i=1}^N\nabla\etat[\alpha](X^{\rI,N}_{\alpha,i}(t))\m B_{\alpha,i}(t).\nonumber
\end{align}
Recalling the definition of $\ast_c$ (cfr. Section \ref{SecSettNotAss}) and summing over all species, we obtain
\begin{align*}
	& \sum_{\alpha=1}^\nS\sum_{\beta=1}^\nS\big\langle\nabla\etat[\alpha] \cdot \big(\nabla\V[\alpha\beta]\ast\big(\empmeas[\beta,t]-\rhobr[\beta](t)\big)\big),	\rhobr[\alpha](t)\big\rangle\\
	& \quad = -\sum_{\beta=1}^\nS \bigg\langle \sum_{\alpha=1}^\nS \nabla\V[\alpha\beta] \ast_c \big(\rhobr[\alpha](t)\nabla\etat[\alpha]\big)
	,\empmeas[\beta,t]-\rhobr[\beta](t)\bigg\rangle,
\end{align*}
where we used $\V[\alpha\beta](-x) =\V[\alpha\beta](x)$ for $x\in\domain$ (cfr. Assumption \ref{ass_CDsystem+MFL}).
Switching the species indices for this term and setting $\tilQ[t](\etat)=\sum_{\alpha=1}^\nS \tilQ[\alpha,t](\etat[\alpha])$ we have 
\begin{multline}\label{eq_dynamics of fluctuations}
	\mbox{d}\bigg(\sum_{\alpha=1}^\nS\big\langle\etat[\alpha],\empmeas[\alpha,t]-\rhobr[\alpha](t)\big\rangle\bigg)\\
	= \sum_{\alpha=1}^\nS\bigg\langle \partial_t\etat[\alpha] 
	+ \sigma_\alpha\Delta\etat[\alpha] 
	- \U[\alpha](t) \cdot \nabla\etat[\alpha]
	+ \sum_{\beta=1}^\nS \nabla\V[\beta\alpha] \ast_c \big(\rhobr[\beta](t)\nabla\etat[\beta]\big),
	\empmeas[\alpha,t]-\rhobr[\alpha](t)\bigg\rangle\m t\\
	-\tilQ[t](\etat) \m t
	+\sum_{\alpha=1}^\nS \frac{\sqrt{2\sigma_\alpha}}{N}\sum_{i=1}^N\nabla\etat[\alpha](X^{\rI,N}_{\alpha,i}(t))\m B_{\alpha,i}(t).
\end{multline}
Thus, we can read off the choice for the backwards evolution, which kills as much of the deterministic drift as possible, and obtain Lemma \ref{lem_backEvoTest} below.

\begin{lemma}[The backwards evolution equation for the test functions]\label{lem_backEvoTest}
Under the assumptions of Theorem \ref{thm_iterative structure},
let $\vphi=(\vphi_\alpha)_\alpha\in H^s(\domain,\R^\nS)$, $s> 2+d/2$.
Then there exists a unique $\phi\in C^1(0,T;H^s(\domain,\R^\nS))$ solving
\begin{equation}\label{eq_backEvoTest}
	\left\{\begin{aligned}
		-\partial_t\phit[\alpha] &=
				 \sigma_\alpha\Delta\phit[\alpha] 
				- \U[\alpha](t)\cdot\nabla\phit[\alpha]  
				+ \sum_{\beta=1}^\nS \nabla\V[\beta\alpha] \ast_c \big(\rhobr[\beta](t)\nabla\phit[\beta]\big)
				\quad\text{on }\domain\times(0,T),\\
		\phi^T_\alpha	&= \vphi_\alpha,
				\qquad \alpha=1,\ldots,\nS,
	\end{aligned}\right.
\end{equation}
where $\phit=\phi(t,\cdot)$ and for $\alpha=1,\ldots,\nS$, $t\in[0,T]$,
\begin{equation}\label{eq_def U}
	\U[\alpha](t)\coloneqq\sum_{\beta=1}^\nS\nabla\V[\alpha\beta]\ast\rhobr[\beta](t).
\end{equation}
Furthermore, it holds that 
\begin{multline}\label{eq_dynamics of fluctuations for backEvoTest}
	\mbox{d}\bigg(\sum_{\alpha=1}^\nS\big\langle\phit[\alpha],\empmeas[\alpha,t]-\rhobr[\alpha](t)\big\rangle\bigg)\\
	=-\tilQ[t](\phit) \m t
	+\sum_{\alpha=1}^\nS \frac{\sqrt{2\sigma_\alpha}}{N}\sum_{i=1}^N\nabla\phit[\alpha](X^{\rI,N}_{\alpha,i}(t))\m B_{\alpha,i}(t),
\end{multline}
where
\begin{equation}\label{eq_def tilQ}
	\tilQ[t](\phit)
	\coloneqq
	\sum_{\alpha,\beta=1}^\nS
	\big\langle\nabla\phit[\alpha] \cdot \big(\nabla\V[\alpha\beta]\ast\big(\empmeas[\beta,t]-\rhobr[\beta](t)\big)\big),
	\empmeas[\alpha,t]-\rhobr[\alpha](t)\big\rangle.
\end{equation}
\end{lemma}

\begin{proof}
For existence, uniqueness and regularity, see Lemma \ref{lem_RegularityContinuousTestFunctions}.
Then \eqref{eq_dynamics of fluctuations for backEvoTest} immediately follows from \eqref{eq_dynamics of fluctuations}, since the regularity is sufficient for the calculations to hold.
\end{proof}

\begin{rem}
Note that the reason, for which we are not able to kill all of the deterministic drift by evolving the test functions backwards, is the underlying non-linearity of the system.
Specifically, the linearization compensation $\tilQ[t](\cdot)$ associated with the nonlinear interaction terms of the cross-diffusion system is the only part of the drift that survives.
\end{rem}

\subsubsection{Dynamics of the solutions to the discretized Dean--Kawasaki model}
The following is the discrete version of Lemma \ref{lem_backEvoTest} (recall the definitions of $\ast_h,\ast_{h,c}$ from Section \ref{SecSettNotAss}).

\begin{lemma}[The backwards evolution equation for the discrete test functions]\label{lem_discrBackEvoTest}
Under the assumptions of Theorem \ref{thm_iterative structure}, 
let $\vphi_h=(\vphi_{h,\alpha})_\alpha\in L^2(\Ghd,\R^\nS)$.
Then there exists a unique $\phi_h\in C^1(0,T;L^2(\Ghd,\R^\nS))$ solving
\begin{equation}\label{eq_discrBackEvoTest}
	\left\{\begin{aligned}
		-\partial_t\phith[\alpha] 
			&= \sigma_\alpha \Delta_h\phith[\alpha]
			- \U[h,\alpha]\cdot\nabla_h\phith[\alpha]
			+ \sum_{\beta=1}^\nS \Ih[\nabla\V[\beta\alpha]]\ast_{h,c}(\rhobr[h,\beta]\nabla_h\phith[\beta]),\\
		\phi^T_{h,\alpha}&= \vphi_{h,\alpha},
				\qquad \alpha=1,\ldots,\nS,
	\end{aligned}\right.
\end{equation}
where for $\alpha=1,\ldots,\nS$, $t\in[0,T]$, we have set
\begin{equation}\label{eq_defUh}
	\U[h,\alpha](t)\coloneqq \sum_{\beta=1}^\nS \Ih[\nabla\V[\alpha\beta]]\ast_h\rhobr[h,\beta].
\end{equation}
Furthermore, it holds for $t\in[0,T]$ that 
\begin{align}\label{eq_discrDynamics of fluctuations for discrBackEvoTest}
	& \m\bigg(\sum_{\alpha=1}^\nS(\phith[\alpha],(\rhor[h,\alpha]-\rhobr[h,\alpha])(t))_h\bigg)
	\nonumber\\
	& \quad =
	-\tilQ[h](\phith)\m t \nonumber\\
	& \quad \quad \quad -\sum_{\alpha=1}^\nS \sqrt{2\sigma_\alpha} N^{-1/2} \sum_{\substack{y\in\Ghd \\ l\in\lbrace 1,\ldots,d\rbrace}} \bigg(\sqrt{(\rhor[h,\alpha](t))^+}\hbase[d]{y,l},\nabla_h \phith[\alpha]\bigg)_h \m  W_{(y,l)}^\alpha,
\end{align}
where $\lbrace f_{y,l}^d\rbrace$ is the basis of $L^2(\Ghd,\R^d)$ defined in Subsection \ref{FiniteDiffDiscr} and we have set
\begin{equation}\label{eq_deftilQh}
	\tilQ[h,t](\phith)
	\coloneqq
	\sum_{\alpha,\beta=1}^\nS
	\big(\nabla_h \phith[\alpha]\cdot (\Ih[\nabla\V[\alpha\beta]]\ast_h(\rhor[h,\beta]-\rhobr[h,\beta])(t)) ,(\rhor[h,\alpha]-\rhobr[h,\alpha])(t)\big)_h.
\end{equation}
\end{lemma}

\begin{proof}
We first take a set of generic discrete test functions $\eta_{h,\alpha}\colon [0,T]\times\Ghd\rightarrow\R$ which are differentiable in time.
Since $\lbrace \hbase{x}\rbrace_{x\in\Ghd}$ as defined in Section \ref{FiniteDiffDiscr} is an orthonormal basis of $\LzGhd$, we expand 
$
	\eta_{h,\alpha} = \sum_{x\in\Ghd}(\eta_{h,\alpha},\hbase{x})_h \hbase{x}
$. Applying the It\^o rule to 
\begin{equation*}
	(\eta_{h,\alpha},\rhor[h,\alpha])_h(t) = \sum_{x\in\Ghd} (\eta_{h,\alpha}(t),\hbase{x})_h (\rhor[h,\alpha](t),\hbase{x})_h
\end{equation*}
using \eqref{eq_discretizedDK}, and the expansion $\nabla_h\eta_{h,\alpha} = \sum_{x\in\Ghd}(\eta_{h,\alpha},\hbase{x})_h \nabla_h \hbase{x}$, we obtain 
\begin{align*}
	 \mbox{d}(\eta_{h,\alpha},\rhor[h,\alpha])_h
	& = \bigg[ (\partial_t\eta_{h,\alpha},\rhor[h,\alpha])_h + \sigma_\alpha (\eta_{h,\alpha},\Delta_h\rhor[h,\alpha])_h\\
	& \quad \quad - \sum_{\beta=1}^\nS\big(\nabla_h \eta_{h,\alpha},\rhor[h,\alpha](\Ih[\nabla\V[\alpha\beta]]\ast_h\rhor[h,\beta])\big)_h\bigg] \m t\\
	& \quad - \sqrt{2\sigma_\alpha} N^{-1/2} \sum_{\substack{y\in\Ghd \\ l\in\lbrace 1,\ldots,d\rbrace}}
					\bigg(\sqrt{(\rhor[h,\alpha])^+}\hbase[d]{y,l},\nabla_h \eta_{h,\alpha}\bigg)_h \m  W_{(y,l)}^\alpha.
\end{align*}
Following the same steps as in Lemma \ref{lem_backEvoTest}, we get the discrete analogue to \eqref{eq_dynamics of fluctuations}.
Inserting the test functions evolving according to the discrete backwards evolution \eqref{eq_discrBackEvoTest}, we obtain \eqref{eq_discrDynamics of fluctuations for discrBackEvoTest}.
The existence and uniqueness of the solution to \eqref{eq_discrBackEvoTest} follows from linearity.
\end{proof}

\subsection{The generalised moment structure}\label{subsec_MomStruc}
In \ref{subsubsec_contMomStruc} we analyze the generalised moment structure in the continuous case, highlighting how we choose the backwards evolution for the generalised moment functions.
In \ref{subsubsec_discrMomStruc} we calculate the resulting generalised moment structure for the discrete setting.
Finally, in \ref{subsubsec_34ab equivalents} we will compare the continuous and discrete setting, collecting and estimating the error terms from adjusting the discrete terms to fit the iterative structure.
By defining as well as estimating $\widetilde{\psi}_{k\tilde{k}}^t$, $\f{\widetilde{\phi}}_{k\tilde{k}}^t$, and $\tilde{\bT}_{k\tilde{k}}\in[0,T]^{K+1}$ we complete the first big step towards proving Theorem \ref{thm_iterative structure}.

\subsubsection{Deriving the generalised moment structure in the continuous setting}\label{subsubsec_contMomStruc}
\begin{lemma}\label{lem_pre34a equivalent}
Let the assumptions of Theorem \ref{thm_iterative structure} hold.
Then there exists a unique set of functions $\bphi=(\phi_1,\ldots,\phi_K)\in \left[C^1(0,T;H^s(\domain,\R^\nS))\right]^K$ such that
\begin{equation}\label{eq_BackEvoTestMulti}
	\left\{\begin{aligned}
		-\partial_t\phit[k,\alpha] &=
				 \sigma_\alpha\Delta\phit[k,\alpha] 
				- \U[\alpha](t)\cdot\nabla\phit[k,\alpha]  
				+ \sum_{\beta=1}^\nS \nabla\V[\beta\alpha] \ast_c \big(\rhobr[\beta](t)\nabla\phit[k,\beta]\big),
				\quad t<T_k,\\
		\phit[k,\alpha]	&= \vphi_{k,\alpha},
				\quad t\in[T_k,T]
	\end{aligned}\right.
\end{equation}
for all $\alpha=1,\ldots,\nS$ and each $k=1,\ldots,K$.
Further, there exists a unique backwards evolution of $\psi$ in $C^1(0,T;\mathcal{L}_{pow,r}^q(\R^K))$ satisfying
\begin{equation}\label{eq_backEvoMoment}
	\left\{\begin{aligned}
		-\partial_t\psit &=
				\sum_{k,\tilde{k}=1}^K 
				\chi_{t\leq T_k\wedge T_{\tilde{k}}}\sum_{\alpha=1}^\nS \sigma_\alpha 
				\big\langle \nabla\phit[k,\alpha]\cdot\nabla\phit[\tilde{k},\alpha],\rhobr[\alpha](t) \big\rangle
				\partial_k\partial_{\tilde{k}}\psit
				\quad\text{on }\R^K\times(0,T),\\
		\psi^T	&= \psi.
	\end{aligned}\right.
\end{equation}
Then, it holds that
\begin{multline}\label{eq_pre34a equivalent}
	\Ev\bigg[\psi\bigg(N^{1/2}\big\langle\bvphi,\empmeas[\bT]-\rhobr(\bT)\big\rangle\bigg)\bigg]\\
	\shoveleft{	
	= 
	\Ev\bigg[\psi^0\bigg(N^{1/2}\big\langle\bphi^0,\empmeas[0]-\rhobr(0)\big\rangle\bigg)\bigg]
	}\\
	\shoveleft{
	+N^{-1/2}\sum_{k,\tilde{k}=1}^K \int_0^{T_k\wedge T_{\tilde{k}}} 
	\Ev\bigg[\partial_k\partial_{\tilde{k}}\psit\bigg(N^{1/2}\big\langle\bphit,\empmeas[t\wedge\bT]-\rhobr(t\wedge\bT)\big\rangle\bigg)
	}\\
	\shoveright{
	\times N^{1/2}\sum_{\alpha=1}^\nS \big\langle \sigma_\alpha\nabla\phit[k,\alpha]\cdot\nabla\phit[\tilde{k},\alpha],\empmeas[\alpha,t]-\rhobr[\alpha](t) \big\rangle\bigg] \m t
	}\\
	\shoveleft{
	- N^{1/2}\sum_{k=1}^K\int_0^{T_k}\Ev\big[\partial_k\psit(N^{1/2}\big\langle\bphit,\empmeas[t\wedge\bT]-\rhobr(t\wedge\bT)\big\rangle)\tilQ[t](\phit[k])\big]\m t.}
\end{multline}
Additionally, 
there exists $N_0=N_0(\delflex,r,K,data)$ such that for all $0\leq \tilde{q}\leq q$ and $t\in[0,T]$
\begin{equation}\label{eq_psitLqpowr}
	\|\psi^t\|_{{\mathcal{L}}_{pow,r}^{\tilde{q}}}
	\leq
	C  N^{\delflex}\left(1+\|\vphi\|_{H^{s}}^r\right) \|\psi\|_{{\mathcal{L}}_{pow,r}^{\tilde{q}}}.
\end{equation}
\end{lemma}

\begin{proof}
The existence, uniqueness and regularity of $\bphi$ is settled as in Lemma \ref{lem_backEvoTest}.
For the existence, uniqueness and regularity of $\psit$ see \cite[Proof of Proposition 6, Step 1]{cornalba2021dean}.
These arguments applied for every $0\leq \tilde{q}\leq q$ yield
\begin{align*}
	\|\psi^t\|_{{\mathcal{L}}_{pow,r}^{\tilde{q}}} 
	&\leq 
	C \left(1+ \|\bphi\|_{L^\infty(0,T;W^{1,\infty}(\domain))}^r\right)\|\psi\|_{{\mathcal{L}}_{pow,r}^{\tilde{q}}}\\
	&\leq
	C \left(1+ \|\bphi\|_{L^\infty(0,T;H^{\ceil{(d+3)/2}}(\domain))}^r\right)\|\psi\|_{{\mathcal{L}}_{pow,r}^{\tilde{q}}}.
\end{align*}
Now \eqref{eq_psitLqpowr} follows from $s\geq \ceil{(d+3)/2}$ and Lemma \ref{lem_RegularityContinuousTestFunctions}.
We now turn to poving \eqref{eq_pre34a equivalent}.
In what follows, the regularity of $\bphi$ and $\psit$ is sufficient for the calculations to hold. 

For the convenience of notation we set
\begin{align*}
	\bzetat 
	&\coloneqq 
	N^{1/2}\big\langle\bphit,\empmeas[t\wedge\bT]-\rhobr(t\wedge\bT)\big\rangle.
\end{align*}
Using this notation, we rewrite \eqref{eq_dynamics of fluctuations for backEvoTest} from Lemma \ref{lem_backEvoTest} for $k=1,\ldots,K$ as
\begin{equation*}
	\mbox{d} \zetat[k] 
	= -\chi_{t\leq T_k}N^{1/2}\tilQ[t](\phit[k])\m t 
	+ \mbox{d} M^t_k
\end{equation*}
with martingale part  $M^t_k$ given by 
\begin{equation*}
	\mbox{d} M^t_k \coloneqq 
	\chi_{t\leq T_k} \sum_{\alpha=1}^\nS \frac{\sqrt{2\sigma_\alpha}}{N^{1/2}}\sum_{i=1}^N\nabla\phit[k,\alpha](X^{\rI,N}_{\alpha,i}(t))\m B_{\alpha,i}(t).
\end{equation*}
Thus, as outlined in \ref{ParticleCovStructure}, we observe the linearity preserving property of the cross-variation (see Subsection \ref{subsubsec_OVBackEvoTest}, specifically \eqref{ParticleCovStructure}), which yields for $k,\tilde{k}\in\lbrace 1,\ldots,K\rbrace$ that
\begin{align}\label{eq_empmeasCrossvariation}
	\mbox{d}\left[ M^t_k, M^t_{\tilde{k}} \right]
	&= \chi_{t\leq T_k}\chi_{t\leq T_{\tilde{k}}} \sum_{\alpha=1}^\nS \frac{2\sigma_\alpha}{N}\sum_{i=1}^N
		\nabla\phit[k,\alpha](X^{\rI,N}_{\alpha,i}(t))\cdot\nabla\phit[\tilde{k},\alpha](X^{\rI,N}_{\alpha,i}(t))\m t\notag\\
	&= \chi_{t\leq T_k\wedge T_{\tilde{k}}}\sum_{\alpha=1}^\nS 2\sigma_\alpha \big\langle \nabla\phit[k,\alpha]\cdot\nabla\phit[\tilde{k},\alpha],\empmeas[\alpha,t] \big\rangle \m t.
\end{align}
Combining these observations with the It\^o rule 
and taking the expected value yields
\begin{multline*}\label{eq_InfinitIncr}
	\mbox{d}\Ev\big[\psit(\bzetat)\big] 
	=
	\Ev\big[\partial_t\psit(\bzetat)\big]\m t
	-\Ev\bigg[\sum_{k=1}^K \partial_k\psit(\bzetat)\chi_{t\leq T_k}N^{1/2}\tilQ[t](\phit[k])\bigg]\m t\\
	+\Ev\bigg[\sum_{k,\tilde{k}=1}^K \partial_k\partial_{\tilde{k}}\psit(\bzetat)
	\chi_{t\leq T_k\wedge T_{\tilde{k}}}\sum_{\alpha=1}^\nS \sigma_\alpha \big\langle \nabla\phit[k,\alpha]\cdot\nabla\phit[\tilde{k},\alpha],\empmeas[\alpha,t] \big\rangle\bigg] \m t.
\end{multline*}
Thinking of the term with $\tilQ[t](\phit[k])$ as a linearisation error, we see that \eqref{eq_backEvoMoment} is the clear choice for the backwards evolution of $\psit$ if we want the first two terms to fit the iterative structure, that is combining linear functionals of $\empmeas-\rhobr$. 
Plugging \eqref{eq_backEvoMoment} in and integrating the equation in time yields \eqref{eq_pre34a equivalent}.
\end{proof}

\subsubsection{Deriving the generalised moment structure in the discrete setting}\label{subsubsec_discrMomStruc}
Analogously to Lemma \ref{lem_pre34a equivalent} we obtain the following result in the discrete case. 
However, since the backwards evolution for $\psit$ is tailored for the continuous case, the generalised moment structure does not immediately show the iterative structure.
In Subsection \ref{subsubsec_34ab equivalents} we will estimate the cost of the adjustments necessary in order to obtain this structure.

\begin{lemma}\label{lem_pre34b equivalent}
Let the assumptions of Theorem \ref{thm_iterative structure} hold.
Then there exists a unique set of functions 
	$\bphi_h=(\phi_{h,1},\ldots,\phi_{h,K})\in \left[C^1(0,T;L^2(\Ghd,\R^\nS))\right]^K$
solving the respective discrete backwards evolution equation for test functions, that is
\begin{equation}\label{eq_discrBackEvoTestMulti}
	\left\{\begin{aligned}
		-\partial_t\phith[k,\alpha] 
			&= \sigma_\alpha \Delta_h\phith[k,\alpha]
			- \U[h,\alpha]\cdot\nabla_h\phith[k,\alpha]
			+ \sum_{\beta=1}^\nS \Ih[\nabla\V[\beta\alpha]]\ast_{h,c}(\rhobr[h,\beta]\nabla_h\phi_{h,k,\beta})\\
			&\qquad\qquad\qquad\qquad\qquad\qquad\qquad\qquad\qquad\qquad\qquad\qquad\text{for } t<T_k,\\
		\phith[k,\alpha]&= \Ih[\vphi_{k,\alpha}]
			\quad \text{for } t\in[T_k,T]
	\end{aligned}\right.
\end{equation}
for all $\alpha=1,\ldots,\nS$ and each $k=1,\ldots,K$.
Further, let $t\rightarrow \psit$ with $\psi^T=\psi$ be the solution of \eqref{eq_backEvoMoment} from Lemma \ref{lem_pre34a equivalent}.
As assumed in Theorem \ref{thm_iterative structure}, let $\Ts=\Ts(\ep/2,\ep)$ be the stopping time as defined in \eqref{defn_stopping}.
Then 
\begin{multline}\label{eq_pre34b equivalent}
	\Ev\big[\psi\big(N^{1/2}\big(\Ih[\bvphi],(\rhor[h]-\rhobr[h])(\bT\wedge\Ts)\big)_h\big)\big]
	\\
	\shoveleft{
	=
	\Ev\bigg[\psi^0\bigg(N^{1/2}\big(\bphi^0_h,(\rhor[h]-\rhobr[h])(0)\big)_h \bigg)\bigg]
	}\\
	\shoveleft{
	+\sum_{k,\tilde{k}=1}^K\sum_{\alpha=1}^\nS \sigma_\alpha \int_0^{T_k\wedge T_{\tilde{k}}}
	\Ev\bigg[\partial_k\partial_{\tilde{k}}\psi^t\bigg(N^{1/2}\big(\bphith,(\rhor[h]-\rhobr[h])(t\wedge\bT\wedge\Ts)\big)_h\bigg)
	}\\
	\shoveright{
	\times \left(\big(\nabla_h \phith[k,\alpha]\cdot\nabla_h \phith[\tilde{k},\alpha],(\rhor[h,\alpha])^+(t\wedge\Ts)\big)_h
	-\big\langle \nabla\phit[k,\alpha]\cdot\nabla\phit[{\tilde{k},\alpha}],\rhobr[\alpha](t\wedge\Ts) \big\rangle\right)\bigg]\m t
	}\\
	\shoveleft{
	-N^{1/2}\sum_{k=1}^K \int_0^{T_k}\Ev\bigg[\partial_k\psi^t\bigg(N^{1/2}\big(\bphith,(\rhor[h]-\rhobr[h])(t\wedge\bT\wedge\Ts)\big)_h\bigg)\tilQ[h,t\wedge\Ts](\phith[k])\bigg]\m t.
	}\\
	\shoveright{
	+\Errstop.}
\end{multline}
For any $\delflex>0$, 
there exists $N_0=N_0(\delflex,\ep,r,s,K,\text{data})$, such that $\Errstop$ is subject to \eqref{eq_IS_ErrstopEst} for all $N>N_0$.
\end{lemma}

\begin{proof}
The existence, uniqueness and regularity for $\bphi[h]$ is as in Lemma \ref{lem_discrBackEvoTest}.
The regularity is sufficient for the following calculations to hold. 

{\bf Step 1: \eqref{eq_pre34b equivalent} and defining the error term.} 
We denote
\begin{align*}
	\bzetath 
	&\coloneqq 
	N^{1/2}\big(\bphith,(\rhor[h]-\rhobr[h])(t\wedge\bT\wedge\Ts)\big)_h.
\end{align*}
Rewriting \eqref{eq_discrDynamics of fluctuations for discrBackEvoTest} in Lemma \ref{lem_discrBackEvoTest} with this notation yields
\begin{align*}
	\mbox{d}\zetatTsh[k]
	=
	-\chi_{t\leq T_k\wedge\Ts}N^{1/2}\tilQ[h,t](\phith[k])\m t
	+\mbox{d}M^{t\wedge\Ts}_{h,k}
\end{align*}
with martingale part $M_{h,k}$ given by
\begin{align*}
	\mbox{d}M^{t\wedge\Ts}_{h,k} \coloneqq 
	-\chi_{t\leq T_k\wedge\Ts} \sum_{\alpha=1}^\nS \sqrt{2\sigma_\alpha}
	\sum_{\substack{y\in\Ghd \\ l\in\lbrace 1,\ldots,d\rbrace}} \bigg(\sqrt{(\rhor[h,\alpha](t))^+}\hbase[d]{y,l},\nabla_h \phith[k,\alpha]\bigg)_h \m W_{(y,l)}^\alpha(t).
\end{align*}
The cross variation terms with $k,\tilde{k}\in\lbrace1,\ldots,K\rbrace$, as detailed in Subsection \ref{subsubsec_OVBackEvoTest} (specifically \eqref{DiscreteCovStructure}), are given by 
\begin{align*}
	\mbox{d}\left[ M^{t\wedge\Ts}_{h,k}, M^{t\wedge\Ts}_{h,\tilde{k}} \right]
	&= \chi_{t\leq T_k\wedge T_{\tilde{k}}\wedge\Ts} \sum_{\alpha=1}^\nS 2\sigma_\alpha \big(\nabla_h \phith[k,\alpha]\cdot\nabla_h \phith[\tilde{k},\alpha],(\rhor[h,\alpha](t))^+\big)_h.
\end{align*}
Now as in the continuous case we have
\begin{multline}\label{eq_pre34bwoErrstop}
	\mbox{d}\psi^t(\bzetatTsh)
	=\partial_t\psi^t(\bzetatTsh)\m t 
	-\sum_{k=1}^K \partial_k\psi^t(\bzetatTsh)\chi_{t\leq T_k\wedge\Ts}N^{1/2}\tilQ[h](\phith[k])\m t\\
	\shoveright{
	+\sum_{k,\tilde{k}=1}^K \partial_k\partial_{\tilde{k}}\psi^t(\bzetatTsh) 
			\chi_{t\leq T_k\wedge T_{\tilde{k}}\wedge\Ts} \sum_{\alpha=1}^\nS \sigma_\alpha \big(\nabla_h \phith[k,\alpha]\cdot\nabla_h \phith[\tilde{k},\alpha],(\rhor[h,\alpha](t))^+\big)_h\m t
	}\\
	+\sum_{k=1}^K\partial_k\psit(\bzetatTsh) \m M^{t\wedge\Ts}_{h,k}.
\end{multline}
Integrating in time from $0$ to $\Ts$, plugging in \eqref{eq_backEvoMoment} and then taking the expected value we obtain \eqref{eq_pre34b equivalent} where the error term is given by $\Errstop=\Errstop[1]+\Errstop[2]$.
The first error term $\Errstop[1]$ stems from the correction $(\psi^\Ts,\bphi^\Ts_h)\rightarrow (\psi^T,\bphi^T_h)=(\psi,\bvphi)$ given by
\begin{multline*}
	\Errstop[1]
	\coloneqq
	\Ev\bigg[\chi_{\Ts<T}\bigg(\psi\big(N^{1/2}\big(\Ih[\bvphi],(\rhor[h]-\rhobr[h])(\bT\wedge\Ts)\big)_h\big)\\
		-\psi^\Ts\big(N^{1/2}\big(\bphi^\Ts_h,(\rhor[h]-\rhobr[h])(\bT\wedge\Ts)\big)_h \big)\bigg)\bigg],
\end{multline*}
while $\Errstop[2]$ comes from extending the integration to $T$, i.e. $\chi_{t\leq\Ts}\rightarrow 1$, that is 
\begin{align*}
	\Errstop[2]
	&\coloneqq
	-\sum_{k,\tilde{k}=1}^K\sum_{\alpha=1}^\nS \sigma_\alpha \int_0^{T_k\wedge T_{\tilde{k}}}
	\Ev\bigg[\chi_{\Ts<t\leq T}\partial_k\partial_{\tilde{k}}\psi^t\big(\bzetath\big)
	\\&
	\times \left(\big(\nabla_h \phith[k,\alpha]\cdot\nabla_h \phith[\tilde{k},\alpha],(\rhor[h,\alpha])^+(t\wedge\Ts)\big)_h
	-\big\langle \nabla\phit[k,\alpha]\cdot\nabla\phit[{\tilde{k},\alpha}],\rhobr[\alpha](t\wedge\Ts) \big\rangle\right)\bigg]\m t
	\\&\quad
	+N^{1/2}\sum_{k=1}^K \int_0^{T_k}\Ev\bigg[\chi_{\Ts<t\leq T}\partial_k\psi^t\big(\bzetath \big)\tilQ[h,t\wedge\Ts](\phith[k])\bigg]\m t.
\end{align*}

{\bf Step 2: Estimating the error terms.} 
For the following estimates we assume $N$ to be large enough to use the suitable auxiliary results.
Additionally, we will rely on the definition of the stopping time $\Ts$, see \eqref{defn_stopping} for $0<\ep<\delZero/4$ and $\delta=\ep/2$ as fixed in the assumptions of Theorem \ref{thm_iterative structure}.
Due to the definition of the ${{\mathcal{L}}_{pow,r}^{q}}$-spaces we have
\begin{align}
	&\big|\psi\big(N^{1/2}\big(\Ih[\bvphi],(\rhor[h]-\rhobr[h])(\bT\wedge\Ts)\big)_h\big)\big|\notag\\
	&\quad\leq
	\|\psi\|_{{\mathcal{L}}_{pow,r}^{0}}\big(1+ \big|N^{1/2}\big(\Ih[\bvphi],(\rhor[h]-\rhobr[h])(\bT\wedge\Ts)\big)_h\big|^2\big)^{r/2}\notag\\
	&\quad\leq
	\|\psi\|_{{\mathcal{L}}_{pow,r}^{0}}\big(1+ \|\Ih[\bvphi]\|_{H^{d+2}_h}^2 N\|(\rhor[h]-\rhobr[h])(\bT\wedge\Ts)\|_{H^{-2(\floor{d/2}+1)}_h}^2\big)^{r/2}\notag\\
	&\quad\leq 
	\|\psi\|_{{\mathcal{L}}_{pow,r}^{0}}\big(1+ C\|\bvphi\|_{C^{d+2}}^2 
	\big(N^{2\ep} + N\chi_{\Ts=0}\|(\rhor[h]-\rhobr[h])(0)\|_{H^{-2(\floor{d/2}+1)}_h}^2\big)  \big)^{r/2}\notag
\end{align}
where the last step follows from the definition of the stopping time $\Ts$.
Additionally plugging in \eqref{eq_DiscreteTestFuncBound} from Lemma \ref{HsBoundDiscrete}
to estimate $\|\bphi^\Ts_h\|_{\HGhd[d+2]}$ and \eqref{eq_psitLqpowr} from Lemma \ref{lem_pre34a equivalent} to bound $\|\psi^\Ts\|_{{\mathcal{L}}_{pow,r}^{q}}$, we obtain that -- as long as $N$ is large enough -- for any small $\delflex>0$ 
\begin{align*}
	&\psi^\Ts\big(N^{1/2}\big(\bphi^\Ts_h,(\rhor[h]-\rhobr[h])(\bT\wedge\Ts)\big)_h \big)\\
	&\quad\leq
	C\|\psi\|_{{\mathcal{L}}_{pow,r}^{0}}N^{\delflex}\big(1+\|\bvphi\|_{H^s(\domain)}^r\big)\\
	&\qquad\quad
	\times\Big(1+ N^{\delflex}\|\bvphi\|_{C^{d+2}(\domain)}^r 
	\big(N^{r\ep}+\chi_{\Ts=0}N^{r/2}\|(\rhor[h]-\rhobr[h])(0)\|_{H^{-2(\floor{d/2}+1)}_h}^r\big) \Big).
\end{align*}
Since Assumption \eqref{eq_IS_sbounds} implies $s>3d/2+2$, we have $\|\bvphi\|_{C^{d+3}}\leq C\|\bvphi\|_{H^s}$ via Sobolev embedding and thus get
\begin{multline}\label{eq_ErrstopProof2}
	\Errstop[1]
	\leq 
	\Pm(\Ts < T) C\|\psi\|_{{\mathcal{L}}_{pow,r}^{0}}\big(1+ \|\bvphi\|_{H^s}^{2r}\big)N^{r\ep+\delflex}\\
	+\Ev\Big[\chi_{\Ts = 0}\|(\rhor[h]-\rhobr[h])(0)\|_{H^{-2(\floor{d/2}+1)}_h}^r\Big]C\|\psi\|_{{\mathcal{L}}_{pow,r}^{0}}\big(1+ \|\bvphi\|_{H^s}^{2r}\big)N^{r/2}.
\end{multline}
For now ignoring the case $\Ts=0$, plugging in the bound for $\Pm(\Ts < T)$ from Proposition \ref{prop_StopTBound}, we obtain \eqref{eq_IS_ErrstopEst}.

With respect to the case $\Ts=0$, due to Assumption \ref{ass_InitCond}, the mass restriction and positivity of $\rhor[h](0)$, we obtain the very rough estimate
\begin{equation*}
	\|\rhor[h](0)\|_{H^{-2(\floor{d/2}+1)}_h} \leq  \|\rhor[h](0)\|_{L^2_h} \leq \|\rhor[h](0)\|_{L^\infty_h} \leq h^{-d}\|\rhor[h](0)\|_{L^1_h} \leq Ch^{-d}.
\end{equation*}
Furthermore, since $\rhobr[h](0)=\Ih[\rhobr(0)]$, we also get
\begin{equation*}
	\|\rhobr[h](0)\|_{H^{-2(\floor{d/2}+1)}_h} \leq \|\rhobr[h](0)\|_{L^2_h} \leq  C\|\rhobr(0)\|_{C^0}.
\end{equation*}
Both are controlled by the bound for $\Pm[\Ts = 0]\leq \Pm[\Ts<0]$: In particular
\begin{align*}
	\Pm\big[\Ts = 0\big] 
	&\leq  
	C \exp\big( -  C N^{\ep/2}\big) \leq
	CN^{-r/2}\left(h^{-d} + \|\rhobr(0)\|_{C^1}^{-r}\right)^{-r}\exp\big( -  C/2 N^{\ep/2}\big),
\end{align*}
where the second inequality is due to the scaling regime, see Assumption \ref{ass_Scaling}, for $N$ large enough (depending on $\|\rhobr(0)\|_{C^1}$).
Plugging these estimates into \eqref{eq_ErrstopProof2} we obtain the bound given in \eqref{eq_IS_ErrstopEst}. 

Concerning $\Errstop[2]$, with the integration we can ignore the case $\Ts=0$.
For $0<t<\Ts$ via Remark \ref{rem_rhominh_rhomaxh} we have
\begin{equation}
	\|(\rhor[h,\alpha])^+(t\wedge\Ts)\|_{L^\infty_h}
	\leq 
	\|(\rhobr[h])\|_{L^\infty_h} + N^{-\epsilon}
	\leq
	C\rho_{max,h},
\end{equation}
and using Corollary \ref{cor_DiscrHoelder}, Lemma \ref{HsBoundDiscrete}, and that $\|\nabla_h f\|_{L^2_h}\leq C\|f\|_{H^1_h}$ for $f\in\LzGhd$,
\begin{equation}
	\|\nabla_h \phith[k,\alpha]\cdot\nabla_h \phith[\tilde{k},\alpha]\|_{L^1_h}
	\leq
	C\|\bphith\|_{H^1_h}^2
	\leq 
	CN^\delflex \|\Ih[\bvphi]\|_{H^1_h}^2
	\leq
	CN^\delflex \|\bvphi\|_{H^s}^2.
\end{equation}
For the other term, note that 
\begin{align}
	\big|N^{1/2}\tilQ[h,t\wedge\Ts](\phith[k])\big|
	&\leq
	CN^{\ep+\delflex} \|\nabla_h\bphith\|_{L^1_h} \|(\rhor[h]-\rhobr[h])(t)\|_{L^\infty_h}\notag\\
	&\leq
	C\|\bvphi\|_{H^s}N^{\delflex}
	\label{eq_ErrstopProofQEst},
\end{align}
where for the first inequality we estimated the convolution pointwise via the definition of $\Ts$.
Together with Assumption \ref{ass_RegularityMFL} we obtain 
\begin{equation*}
	\Errstop[2] 
	\leq 
	C \Pm(\Ts<T) \|\psi\|_{\mathcal{L}_{pow,r}^2}(1+\|\bvphi\|_{H^s}^r)\|\bvphi\|_{H^s}^r N^{r\epsilon+\delflex}(\|\bvphi\|_{H^s}^2+\|\bvphi\|_{H^s}).
\end{equation*}
Combining this with the bound for $\Pm(\Ts<T)$, the scaling from Assumption \ref{ass_Scaling} and the estimate for $\Errstop[1]$ we obtain \eqref{eq_IS_ErrstopEst}.
\end{proof}

\subsubsection{Comparing the generalised moments, the iterative structure}\label{subsubsec_34ab equivalents}
Now, comparing \eqref{eq_pre34a equivalent} and \eqref{eq_pre34b equivalent}, up to some small error terms we obtain the iterative structure, i.e. the same type of terms we started with but with an additional small prefactor of $N^{-1/2}$.

\begin{proposition}\label{prop_ContAndDiscrMomStructure}
Under the assumptions of Theorem \ref{thm_iterative structure},
for all $t\in(0,T)$, as well as $k,\tilde{k}\in\lbrace 1,\ldots,K\rbrace$, $\alpha,\beta\in\lbrace 1,\ldots,\nS\rbrace$ there exist functions
\begin{equation*}
	\psi^0\in\mathcal{L}_{pow,r}^{q}(\R^K),  \quad
	\widetilde{\psi}_{k\tilde{k}}^t\in\mathcal{L}_{pow,r+1}^{q-2}(\R^{K+1}),
\end{equation*}
sets of test functions
\begin{equation*}
	\bphi^0\in[H^s(\domain,\R^{\nS})]^K, \quad
	\bphittil[k\tilde{k}]\in [H^{s-1}(\domain,\R^{\nS})]^{K+1},
\end{equation*}
and test times 
	$\tilde{\bT}_{k\tilde{k}}\in[0,T]^{K+1}$
such that 
\begin{align}\label{eq_34a equivalent}
	&\Ev\bigg[\psi\bigg(N^{1/2}\big\langle\bvphi,\empmeas[\bT]-\rhobr(\bT)\big\rangle\bigg)\bigg]\\
	&\quad=
	\Ev\bigg[\psi^0\bigg(N^{1/2}\big\langle\bphi^0,\empmeas[0]-\rhobr(0)\big\rangle\bigg)\bigg]
	\notag\\
	&\qquad
	+\frac{1}{N^{1/2}}\sum_{k,\tilde{k}=1}^K \int_0^{T_k\wedge T_{\tilde{k}}} 
	\Ev\bigg[\widetilde{\psi}_{k\tilde{k}}^t\bigg(N^{1/2}\big\langle\bphittil[k\tilde{k}],\empmeas[t\wedge\bT_{k\tilde{k}}]-\rhobr(t\wedge\f{\tilde{T}}_{k\tilde{k}})\big\rangle\bigg)\bigg]\m t
	\notag\\
	&\qquad
	+\mathrm{Err}_{lin,a}
	\notag
\end{align}
and
\begin{multline}\label{eq_34b equivalent}
	\Ev\bigg[\psi\bigg(N^{1/2}\big(\Ih[\bvphi],(\rhor[h]-\rhobr[h])(\bT\wedge\Ts)\big)_h\bigg)\bigg]\\
	\shoveleft{
	=\Ev\bigg[\psi^0\bigg(N^{1/2}\big(\Ih[\bphi^0],(\rhor[h]-\rhobr[h])(0)\big)_h\bigg)\bigg]
	}\\
	\shoveleft{
	+ N^{-1/2}\sum_{k,\tilde{k}=1}^K \int_0^{T_k\wedge T_{\tilde{k}}} 
	\Ev\bigg[\widetilde{\psi}_{k\tilde{k}}^t\bigg(N^{1/2}\big(\Ih\big[\bphittil[k\tilde{k}]\big],(\rhor[h]-\rhobr[h])(t\wedge\f{\tilde{T}}_{k\tilde{k}}\wedge\Ts)\big)_h\bigg)\bigg]\m t
	}\\
	\shoveleft{
	+\Errstop+\mathrm{Err}_{neg}+\mathrm{Err}_{num}+\mathrm{Err}_{lin,b}.
	\hfill}
\end{multline}
The linearization errors are constructed using $\bphi$ as in \eqref{eq_BackEvoTestMulti}, $\psi^t$ as in \eqref{eq_backEvoMoment}, $\bphi[h]$ as in \eqref{eq_discrBackEvoTestMulti}, $\tilQ$ as in \eqref{eq_def tilQ}, and $\tilQ[h]$ as in \eqref{eq_deftilQh}: explicitly, they read
\begin{align}\label{eq_DefErrlina}
	Err_{lin,a} & := - N^{1/2}\sum_{k=1}^K\int_0^{T_k}\Ev\big[\partial_k\psit(N^{1/2}\big\langle\bphit,\empmeas[t\wedge\bT]-\rhobr(t\wedge\bT)\big\rangle)\tilQ[t](\phit[k])\big]\m t, \\
Err_{lin,b} & := -N^{1/2}\sum_{k=1}^K \int_0^{T_k}\Ev\big[\partial_k\psi^t\big(N^{1/2}\big(\bphith,(\rhor[h]-\rhobr[h])(t\wedge\bT\wedge\Ts)\big)_h\big)\label{eq_DefErrlinb}\\
	& \qquad \qquad \qquad \qquad \qquad \qquad \qquad \qquad \qquad  \qquad 
			\times\tilQ[h,t\wedge\Ts](\phi_{h,k}(t))\big]\m t,\nonumber
\end{align}
For any $\delflex>0$ there exists $N_0=N_0(\delflex,\ep,r,s,K,\text{data})$, such that for all $N>N_0$ the following estimates hold.
The stopping time satisfies \eqref{eq_IS_ErrstopEst}, 
the new generalised moment functions $\psi^0$ and $\{\widetilde{\psi}_{k\tilde{k}}^t\}_{{t,k,\tilde{k}},m,n}$ are subject to \eqref{eq_IS_psi0Est} and \eqref{eq_IS_psiRegEst} respectively. 
The test function $\bphi^0$ satisfy \eqref{eq_IS_phi0Est}, and the test functions $\{\bphittil[k\tilde{k}]\}_{t,k,\tilde{k}}$ satisfy \eqref{eq_IS_phiRegEst}.
Finally, for the error terms the estimates \eqref{eq_IS_ErrstopEst}, \eqref{eq_IS_ErrnegEst}, and \eqref{eq_IS_ErrnumEst} hold respectively.

\end{proposition}

\begin{rem}
The terms $\mathrm{Err}_{lin,a}$ and $\mathrm{Err}_{lin,b}$ do not show the itarative structure yet due to the convolutional structure.
In Subsection \ref{subsec_Errlin} below we will follow \ref{subsubsec_TCErrlin} to separate the two contributions of the quadratic non-linearity.
\end{rem}

\begin{proof}
In Step 1 and 2 we will introduce the definitions which yield \eqref{eq_34a equivalent} and \eqref{eq_34b equivalent} from Lemma \ref{lem_pre34a equivalent} and Lemma \ref{lem_pre34b equivalent} respectively. 
In Step 3, 4, and 5 we will prove the estimates. 
 
{\bf Step 1:  Defining $\widetilde{\psi}_{k\tilde{k}}^t$, $\bphittil[k\tilde{k}]$,$\tilde{\bT}_{k\tilde{k}}$ to obtain \eqref{eq_34a equivalent}.}
We introduce $\bphittil[k\tilde{k}](x)\colon \domain\rightarrow\R^{K+1}$ and $\tilde{\bT}_{k\tilde{k}}\in[0,T]^{K+1}$ with 
\begin{equation}\label{eq_DefTildePhi}
	\bphittil[k\tilde{k}](x) 
	= 
	\begin{pmatrix}
		\big(\phit[1,\alpha](x)\big)_{\alpha=1,\ldots,\nS}\\
		\vdots\\
		\big(\phit[K,\alpha](x)\big)_{\alpha=1,\ldots,\nS}\\
		\big(\sigma_\alpha\nabla\phit[k,\alpha](x)\cdot\nabla\phit[\tilde{k},\alpha](x)\big)_{\alpha=1,\ldots,\nS}
	\end{pmatrix},
	\qquad
	\tilde{\bT}_{k\tilde{k}}
	=
	\begin{pmatrix}
		T_1\\
		\vdots\\
		T_K\\
		T_k\wedge T_{\tilde{k}}
	\end{pmatrix}
\end{equation}
where $\bphit$ and thus particularly $\bphi^0$ come from \eqref{eq_BackEvoTestMulti}.
Further defining $\widetilde{\psi}_{k\tilde{k}}^t\colon \R^{K+1}\rightarrow\R$,
\begin{align*}
	\widetilde{\psi}_{k\tilde{k}}^t(\f{z})
	=
	\chi_{t\leq T_k\wedge T_{\tilde{k}}}\partial_k\partial_{\tilde{k}}\psi^t(z_1,\ldots,z_K)z_{K+1}
\end{align*}
where $\psit$, and thus also $\psi^0$ are given by \eqref{eq_backEvoMoment},
we immediately obtain \eqref{eq_34a equivalent} from \eqref{eq_pre34a equivalent}.

{\bf Step 2: \eqref{eq_34b equivalent} and defining the remaining error terms.}
From Lemma \ref{lem_pre34b equivalent} we have \eqref{eq_pre34b equivalent}.
Comparing this with \eqref{eq_34b equivalent}, there is a series of slight adjustments to obtain the appropriate form, for which we have to pay with the respective error term.
Recall that
\begin{align*}
	\bzetath 
	&\coloneqq 
	N^{1/2}\big(\bphith,(\rhor[h]-\rhobr[h])(t\wedge\bT\wedge\Ts)\big)_h.
\end{align*}
First, in order to adjust $(\rhor[h,\alpha])^+ \longrightarrow \rhor[h,\alpha]$, we have to pay
\begin{multline*}
	\mathrm{Err}_{neg}
	\coloneqq
	\sum_{k,\tilde{k}=1}^K\sum_{\alpha=1}^\nS \sigma_\alpha \int_0^{T_k\wedge T_{\tilde{k}}}
		\Ev\big[\partial_k\partial_{\tilde{k}}\psi^t(\bzetath)\\
	\times \big(\nabla_h \phith[k,\alpha]\cdot\nabla_h \phith[\tilde{k},\alpha],(\rhor[h,\alpha])^-(t\wedge\Ts)\big)_h\big]\m t.
\end{multline*}
Second, to adjust $\nabla_h \phith[k,\alpha]\cdot\nabla_h \phith[\tilde{k},\alpha] \longrightarrow \Ih\big[\nabla\phit[k,\alpha]\cdot\nabla\phit[\tilde{k},\alpha]\big]$, we pay
\begin{multline*}
	\mathrm{Err}_{num,1}
	\coloneqq
	\sum_{k,\tilde{k}=1}^K\sum_{\alpha=1}^\nS \sigma_\alpha \int_0^{T_k\wedge T_{\tilde{k}}}
		\Ev\bigg[\partial_k\partial_{\tilde{k}}\psi^t(\bzetath)\\
	\times\bigg(\big(\nabla_h \phith[k,\alpha]\cdot\nabla_h \phith[\tilde{k},\alpha]
		-\Ih[\nabla\phit[k,\alpha]\cdot\nabla\phit[\tilde{k},\alpha]],\rhor[h,\alpha](t\wedge\Ts)\big)_h\bigg)\bigg]\m t.
\end{multline*}
Third, for the adjustment $\rhobr[\alpha]\longrightarrow \rhobr[h,\alpha]$ we have to pay
\begin{multline*}
	\mathrm{Err}_{num,2}
	\coloneqq
	\sum_{k,\tilde{k}=1}^K\sum_{\alpha=1}^\nS \sigma_\alpha \int_0^{T_k\wedge T_{\tilde{k}}}
		\Ev\big[\partial_k\partial_{\tilde{k}}\psi^t(\bzetath)\big]\\
	\times\bigg(\big(\Ih[\nabla\phit[k,\alpha]\cdot\nabla\phit[\tilde{k},\alpha]],\rhobr[h,\alpha](t\wedge\Ts)\big)_h
		-\big\langle\nabla\phit[k,\alpha]\cdot\nabla\phit[\tilde{k},\alpha],\rhobr[\alpha](t\wedge\Ts)\big\rangle\bigg)\m t.
\end{multline*}
For adjusting $\bzetath \longrightarrow N^{1/2}\big(\Ih[\bphit],(\rhor[h]-\rhobr[h])(t\wedge\bT)\big)_h$ as the argument of $\psit$, we pay
\begin{multline*}
	\mathrm{Err}_{num,3}
	\coloneqq
	\sum_{k,\tilde{k}=1}^K\sum_{\alpha=1}^\nS \sigma_\alpha \int_0^{T_k\wedge T_{\tilde{k}}}
		\Ev\bigg[\bigg(\partial_k\partial_{\tilde{k}}\psi^t\big(N^{1/2}\big(\bphith,(\rhor[h]-\rhobr[h])(t\wedge\bT\wedge\Ts)\big)_h\big)\\
		\shoveright{
			-\partial_k\partial_{\tilde{k}}\psi^t\big(N^{1/2}\big(\Ih[\bphit],(\rhor[h]-\rhobr[h])(t\wedge\bT\wedge\Ts)\big)_h\big)\bigg)
		}\\
	\times\big(\Ih[\nabla\phit[k,\alpha]\cdot\nabla\phit[\tilde{k},\alpha]],(\rhor[h,\alpha]-\rhobr[h,\alpha])(t\wedge\Ts)\big)_h\bigg]\m t.
\end{multline*}
Finally, in order to adjust $\bphi^0_{h}\rightarrow \Ih[\bphi^0]$, we pay
\begin{multline*}
	\mathrm{Err}_{num,4}
	\coloneqq
	\Ev\bigg[\psi^0\bigg(N^{1/2}\big(\Ih[\bphi^0],(\rhor[h]-\rhobr[h])(0)\big)_h\bigg)\\
	-\psi^0\bigg(N^{1/2}\big(\bphi^0_h,(\rhor[h]-\rhobr[h])(0)\big)_h \bigg)\bigg].
\end{multline*}
Setting 
\begin{equation*}
	\mathrm{Err}_{num}
	\coloneqq
	\mathrm{Err}_{num,1}+\mathrm{Err}_{num,2}+\mathrm{Err}_{num,3}+\mathrm{Err}_{num,4},
\end{equation*}
we obtain \eqref{eq_34b equivalent}.

{\bf Step 3: Estimates for $\psi^0$, $\widetilde{\psi}_{k\tilde{k}}^t$, $\bphi^0$, $\bphittil[k\tilde{k}]$.}
Estimates for $\bphi^0$ and $\psi^0$ immediately follow from Lemma \ref{lem_backEvoTest} and Lemma \ref{lem_pre34a equivalent} respectively.
With respect to $\widetilde{\psi}_{k\tilde{k}}^t$ we observe that
\begin{equation*}
	\|\nabla\phit[k,\alpha]\cdot\nabla\phit[\tilde{k},\alpha]\|_{H^{s-1}}
	\leq
	2C(s) \sum_{l=0}^{\floor{(s-1)/2}} \|D^{l+1}\bphit\|_{L^\infty} \|D^{s-l}\bphit\|_{L^2}
	\leq
	C(s) \|\bphit\|_{H^s}^2
\end{equation*}
where we use $s>d+2$ for the Sobolev embedding $H^s\compemb C^{\floor{(s+1)/2}}$.
With the definition of $\widetilde{\psi}_{k\tilde{k}}^t$ and Lemma \ref{lem_RegularityContinuousTestFunctions} we thus obtain \eqref{eq_IS_phiRegEst}.
We obtain \eqref{eq_IS_psiRegEst} for $\widetilde{\psi}_{k\tilde{k}}^t$ by using its definition and \eqref{eq_psitLqpowr}, thus getting
\begin{align*}
	\|\widetilde{\psi}_{k\tilde{k}}^t\|_{\mathcal{L}_{pow,r+1}^{q-2}}
	\leq
	\|\psit\|_{\mathcal{L}_{pow,r}^{q}}
	\leq
	C  N^{\delflex}\left(1+\|\bvphi\|_{H^{s}}^r\right) \|\psi\|_{{\mathcal{L}}_{pow,r}^{q}}.
\end{align*}

{\bf Step 4: Estimates for the negativity and numerical errors.}
Estimate \eqref{eq_IS_ErrstopEst} for $\Errstop$ corresponds to the estimates in the proof of Lemma \ref{lem_pre34b equivalent}.
With respect to \eqref{eq_IS_ErrnegEst} we have $\mathrm{Err}_{neg}=0$ since for $\Ts>0$, all $\alpha=1,\ldots,\nS$
\begin{equation}
	\rhor[h,\alpha](t\wedge\Ts)\geq \rhobr[h,\alpha](t\wedge\Ts) - N^{-\ep} \geq \rho_{\min,h} - N^{-\ep} \geq 0
\end{equation}
due to the definition of $\Ts$ and Remark \ref{rem_rhominh_rhomaxh} for $N$ large enough.

For the rest of this step, to prove the bound on $\mathrm{Err}_{num}$, take an arbitrary $\delflex>0$.  
The following estimates hold for $N$ large enough with the $\delflex$-depending bound corresponding to the respectively used results. 
To simplify the notation, we allow for the specific value of $\delflex$ to change from line to line here.
Note that for $\Ts>0$
\begin{align}
	|\bzetath|
	&\leq
	N^{1/2}\|\bphith\|_{H^{d+2}_h}\|(\rhor[h]-\rhobr[h])(t\wedge\bT\wedge\Ts)\|_{H^{-2\floor{d/2}+2}_h}
	\leq	
	N^{\ep+\delflex}\|\bvphi\|_{H^s}\label{eq_bzetathEst}
\end{align}
due to \eqref{eq_DiscreteTestFuncBound} from Lemma \ref{HsBoundDiscrete}
paired with 
$\|\Ih[\cdot]\|_{H^{d+2}_h}\lesssim \|\cdot\|_{C^{d+2}}\lesssim \|\cdot\|_{H^s}$, 
as well as the definition of $\Ts$.

With respect to $\mathrm{Err}_{num,1}$, if $\Ts>0$ we have
\begin{align*}
	\|\rhor[h](t\wedge\Ts)\|_{L^2_h}
	&\leq 
	\|\rhobr[h]\|_{L^\infty(L^2_h)} + C\|(\rhor[h]-\rhobr[h])(t\wedge\Ts)\|_{L^\infty_h}\\
	&\leq
	C\|\rhobr\|_{L^\infty(C^0)} + \|\rhobr[h]-\Ih[\rhobr]\|_{L^\infty(L^2_h)} + N^{-\ep}
	\leq
	C
\end{align*}
by the definition of $\Ts$, Proposition \ref{prop_MFLconsistency} and the scaling from Assumption \ref{ass_Scaling}.
Thus, with \eqref{eq_psitLqpowr} to bound $\|\psit\|_{\mathcal{L}_{pow,r}^2}$ and \eqref{H1ErrorTestFunctionsProduct} from Lemma \ref{lemma_H1ErrorTestFunctions} we obtain
\begin{align}\label{eq_Errnum1Est}
	|\mathrm{Err}_{num,1}| 
	\leq 
	C(\|\rhobr\|_{L^\infty}) \|\psi\|_{\mathcal{L}_{pow,r}^2} \big(1 + \|\bvphi\|_{H^s}^{2r}\big)\|\bvphi\|_{H^s}^2 N^{r\eps+\delflex}h^{p+1}.
\end{align}
With respect to $\mathrm{Err}_{num,2}$, applying \eqref{eq_MFLconsistency:main estimate} from Proposition \ref{prop_MFLconsistency} as well as the Euler-Maclaurin formula -- see Lemma \ref{lem_EulerMaclaurin} -- we have
\begin{align*}
	&\big|\big(\Ih[\nabla\phit[k,\alpha]\cdot\nabla\phit[\tilde{k},\alpha]],\rhobr[h,\alpha](t)\big)_h
		-\big\langle\nabla\phit[k,\alpha]\cdot\nabla\phit[\tilde{k},\alpha],\rhobr[\alpha](t)\big\rangle\big|\\
	&\quad\leq
	 \big|\big(\Ih[\nabla\phit[k,\alpha]\cdot\nabla\phit[\tilde{k},\alpha]],(\rhobr[h,\alpha]-\Ih[\rhobr[\alpha]])(t)\big)_h \big|\\
	&\qquad\qquad\qquad
	 +  \big|\big(\Ih[\nabla\phit[k,\alpha]\cdot\nabla\phit[\tilde{k},\alpha]],\Ih[\rhobr[\alpha]])(t)\big)_h
	 -\big\langle\nabla\phit[k,\alpha]\cdot\nabla\phit[\tilde{k},\alpha],\rhobr[\alpha](t)\big\rangle \big|\\
	&\quad\leq C\|\bphit\|_{C^2}^2 N^\delflex h^{p+1} + C\|\bphit\|_{C^{p+3}}^2 \|\rhobr(t)\|_{C^{p+2}} h^{p+1}.
\end{align*}
Since \eqref{eq_IS_sbounds} implies $s>p+3+d/2$ and therefore $\|\cdot\|_{C^{p+3}}\lesssim\|\cdot\|_{H^s}$, with \eqref{eq_DiscreteTestFuncBound} it follows that
\begin{equation}\label{eq_Errnum2Est}
	|\mathrm{Err}_{num,2}|
	\leq 
	C(\|\rhobr\|_{L^\infty(C^{p+2})}) \|\psi\|_{\mathcal{L}_{pow,r}^2} \big(1 + \|\bvphi\|_{H^s}^{2r}\big)\|\bvphi\|_{H^s}^2N^{r\ep+\delflex}h^{p+1}.
\end{equation}
Concerning $\mathrm{Err}_{num,3}$, for $\Ts>0$ by the definition of $\Ts$ we have 
\begin{equation}
	|\big(\Ih[\nabla\phit[k,\alpha]\cdot\nabla\phit[\tilde{k},\alpha]],(\rhor[h,\alpha]-\rhobr[h,\alpha])(t\wedge\Ts)\big)_h|
	\leq 
	\|\bphit\|_{C^{d+3}}^2 N^{-1/2+\ep},
\end{equation}
and, using \eqref{OpErrorTestFunctions}, we get 
\begin{align}
	\label{eq_bzetathDiffEst}
	\big|\bzetath-N^{1/2}\big(\Ih[\bphit],(\rhor[h]-\rhobr[h])(t\wedge\bT)\big)_h\big|
	&\leq
	C \|\Ih[\bphit]-\bphith\|_{H^{d+2}_h} N^{\ep}\\
	&\leq
	\|\bvphi\|_{H^s}N^{\delflex}h^{p+1}N^{\ep},
	\notag
\end{align}
where we also used $\|\cdot\|_{C^{d+2+p+3}}\lesssim \|\cdot\|_{H^s}$ since $s>p+3d/2+5$.
In particular, for the connecting segment 
	$I\coloneqq\big[\bzetath, N^{1/2}\big(\Ih[\bphit],(\rhor[h]-\rhobr[h])(t\wedge\bT)\big)_h\big]\subset\R^K$ 
it holds that
	$\sup_{u\in I}|u|\leq C N^{\ep+\delflex}\|\bvphi\|_{H^s}$.
Thus, by mean value theorem and the definition of ${\|\cdot\|_{\mathcal{L}_{pow,r}^q}}$,
\begin{equation}\label{eq_Errnum3Est}
	|\mathrm{Err}_{num,3}|
	\leq
	C \|\psi\|_{\mathcal{L}_{pow,r}^3} \big(1 + \|\bvphi\|_{H^s}^{2r}\big)\|\bvphi\|_{H^s}^3 N^{-1/2+(r+2)\ep+\delflex}h^{p+1}.
\end{equation}
As for $\mathrm{Err}_{num,4}$, including the case $\Ts=0$ in \eqref{eq_bzetathDiffEst} and dealing with it as in the estimate of $\Errstop$ in the proof of Lemma \ref{lem_pre34b equivalent}, via the mean value theorem we obtain
\begin{equation}\label{eq_Errnum4Est}
	|\mathrm{Err}_{num,4}|
	\leq
	C \|\psi\|_{\mathcal{L}_{pow,r}^1} \big(1 + \|\bvphi\|_{H^s}^{2r}\big)\|\bvphi\|_{H^s} N^{(r+1)\ep+\delflex}h^{p+1}.
\end{equation}
Combining \eqref{eq_Errnum1Est}, \eqref{eq_Errnum2Est}, \eqref{eq_Errnum3Est}, and \eqref{eq_Errnum4Est} yields \eqref{eq_IS_ErrnumEst} for $\mathrm{Err}_{num}$.
\end{proof}

\subsection{Including the linearisation errors in the iterative structure}\label{subsec_Errlin}
Proposition \ref{prop_ContAndDiscrMomStructure} leaves us with the task to fit $\mathrm{Err}_{lin,a}$ and $\mathrm{Err}_{lin,b}$ to the iterative structure. 
These are the terms induced by the linearisation compensations $\tilQ[t](\cdot)$ and $\tilQ[h,t](\cdot)$, which stem from linearising the respective interaction terms in the continuous and discrete setting;
\begin{align*}
	\tilQ[t](\phit[k])
	&=\sum_{\alpha,\beta=1}^\nS
	\big\langle\nabla\phit[k,\alpha] \cdot \big(\nabla\V[\alpha\beta]\ast\big(\empmeas[\beta,t]-\rhobr[\beta](t)\big)\big),
	\empmeas[\alpha,t]-\rhobr[\alpha](t)\big\rangle,\\
	\tilQ[h,t](\phith[k])
	&=\sum_{\alpha,\beta=1}^\nS
	\big(\nabla_h \phith[k,\alpha]\cdot (\Ih[\nabla\V[\alpha\beta]]\ast_h(\rhor[h,\beta]-\rhobr[h,\beta])(t)) ,(\rhor[h,\alpha]-\rhobr[h,\alpha])(t)\big)_h.
\end{align*}
We now rewrite $\tilQ[t](\phit[k])$ and $\tilQ[h,t](\phith[k])$ (and thus $\mathrm{Err}_{lin,a}$ and $\mathrm{Err}_{lin,b}$) to conform them to the iterative structure as outlined in Subsection \ref{subsubsec_TCErrlin}.

\begin{proposition}[Iterative structure for the linearisation compensations]\label{prop_Iterative Structure for Errlin}
Under the assumptions of Theorem \ref{thm_iterative structure}, let $\mathrm{Err}_{lin,a}$ and $\mathrm{Err}_{lin,b}$ be defined as in Proposition \ref{prop_ContAndDiscrMomStructure}, \eqref{eq_DefErrlina} and \eqref{eq_DefErrlinb}.
Then for all $t\in(0,T)$, $k\in\lbrace 1,\ldots,K\rbrace$, $m,n\in\Z^d$, $\alpha,\beta\in\lbrace 1,\ldots,\nS\rbrace$ and $\ell\in\lbrace 1,\ldots,d\rbrace$ there exist generalised moment functions
	$\psitlin[k]\in\mathcal{L}_{pow,r+2}^{q-1}(\R^{K+2})$,
coefficients 
	$\hat{F}_{k,\alpha\beta}^\ell[m,n]$
with sets of test functions
	$\bphitlin[k,mn,\alpha\beta,\ell]\in [H^{s}(\domain,\R^{\nS})]^{K+2}$
and test times 
	$\bTlin[k]\in[0,T]^{K+2}$
such that
\begin{multline*}
	\mathrm{Err}_{lin,a} 
	= 
	N^{-1/2}\sum_{k=1}^K\sum_{\alpha,\beta=1}^\nS\sum_{\ell=1}^d\int_0^{T_k}\sum_{n,m\in\Z^d} \hat{F}_{k,\alpha\beta}^\ell[m,n]\\
		\times\Ev\big[\psitlin[k]\big(N^{1/2}\big\langle\bphitlin[k,mn,\alpha\beta,\ell], \empmeas[t\wedge{\bTlin[k]}]-\rhobr(t\wedge\bTlin[k]) \big\rangle\big)\big]\m t
\end{multline*}
and
\begin{multline*}
	\mathrm{Err}_{lin,b}
	= N^{-1/2}\sum_{k=1}^K\sum_{\alpha,\beta=1}^\nS\sum_{\ell=1}^d\int_0^{T_k}\sum_{n,m\in\Z^d} \hat{F}_{k,\alpha\beta}^\ell[m,n]\\
	\shoveright{
		\times\Ev\big[\psitlin[k]\big(N^{1/2}\big(\Ih[\bphitlin[k,mn,\alpha\beta,\ell]], (\rhor[h]-\rhobr[h])(t\wedge\bTlin[k]\wedge\Ts) \big)_h\big)\big]\m t}\\
	+ \mathrm{Err}_{num,lin}.
\end{multline*}
For any $\delflex>0$ there exists $N_0=N_0(\delflex,\ep,r,s,K,\text{data})$, such that for $N>N_0$ the following estimates hold.
The generalised moment functions $\{\psitlin[k]\}_{t,k}$ are subject to \eqref{eq_IS_psiCompEst}, 
the coefficients $\{\hat{F}_{\alpha\beta}[m,n]\}_{\alpha,\beta,m,n}$ to \eqref{eq_IS_FhatEst}.
The test functions $\{\bphitlin[k,mn,\alpha\beta,\ell]\}$ satisfy \eqref{eq_IS_phiCompEst}.
For the error term the estimate \eqref{eq_IS_ErrnumEst} holds.

\end{proposition}

\begin{proof}
{\bf Step 1: Rewriting $\tilQ[t](\phit[k])$.}
We have 
	$\tilQ[t](\phit[k]) = \sum_{\alpha,\beta=1}^\nS I_{\alpha,\beta}$ 
with
\begin{equation*}
	I_{\alpha,\beta} 
	= \int_\domain \int_\domain \nabla\phit[k,\alpha](x) \cdot \nabla\V[\alpha\beta](x-y)
		\m\big(\empmeas[\beta,t]-\rhobr[\beta](t)\big)(y)		
		\m\big(\empmeas[\alpha,t]-\rhobr[\alpha](t)\big)(x).						
\end{equation*}
For $\ell=1,\ldots,d$ we use the $2d$-dimensional Fourier series representation
\begin{equation}\label{eq_Fourier series}
	\partial_\ell\V[\alpha\beta](x-y) = \sum_{n,m\in\Z^d}\hat{F}_{k,\alpha\beta}^\ell[m,n] \cos(m\cdot y)\cos(n\cdot x),
\end{equation} 
with 
\begin{equation*}
	\hat{F}_{k,\alpha\beta}^\ell[m,n] \coloneqq  \frac{1}{(2\pi)^{2d}}\int_{\domain\times\domain}\partial_\ell\V[\alpha\beta](x-y) \cos(m\cdot y)\cos(n\cdot x) \m x\m y.
\end{equation*}
Here we replaced $e^{in\cdot x}$ with $\cos(n\cdot x)$ since $V$ is real-valued.
Since $(x,y)\rightarrow \partial_\ell\V[\alpha\beta](x-y)$ is in $W^{1,\sI-1}(\domain\times\domain)$, we have pointwise convergence of the Fourier series (e.g., see \cite[Theorem 3.3.9 and Corollaries 3.4.9/3.4.10]{grafakos2008classical}) with coefficients estimated by
\begin{equation*}
	|\hat{F}_{k,\alpha\beta}^\ell[m,n]|
	\leq 
	C(\sI)\frac{\|\nabla\V[\alpha\beta]\|_{W^{\sI-1,1}(\domain)}}{(|m|_2+|n|_2)^{\sI-1}},
\end{equation*}
which yields \eqref{eq_IS_FhatEst} with the scaling, i.e. Assumption \ref{ass_Scaling}, since $\|\nabla\V[\alpha\beta]\|_{W^{\sI-1,1}(\domain)}\leq C\rI^{-\sI}$.
Thus, since $\sI>2d+1$, we also have uniform convergence of the Fourier series and therefore, abbreviating $\vtheta_n(x)\coloneqq \cos(n\cdot x)$ for $n\in\Z^d$, we obtain
\begin{equation*}
	I_{\alpha,\beta}
	= \sum_{\ell=1}^d\sum_{n,m\in\Z^d} \hat{F}_{k,\alpha\beta}^\ell[m,n] 
	\big\langle \vtheta_m,\empmeas[\beta,t]-\rhobr[\beta](t)\big\rangle
	\big\langle \partial_\ell\phit[k,\alpha]\vtheta_n,\empmeas[\alpha,t]-\rhobr[\alpha](t)\big\rangle.
\end{equation*}
Adjusting for $N$-prefactors, this leads to 
\begin{multline*}
	N\tilQ[t](\phit[k])
	=\sum_{\alpha,\beta=1}^\nS\sum_{\ell=1}^d\sum_{n,m\in\Z^d} \hat{F}_{k,\alpha\beta}^\ell[m,n] 
	\bigg(N^{1/2}\sum_{\gamma=1}^\nS		\big\langle\delta^{\beta\gamma}\vtheta_m,\empmeas[\gamma,t]-\rhobr[\gamma](t)\big\rangle\bigg)\\
	\times \bigg(N^{1/2}\sum_{\gamma=1}^\nS	\big\langle\delta^{\alpha\gamma}\partial_\ell\phit[k,\alpha]\vtheta_n,\empmeas[\gamma,t]-\rhobr[\gamma](t)\big\rangle\bigg),
\end{multline*}
where $\delta^{\beta\gamma}$ is the Kronecker-Delta.

{\bf Step 2: Rewriting $\tilQ[h,t\wedge\Ts](\phith[k])$.} First, we get an error for switching from $\nabla_h \phith[k,\alpha]$ to $\Ih[\nabla \phit[k,\alpha]]$. 
That is, we have
\begin{equation*}
	\tilQ[h,t\wedge\Ts](\phith[k])
	= \sum_{\alpha,\beta=1}^\nS I_{\alpha,\beta}^h + \overline{\mathrm{Err}}_{num,1}^{k,t}
\end{equation*}
with 
\begin{multline}\label{eq_DefErrbar}
	\overline{\mathrm{Err}}_{num,1}^{k,t}
	\coloneqq 
	\sum_{\alpha,\beta=1}^\nS\Big(\big(\nabla_h \phith[k,\alpha]-\Ih[\nabla \phit[k,\alpha]]\big)
	\cdot (\Ih[\nabla\V[\alpha\beta]]\ast_h(\rhor[h,\beta]-\rhobr[h,\beta])(t\wedge\Ts)),\\
	(\rhor[h,\alpha]-\rhobr[h,\alpha])(t\wedge\Ts)\Big)_h,
\end{multline}
\begin{multline*}
	I_{\alpha,\beta}^h
	\coloneqq 
	h^{2d}\sum_{x,y\in\Ghd}\nabla \phit[k,\alpha](x)\cdot \nabla\V[\alpha\beta](x-y)(\rhor[h,\beta]-\rhobr[h,\beta])(t\wedge\Ts,y)\\
	\times(\rhor[h,\alpha]-\rhobr[h,\alpha])(t\wedge\Ts,x).
\end{multline*}
Plugging in \eqref{eq_Fourier series} we obtain 
\begin{equation*}
	I_{\alpha,\beta}^h
	=
	\sum_{\ell=1}^d\sum_{n,m\in\Z^d} \hat{F}_{k,\alpha\beta}^\ell[m,n]
	\big(\Ih[\vtheta_m],\rhor[h,\beta]-\rhobr[h,\beta]\big)_h
	\big(\Ih[\partial_\ell\phit[k,\alpha]\vtheta_n],\rhor[h,\alpha]-\rhobr[h,\alpha]\big)_h(t\wedge\Ts)
\end{equation*}
and therefore
\begin{multline*}
	N\tilQ[h,t](\phith[k])
	=
	\sum_{\alpha,\beta=1}^\nS\sum_{\ell=1}^d\sum_{n,m\in\Z^d} \hat{F}_{k,\alpha\beta}^\ell[m,n] 
	\bigg(N^{1/2}\sum_{\gamma=1}^\nS \big(\Ih[\delta^{\beta\gamma}\vtheta_m],\rhor[h,\gamma]-\rhobr[h,\gamma]\big)_h\bigg)\\
	\times \bigg(N^{1/2}\sum_{\gamma=1}^\nS \big(\Ih[\delta^{\alpha\gamma}\partial_\ell\phit[k,\alpha]\vtheta_n],\rhor[h,\gamma]-\rhobr[h,\gamma]\big)_h\bigg)(t\wedge\Ts) 
	+ N\overline{\mathrm{Err}}_{num,1}^{k,t}.
\end{multline*}

{\bf Step 3: Defining the iterated functions.}
Defining $\psitlin[k]\colon\R^{K+2}\rightarrow\R$ as 
\begin{equation*}
	\psitlin[k](\f{z})
	\coloneqq
	 -\chi_{t\leq T_k}\partial_k\psi^t(z_1,\ldots,z_K)z_{K+1}z_{K+2},
\end{equation*}
we have via the dominated convergence theorem (see Step 5 for a full justification)
\begin{multline}\label{eq_ErrlinaRephrase}
	\mathrm{Err}_{lin,a} 
	= 
	N^{-1/2}\sum_{k=1}^K\sum_{\alpha,\beta=1}^\nS\sum_{\ell=1}^d\int_0^{T_k}\sum_{n,m\in\Z^d} \hat{F}_{k,\alpha\beta}^\ell[m,n]\\
		\times\Ev\big[\psitlin[k]\big(N^{1/2}\big\langle\bphitlin[k,mn,\alpha\beta,\ell], \empmeas[t\wedge{\bTlin[k]}]-\rhobr(t\wedge\bTlin[k]) \big\rangle\big)\big]\m t
\end{multline}
with 
\begin{equation}\label{eq_defphilinTlin}
	\bphitlin[k,mn,\alpha\beta,\ell](x)
	=
	\begin{pmatrix}
		\big(\phit[1,\gamma](x)\big)_{\gamma=1,\ldots,\nS}\\
		\vdots\\
		\big(\phit[K,\gamma](x)\big)_{\gamma=1,\ldots,\nS}\\
		\big(\delta^{\beta\gamma}\vtheta_m(x)\big)_{\gamma=1,\ldots,\nS}\\
		\big(\delta^{\alpha\gamma}\partial_\ell\phit[k,\alpha](x)\vtheta_n(x)\big)_{\gamma=1,\ldots,\nS}
	\end{pmatrix},
	\qquad
	\bTlin[k]
	=
	\begin{pmatrix}
		T_1\\
		\vdots\\
		T_K\\
		T_k\\
		T_k
	\end{pmatrix}.
\end{equation}
On the discrete side, we analogously have again by dominated convergence (see Step 5)
\begin{multline}\label{eq_ErrlinbRephrase}
	\mathrm{Err}_{lin,b}
	= N^{-1/2}\sum_{k=1}^K\sum_{\alpha,\beta=1}^\nS\sum_{\ell=1}^d\int_0^{T_k}\sum_{n,m\in\Z^d} \hat{F}_{k,\alpha\beta}^\ell[m,n]\\
	\shoveright{
		\times\Ev\big[\psitlin[k]\big(N^{1/2}\big(\Ih[\bphitlin[k,mn,\alpha\beta,\ell]], (\rhor[h]-\rhobr[h])(t\wedge\bTlin[k]\wedge\Ts) \big)_h\big)\big]\m t}\\
	+ \mathrm{Err}_{num,5}
	+ \mathrm{Err}_{num,6}
\end{multline}
where $\mathrm{Err}_{num,5}$ stems from the errors $\overline{\mathrm{Err}}_{num,1}^{k,t}$ defined in \eqref{eq_DefErrbar} of Step 2 and is given by
\begin{multline}\label{eq_DefErrnum5}
	\mathrm{Err}_{num,5}
	\coloneqq
	-\sum_{k=1}^K \int_0^{T_k}\Ev\big[\partial_k\psi^t\big(N^{1/2}\big(\bphith,(\rhor[h]-\rhobr[h])(t\wedge\bT\wedge\Ts)\big)_h\big)\\
	\times N^{1/2}\overline{\mathrm{Err}}_{num,1}^{k,t}\big]\m t,
\end{multline}
while $\mathrm{Err}_{num,6}$ is obtained from replacing $\bzetath \longrightarrow N^{1/2}\big(\Ih[\bphit],(\rhor[h]-\rhobr[h])(t\wedge\bTlin[k]\wedge\Ts)\big)_h$, leading to
\begin{multline}\label{eq_DefErrnum6}
	\mathrm{Err}_{num,6}
	\coloneqq
	-\sum_{k=1}^K \int_0^{T_k}\Ev\bigg[\bigg(\partial_k\psi^t\big(N^{1/2}\big(\bphith,(\rhor[h]-\rhobr[h])(t\wedge\bT\wedge\Ts)\big)_h\big)\\
		-\partial_k\psi^t\big(N^{1/2}\big(\Ih[\bphit],(\rhor[h]-\rhobr[h])(t\wedge\bT\wedge\Ts)\big)_h\big)\bigg)
		N^{1/2}\tilQ[h,t\wedge\Ts](\phith[k])\bigg]\m t.
\end{multline}

{\bf Step 4: Collecting the estimates.}
With respect for the test functions we recall that $\vtheta_n(x)\coloneqq \cos(n\cdot x)$. 
Thus, for $n\in\Z^d$, $\tilde{s}\in\N_0$ we have 
\begin{align}\label{eq_vthetaHs}
	\|\vtheta_n\|_{H^{\tilde{s}}(\domain)} \leq C |n|_\infty^{\tilde{s}}.
\end{align}
This, based on \eqref{eq_defphilinTlin} together with Lemma \ref{lem_RegularityContinuousTestFunctions} and the same argument from Step 3 in the proof of Proposition \ref{prop_ContAndDiscrMomStructure}, yields \eqref{eq_IS_phiCompEst}.
Using the definition of the new generalised moment function and \eqref{eq_psitLqpowr} we obtain
\begin{equation*}
	\|\psitlin[k]\|_{\mathcal{L}_{pow,r+2}^{q-1}} 
	\leq
	\|\psit\|_{\tilde{\mathcal{L}}_{pow,r}^{q}}
	\leq
	C N^{\delflex}\left(1+\|\vphi\|_{H^{s}}^r\right) \|\psi\|_{{\mathcal{L}}_{pow,r}^{q}}.
\end{equation*}
For estimating the error terms, let $\delflex>0$ be an arbitrary, small exponent.
The following estimates hold for $N$ large enough with the $\delflex$-depending bound corresponding to the respectively used results. 
To simplify the notation, we allow for the value of $\delflex$ to change from line to line.
For any $k=1,\ldots,K$, if $\Ts>t>0$ the discrete H\"older inequality (Corollary \ref{cor_DiscrHoelder}) entails
\begin{align*}
	\big|N^{1/2}\overline{\mathrm{Err}}_{num,1}^{k,t}\big|
	&\leq
	CN^{\ep+\delflex} \|\nabla_h\bphith-\Ih[\nabla\bphit]\|_{L^1_h} \|(\rhor[h]-\rhobr[h])(t)\|_{L^\infty_h} \\
	& \leq 
	C\|\bvphi\|_{H^s}N^{\delflex}h^{p+1},
\end{align*}
where for the second inequality we used the definition of $\Ts$, that $\|\cdot\|_{L^1_h}\lesssim\|\cdot\|_{L^2_h}$ and \eqref{Diff_Gradients};
for the first inequality we used that, for $\Ts>t>0$, it holds that
\begin{align*}
	&N^{1/2}\|\Ih[\nabla\V[\alpha\beta]]\ast_h(\rhor[h,\beta]-\rhobr[h,\beta])(t)\|_{L^\infty_h}\\
	&\qquad\leq
	N^{1/2}\|\Ih[\nabla\V]\|_{H^{2\floor{d/2}+2}} \|\rhor[h]-\rhobr[h]\|_{H^{ -2\floor{d/2}-2}}(t)\\
	&\qquad\leq	N^{\ep}\|\nabla\V\|_{C^{d+3}}
	\leq 		N^{\ep}\rI^{-(2d+4)}
	\leq		N^{\ep+\delflex}.
\end{align*}
Here, the last two inequalities are due to the definition of $\V$, given in \eqref{DefPotentialsV_alpha_beta}, and Assumption \ref{ass_Scaling}.
Combining this estimate with $|\bzetath|$-estimate \eqref{eq_bzetathEst} from the proof of Proposition \ref{prop_ContAndDiscrMomStructure} and \eqref{eq_psitLqpowr} from Lemma \ref{lem_pre34b equivalent} yields
\begin{equation}\label{eq_Errnum5Est}
	|\mathrm{Err}_{num,5}| 
	\leq 
	C\|\psi\|_{\mathcal{L}_{pow,r}^1} \big(1 + \|\bvphi\|_{H^s}^{2r}\big)\|\bvphi\|_{H^s}N^{r\ep+\delflex}h^{p+1}.
\end{equation}
With respect to $\mathrm{Err}_{num,6}$ using \eqref{eq_ErrstopProofQEst} for $0<t<\Ts$ from the proof of Lemma \ref{lem_pre34b equivalent}, analogously to the $|\mathrm{Err}_{num,3}|$-estimate \eqref{eq_Errnum3Est} from the proof of Proposition \ref{prop_ContAndDiscrMomStructure} via the mean value theorem we obtain
\begin{equation}\label{eq_Errnum6Est}
	C \|\psi\|_{\mathcal{L}_{pow,r}^2} \big(1 + \|\bvphi\|_{H^s}^{2r}\big)\|\bvphi\|_{H^s}^2 N^{(r+1)\ep+\delflex}h^{p+1}.
\end{equation}
Combining \eqref{eq_Errnum5Est} and \eqref{eq_Errnum6Est} yields \eqref{eq_IS_ErrnumEst} for $\mathrm{Err}_{num,5}+\mathrm{Err}_{num,6}$.

{\bf Step 5: Commuting infinite sum and expected value for \eqref{eq_ErrlinaRephrase} and \eqref{eq_ErrlinbRephrase}.}
We have with \eqref{eq_IS_FhatEst} and \eqref{eq_IS_phiCompEst} (or more precisely \eqref{eq_vthetaHs}) that
\begin{align*}
	&\left|\hat{F}_{k,\alpha\beta}^t[m,n]
	\psitlin[k]\big(N^{1/2}\big\langle\bphitlin[k,mn,\alpha\beta,\ell], \empmeas[t\wedge{\bTlin[k]}]-\rhobr(t\wedge\bTlin[k]) \big\rangle\big)\right|\\
	&\qquad\leq
	CN^{\delflex}(|m|_2+|n|_2)^{-\sI+1}
	\|\psitlin[k]\|_{\mathcal{L}_{pow,r+2}^0}\\
	&\qquad\qquad\qquad\qquad\qquad\times
	N^{(r+2)/2}\|\bphitlin[k,mn,\ell]\|_{H^{\ceil{d/2+1}}}^{r+2}
	\|\empmeas-\rhobr\|_{H^{-\ceil{d/2+1}}}^{r+2}\\
	&\qquad\leq
	CN^{\delflex}(1+\|\bvphi\|_{H^{s}}^{r+2})(|m|_2+|n|_2)^{-\sI+1}(|m|_\infty+|n|_\infty)^{(r+2)\ceil{d/2+1}}\\
	&\qquad\qquad\qquad\qquad\qquad\times
	\|\psitlin[k]\|_{\mathcal{L}_{pow,r+2}^0}
	N^{(r+2)/2}\|\empmeas-\rhobr\|_{H^{-\ceil{d/2+1}}}^{r+2},
\end{align*}
which is a converging sum in $n,m$ since $\sI>2d+1+(r+2)\ceil{d/2+1}$ with \eqref{eq_IS_sIbounds}, while the expected value is controlled due to Proposition \ref{PropErrorEstimateAuxiliarySystem}.
Thus, the application of the dominated convergence theorem in \eqref{eq_ErrlinaRephrase} is justified.
In the discrete case, with $\|\Ih[\cdot]\|_{L^1_h}\lesssim\|\cdot\|_{C^0}$,
\begin{align*}
	&\left|\hat{F}_{k,\alpha\beta}^\ell[m,n]
	\psitlin[k]\big(N^{1/2}\big(\Ih[\bphitlin[k,mn,\alpha\beta,\ell]], (\rhor[h]-\rhobr[h])(t\wedge\bTlin[k]\wedge\Ts) \big)_h\right|\\
	&\qquad\leq
	CN^{\delflex}\|\bvphi\|_{H^{s}}(|m|_2+|n|_2)^{-\sI+1}
	\|\psitlin[k]\|_{\mathcal{L}_{pow,r+2}^0}\\
	&\qquad\qquad\qquad\qquad\qquad\times
	N^{(r+2)/2}\|\bphitlin[k,mn,\ell]\|_{C^0}^{r+2}
	\|(\rhor[h]-\rhobr[h])(t\wedge\bTlin[k]\wedge\Ts)\|_{L^\infty_h}^{r+2}\\
	&\qquad\leq
	CN^{\delflex}(1+\|\bvphi\|_{H^{s}}^{r+3})(|m|_2+|n|_2)^{-s+1}
	N^{(r+2)/2-\epsilon(r+2)}
\end{align*}
due to the definition of the stopping time $\Ts$ in \eqref{defn_stopping} and that $\|\bphitlin[k,mn,\ell]\|_{C^0}\leq 1+\|\bphit\|_{H^s}$.
Since $\sI>2d+1$, the infinite sum other these terms converges and \eqref{eq_ErrlinbRephrase} is valid.
\end{proof}

\section{Proof of Theorem \ref{MainThm}}\label{sec_ProofMT}
Our main result follows from the iterative application of Theorem \ref{thm_iterative structure}, where Proposition \ref{PropErrorEstimateAuxiliarySystem}, Proposition \ref{prop_StopTBound} and the closeness of the initial fluctuations provided by Assumption \ref{ass_InitCond} are used to close the estimate after a sufficient number of iterations.

\begin{proof}
We choose the stopping time $\Ts=\Ts(\ep/2,\ep)$ as defined in \eqref{defn_stopping}.
With this choice, Proposition \ref{prop_StopTBound} directly yields \eqref{NegBoundRig}.
We further assume that in the following $N$ is always large enough for the used auxiliary results to hold. 
The precise dependencies are summarized in the result.
For this proof we allow for the specific value of $\delflex>0$ to change from line to line, since all estimates hold for arbitrary small $\delflex$, given $N$ is large enough (this can be easily adjusted down the line).

{\bf Step 1: The case $j=1$.}
Since some of the following considerations will be needed for the case $j>1$, we proceed under slightly more general assumptions than necessary for just the case $j=1$.
That is, let $\psi\in \mathcal{L}_{pow,r}^q$, $\bvphi\in \big[H^s(\domain,\R^\nS)\big]^K$  and $V\in \big[W^{\sI,1}(\domain)\big]^{\nS\times\nS}$ with $q,r,s,\sI\in\Nz$.
As in \eqref{InductiveQuantity} we set
\begin{multline*}
	\M(\psi,\bvphi,\bT)
	\coloneqq 
	\Ev\Big[\psi\Big(N^{1/2}\big\langle\bvphi,\empmeas[\bT]-\rhobr(\bT)\big\rangle\Big)\Big]\\
	- \Ev\Big[\psi\Big(N^{1/2}\big(\Ih[\bvphi],(\rhor[h]-\rhobr[h])(\bT\wedge\Ts)\big)_h\Big)\Big].
\end{multline*}
Now, if $q\geq 3$, $s>p+3d/2+5$ and $\sI>(d+2)(r+2)+2d+1\rbrace$, Theorem \ref{thm_iterative structure} yields
\begin{align}
\label{eq_MP_iterativeStep}
	\M(\psi,\bvphi,\bT)
	& =
	N^{-1/2}\sum_{k,\tilde{k}}^K\int_0^{T_k\wedge T_{\tilde{k}}}
		\M(\widetilde{\psi}_{k\tilde{k}=1}^t,\bphittil[k\tilde{k}],t\wedge\tilde{\bT}_{k\tilde{k}})
	\\
	&\quad 
	+ N^{-1/2}\sum_{k=1}^K\sum_{\alpha,\beta=1}^\nS\sum_{\ell=1}^d\int_0^{T_k}\sum_{m\in\Z^{2d}}\hat{F}_{k,\alpha\beta}^\ell[m] \nonumber\\
	& \qquad \qquad \qquad \quad \times
			\M(\psitlin[k],\bphitlin[k,m,\alpha\beta,\ell],t\wedge\bTlin[k])	+ \mathrm{Err}(\psi,\bvphi,\bT), \nonumber
\end{align} 
with
\begin{equation}\label{eq_MP_defErr}
	\mathrm{Err}(\psi,\bvphi,\bT) 
	\coloneqq 
	\M(\psi^0,\bphi^0,0)  + \mathrm{Err}_{num} + \mathrm{Err}_{neg} + \Errstop
\end{equation}
for new	
	$\psi^0,\widetilde{\psi}_{k\tilde{k}}^t,\psitlin[k]$, 
coefficients 
	$\hat{F}_{k,\alpha\beta}^\ell[m]$,
test functions
	$\bphi^0,\bphittil[k\tilde{k}],\bphitlin[k,m,\alpha\beta,\ell]$
as well as test times 
	$\tilde{\bT}_{k\tilde{k}}, \bTlin[k]$
and errors 
	$\mathrm{Err}_{num},\mathrm{Err}_{neg},\Errstop$
as described in Theorem \ref{thm_iterative structure}.

With \eqref{InitDataFluctInequality} from Assumption \ref{ass_InitCond}, \eqref{eq_IS_psi0Est}, \eqref{eq_IS_phi0Est} and $\|\cdot\|_{C^{p+1}}\lesssim \|\cdot\|_{H^{(d/2+p+1)+}}$ we have 
\begin{equation*}
	|\M(\psi^0,\bphi^0,0)|
	\leq	
	C\|\psi\|_{\mathcal{L}_{pow,r}^1}\left(1+\|\bvphi\|_{H^s}^{r+1}\right)h^{p+1}
\end{equation*}
and thus with \eqref{eq_IS_ErrnumEst}, \eqref{eq_IS_ErrnegEst} and \eqref{eq_IS_ErrstopEst}
\begin{equation}\label{eq_MP_EstErrPsiPhi}
	|\mathrm{Err}(\psi,\bvphi,\bT)| 
	\leq 
	C\|\psi\|_{\mathcal{L}_{pow,r}^3} \big(1 + \|\bvphi\|_{H^s}^{2r+3}\big)N^{(r+1)\ep+\delflex}\left(h^{p+1}+\exp\left(-CN^{\ep/2}\right)\right).
\end{equation}
Due to Proposition \ref{PropErrorEstimateAuxiliarySystem}, definition \eqref{defn_stopping}, and $\|\Ih(\cdot)\|_{H^{d+2}_h}\lesssim \|\cdot\|_{C^{d+2}}$, we obtain the estimate 
\begin{equation*}
	|\M(\psi,\bvphi,\bT)|
	\leq
	C\|\psi\|_{\mathcal{L}_{pow,r}^0} \|\bvphi\|_{H^{d/2+2}}^r N^{\delflex}
	+C\|\psi\|_{\mathcal{L}_{pow,r}^0} \|\bvphi\|_{C^{d+2}}^r N^{r\ep}.
\end{equation*}
Recalling the definition \eqref{eq_IS_psiRegEst} for $\widetilde{\psi}_{k\tilde{k}}^t$ and \eqref{eq_IS_phiRegEst} for $\bphittil[k\tilde{k}]$, and using $\|\cdot\|_{C^{d+2}}\lesssim \|\cdot\|_{H^s}$, 
\begin{equation}\label{eq_MP_EstMwtPsi}
	|\M(\widetilde{\psi}_{k\tilde{k}}^t, \bphittil[k\tilde{k}],t\wedge\tilde{\bT}_{k\tilde{k}})|
	\leq
	C\|\psi\|_{\mathcal{L}_{pow,r}^2} \big(1 + \|\bvphi\|_{H^s}^{3r+2}\big)N^{(r+1)\ep+\delflex}
\end{equation}
Furthermore, with \eqref{eq_IS_psiCompEst} and  $\|\bphitlin[m]\|_{C^{d+2}}\leq (1+|m|^{d+2})(1+N^{\delflex}\|\bphi\|_{H^s})$ due to \eqref{eq_defphilinTlin}, we get an inequality concerning the terms from the linearization compensation, namely
\begin{multline}\label{eq_MP_EstMwtPsiLin}
	|\M(\psitlin[k], \bphitlin[k,m,\alpha\beta,\ell],t\wedge\bTlin[k])|\\
	\leq
	C\|\psi\|_{\mathcal{L}_{pow,r}^1} \big(1 + \|\bvphi\|_{H^s}^{r}\big)\big(1+|m|_\infty^{(r+2)(d+2)}\big)
		\big(1+\|\bvphi\|_{H^{s}}^{r+2}\big)N^{(r+2)\ep+\delflex}.
\end{multline}
Plugging these back into \eqref{eq_MP_iterativeStep} and applying \eqref{eq_IS_FhatEst} to $\hat{F}_{k,\alpha\beta}^\ell[m]$ we obtain
\begin{multline}\label{eq_MP_Step1FinalEst}
	|\M(\psi,\bvphi,\bT)|\\
	\leq
	C\|\psi\|_{\mathcal{L}_{pow,r}^3} \big(1 + \|\bvphi\|_{H^s}^{3r+3}\big)
	\left(N^{-1/2}+h^{p+1}+\exp\left(-CN^{\ep/2}\right)\right)N^{(r+2)\ep+\delflex},
\end{multline}
where the infinite sum converges since we assumed $\sI> \left(d+2\right)(r+2)+2d+1$.

These considerations, together with the definition of $s(d,p,j)$ and the assumption on $\sI$ as given in Theorem \ref{MainThm} yield \eqref{DistanceRandomVariables} (specifically, the assumptions for $j=1$ correspond to $s=s(d,p,1)=p+3d/2+5$ and $\sI>3(p+3d/2+5)+2d+1$ while the use of the norm $d_{-3}$ corresponds to $q=3$ and $r=0$).


{\bf Step 2: The case $j>1$.}
We iteratively apply Theorem \ref{thm_iterative structure} $j$-times starting with arbitrary $\bvphi\in\big[H^{s_0}(\domain,\R^\nS)\big]^K$, $s_0=s(d,p,j)$, and arbitrary $\psi\in\mathcal{L}_{pow,r_0}^{q_0}$ corresponding to the metric $d_{-(2j+1)}$, i.e. $q_0=2j+1$, $r_0=0$.

This results in a tree structure:
Starting from $(\psi,\bphi)$ the first application of Theorem \ref{thm_iterative structure} spawns three successors: 
The error term $\mathrm{Err}(\psi,\bphi)$ as defined in \eqref{eq_MP_defErr}, the \textit{regular} successor $(\widetilde{\psi}, \f{\widetilde{\phi}})$ and the \textit{compensation} successor $(\psilin, \bphilin[m,\alpha\beta,\ell])_{m\in\Z^{2d}}$ 
-- for simplicity we do not explicitly keep track of indices $t,k,\tilde{k},\alpha,\beta$ and the test times; all estimates are independent of these anyways. 
The edge to the regular successor is labeled with $C_{K,T}N^{-1/2}$, the edge to the compensation successor with $C_{K,T}N^{-1/2}\sum_{m\in\Z^{2d}}\hat{F}[m]$.
The error term stays untouched and forms a leaf node of the tree. 
Then the same process is repeated for the regular and compensation successor.
We iterate until we reach depth $j$.

Each path from the root node to a leaf node corresponds to one distinct group of terms resulting from the iterations. 
If we add up the bounds for these groups, we obtain a bound for $\M(\psi,\bphi)$.
Naturally, each path either ends in an error term or is of length $j$.
Thus, with the considerations for the case $j=1$, the corresponding terms are bound by either $C(\|\psi\|_{C^{2j+1}},\|\bvphi\|_{s_0}) h^{p+1} N^{-\breve{j}/2+\breve{r}\ep+\delflex}$ or $C(\|\psi\|_{C^{2j+1}},\|\bvphi\|_{s_0})N^{-j/2+\tilde{r}\ep+\delflex}$ for suitable $\breve{j},\breve{r},\tilde{r}\in\N_0$.
We only need to show three things:
\begin{itemize}[noitemsep, topsep=0pt, left= 10pt]
	\item		the remaining regularity of the test functions and generalized moment functions at the second to last node of all paths is sufficient to close the estimates as in Step 1,
	\item		all appearing combinations of $\breve{j},\breve{r},\tilde{r}$ fit the scaling in \eqref{FluctPrecise},
	\item		$\sI$ is chosen large enough so that along all paths each infinite sum converges.
\end{itemize}
\noindent
Each path is represented by a sequence 
	$(\psi_{0,\vec{m}_0},\bvphi_{0,\vec{m}_0})_{\vec{m}_0},\ldots,(\psi_{{\tilde{j}},\vec{m}_{\tilde{j}}},\bvphi_{{\tilde{j}},\vec{m}_{\tilde{j}}})_{\vec{m}_{\tilde{j}}},(\mathrm{leaf}_{\vec{m}_{\tilde{j}}})_{\vec{m}_{\tilde{j}}}$ 
with 
	${\tilde{j}}\leq j-1$, 
	$(\psi_{0,\vec{m}_0},\bvphi_{0,\vec{m}_0})_{\vec{m}_0}=(\psi,\bvphi)$, and 
	$\vec{m}_i$ for $0\leq i\leq \tilde{j}$ being the vector of the collected $\Z^{2d}$-indices for all compensation nodes up until $i$.
	
The necessary regularity is given by sequences $(q_i)_{i=1,\ldots,\tilde{j}},(r_i)_{i=1,\ldots,\tilde{j}},(s_i)_{i=1,\ldots,\tilde{j}}\subset\Nz$ such that we control $\|\psi_{i,\vec{m}_i}\|_{\mathcal{L}_{pow,r_{i}}^{q_{i}}}$ and $\|\bvphi_{i,\vec{m}_i}\|_{H^{s_i}}$ in terms of their predecessors. 
With respect to $(q_i)_i$, from Step 1 we know that we need $q_{\tilde{j}}\geq 3$. 
For $i=1,\ldots,\tilde{j}$ we have 
\begin{equation*}
	q_{i}=\left\lbrace\begin{split}
			q_{i-1}-2 \quad&\text{if }(\psi_{i,\vec{m}_i},\bvphi_{i,\vec{m}_i})_{\vec{m}_i}\text{ regular successor, corresponding to \eqref{eq_IS_psiRegEst}},\\
			q_{i-1}-1 \quad&\text{if }(\psi_{i,\vec{m}_i},\bvphi_{i,\vec{m}_i})_{\vec{m}_i}\text{ compensation successor, see \eqref{eq_IS_psiCompEst}}.
		\end{split}\right.
\end{equation*}
Since by definition $q_0=2j+1$, we satisfy $q_{\tilde{j}}\geq 3$.
For $(r_i)_i$, by definition $r_0=0$ and 
\begin{equation*}
	r_{i}=\left\lbrace\begin{split}
			r_{i-1}+1 \quad&\text{if }(\psi_{i,\vec{m}_i},\bvphi_{i,\vec{m}_i})_{\vec{m}_i}\text{ regular successor, corresponding to \eqref{eq_IS_psiRegEst}},\\
			r_{i-1}+2 \quad&\text{if }(\psi_{i,\vec{m}_i},\bvphi_{i,\vec{m}_i})_{\vec{m}_i}\text{ compensation successor, see \eqref{eq_IS_psiCompEst}}.
		\end{split}\right.
\end{equation*}
Concerning $(s_i)_i$, with \eqref{eq_IS_phiRegEst} and \eqref{eq_IS_phiCompEst}, for both types of successors we have $s_i=s_{i-1}-1$.
Since $s(d,p,j)=p+3d/2+4+j$ we guarantee $s_{\tilde{j}}>p+3d/2+5$ as assumed in Step 1: 
In summary, for all $i$ we have 
\begin{equation}\label{eq_MP_risi}
	r_i\leq 2i,
	\qquad
	s_i = p+\frac{3d}{2}+4+j-i.
\end{equation}
It now remains to check that $N^{\ep}$ only appears as given in \eqref{FluctPrecise} and that $\sI$ is chosen large enough to guarantee the convergence of the infinite sums.
Let $\big((\psi_{{i_\ell},\vec{m}_{i_\ell}},\bvphi_{{i_\ell},\vec{m}_{i_\ell}})_{\vec{m}_{i_\ell}}\big)_{\ell=1,\ldots,l}$ with $(i_\ell)_\ell\subset\lbrace 1,\ldots,\tilde{j}\rbrace$ be the compensation nodes on the path.
Let $\G$ be the group of terms corresponding to this path. 
Respectively, \eqref{eq_MP_EstErrPsiPhi}, \eqref{eq_MP_EstMwtPsi}, and \eqref{eq_MP_EstMwtPsiLin} yield an estimate for the leaf term as in \eqref{eq_MP_Step1FinalEst}.
Adding the edge labels we obtain
\begin{multline*}
	|\G|
	\leq 
	C_{K,T}\sum_{\vec{m}_{\tilde{j}}}N^{-\tilde{j}/2}\bigg(\prod_{\ell=1}^l\hat{F}[m_{i_\ell}]\bigg)
	\|\psi_{\tilde{j},\vec{m}_{\tilde{j}}}\|_{\mathcal{L}_{pow,r_{\tilde{j}}}^{q_{\tilde{j}}}}\\
	\times
	\Big(1+\|\bvphi_{\tilde{j}}\|_{H^{s_{\tilde{j}}}}^{3r_{\tilde{j}}+3}\Big)
	\left(N^{-(j-\tilde{j})/2}+h^{p+1}+\exp\left(-CN^{\ep/2}\right)\right)N^{(r_{\tilde{j}}+2)\ep+\delflex}.
\end{multline*}
With respect to the $N^\ep$-scaling, note that the term $N^{-(j-\tilde{j})/2}$ actually only appears if the leaf does not correspond to an error term, in particular if $\tilde{j}=j-1$. 
Plugging in $r_{\tilde{j}}\leq 2\tilde{j}$ and using that $N^{-\tilde{j}/2}N^{(2\tilde{j}+2)\ep}<N^{2\ep}$ since $\ep<1/4$, we obtain the scaling from \eqref{FluctPrecise}.

With respect to the convergence of the sums, we iteratively replace
	$((\psi_{{\tilde{j}},\vec{m}_{\tilde{j}}},\bvphi_{{\tilde{j}},\vec{m}_{\tilde{j}}})_{\vec{m}_{\tilde{j}}})$ 
with their predecessors via \eqref{eq_IS_psiRegEst}, \eqref{eq_IS_psiCompEst}, \eqref{eq_IS_phiRegEst}, and \eqref{eq_IS_phiCompEst}.
Assuming convergence for the previous steps, we obtain a sequence $(\lambda_i)_{i=1,\ldots,\tilde{j}}\subset\N$ such that
\begin{multline}\label{eq_MP_lambdaFramework}
	|\G|
	\leq 
	C_{K,T}\sum_{\vec{m}_{i}}
	\bigg(\prod_{i_\ell\leq i}\hat{F}[m_{i_\ell}]\bigg)
	\|\psi_{i,\vec{m}_i}\|_{\mathcal{L}_{pow,r_i}^{q_i}}\\
	\times
	\big(1+\|\bvphi_{i}\|_{H^{s_i}}^{\lambda_i}\big)
	\left(N^{-(j-\tilde{j})/2}+h^{p+1}+\exp\left(-CN^{\ep/2}\right)\right)N^{-\tilde{j}/2+(2\tilde{j}+2)\ep+\delflex}.
\end{multline}
This sequence is given by $\lambda_{\tilde{j}}\coloneqq 3r_{\tilde{j}}+3$ and from \eqref{eq_IS_psiRegEst}/\eqref{eq_IS_psiCompEst} for $\|\psi_{i,\vec{m}_i}\|_{\mathcal{L}_{pow,r_i}^{q_i}}$ and \eqref{eq_IS_phiRegEst}/\eqref{eq_IS_phiCompEst} for $\|\bvphi_{i}\|$ we can read off that
\begin{equation*}
	\lambda_{i-1}=\left\lbrace\begin{split}
			2\lambda_{i}+r_{i-1} \quad&\text{if }(\psi_{i,\vec{m}_i},\bvphi_{i,\vec{m}_i})_{\vec{m}_i}\text{ regular successor},\\
			\lambda_{i}+r_{i-1} \quad&\text{if }(\psi_{i,\vec{m}_i},\bvphi_{i,\vec{m}_i})_{\vec{m}_i}\text{ compensation successor}.
		\end{split}\right.
\end{equation*}
Plugging in \eqref{eq_IS_FhatEst} for $\hat{F}[m_{i}]$ (assuming that $(\psi_{i,\vec{m}_i},\bvphi_{i,\vec{m}_i})_{\vec{m}_i}$ is a compensation successor) and \eqref{eq_IS_phiCompEst} for $\|\bvphi_{i}\|_{H^{s_i}}$ into \eqref{eq_MP_lambdaFramework}, we see that the convergence of the sum over $m_{i}$ can be guaranteed if for each $\ell=1,\ldots,l$ we have
$
	\sI 
	>  s_{i_\ell}\lambda_{i_\ell} + 2d + 1. 
$
With \eqref{eq_MP_risi} we roughly estimate $\lambda_i\leq  (j-i)2^{j-i}3j$.
Hence, due to our choice of $\sI$ we indeed have for $i\geq 1$ that
\begin{equation*}
	\sI 
	>	2^{j-1}3j^2s(d,p,j) +2d+1
	\geq	\lambda_{i}s_i + 2d+1.
\end{equation*}
Thus, all infinite sums along the path converge.
This completes the proof.
\end{proof}

\section{Quantitative convergence to the mean-field limit}\label{SecQuantConvMeanF}

\begin{proposition}
\label{PropErrorEstimateAuxiliarySystem}
Let $\{X^{\rI,N}_{\alpha,i}\}_{i=1,\ldots,N}^{\alpha=1,\ldots,\nS}$ be a solution to the interacting particle system \eqref{eq_CDPartSys} satisfying Assumption \ref{ass_CDsystem+MFL} with mean field limit $\rhobr$ satisfying Assumption \ref{ass_RegularityMFL}. 
Let $\{\tilde X^{\rI,N}_{\alpha,i}\}$ be a solution to the diffusing particle system with mean-field forces
\begin{equation}\label{eq_CDauxPartSys} 
	\left\{\begin{aligned}
	 	\m \tilde X^{\rI}_{\alpha,i}(t) &=
	 			-\sum_{\beta=1}^\nS (\nabla\V[\alpha\beta]\ast\rhobr[\beta])(\tilde X^{\rI,N}_{\alpha,i}(t),t)\m t,
	 			+ \sqrt{2\sigma_\alpha}\m B_{\alpha,i}(t)\\
		\tilde X^{\rI}_{\alpha,i}(0) &= X^{\rI}_{\alpha,i}(0)\\
	\end{aligned}\right.
\end{equation}
(cfr. \cite[System (3)]{chen2019rigorous}). 
Then, for all $\alpha\in \{1,\ldots,n_S\}$ and all $i\in \{1,\ldots,N\}$, we obtain 
\begin{align}
\label{ErrorEstimateAuxiliarySystem}
	&\sup_{t\in [0,T]} \big|X^{\rI,N}_{\alpha,i}-\tilde X^{\rI,N}_{\alpha,i}\big|(t) \\
	&\quad \leq
	C(T+1) \exp\left(C T r_I^{-(d+2)}\right) 
 	\left( \big\| \empmeas[0] - \rhobr(0) \big\|_{H^{-d/2-2}} +\mathcal{C}N^{-1/2}\right)\nonumber,
\end{align}
where $\mathcal{C}$ is a random variable with Gaussian moments $\mathbb{E}[\exp\tfrac{1}{C}\mathcal{C}^2]\leq 3$.
Furthermore, for any $\delflex > 0$, also using \eqref{ScalingRegimeRadius} from Assumption \ref{ass_Scaling}, there exists $N_0=N_0(\delflex,data)$ such that for all $N\geq N_0$ we have the estimate
\begin{align}
	&\big\| {\mu}_{\alpha,t}^{\rI,N} - \rhobr[\alpha](t) \big\|_{H^{-d/2-2}} \nonumber
	\\&\quad \leq\label{ErrorEstimateMeanFieldLimit_2}
	C(T+1) \exp\left(C T r_I^{-(d+2)}\right)\left(\big\| {\mu}_{0}^{\rI,N} - \rhobr(0) \big\|_{H^{-d/2-2}}	+ \mathcal{C} N^{-1/2}\right) 
	\\& \quad\leq \label{ErrorEstimateMeanFieldLimit}
	\mathcal{C}N^{\delflex T}N^{-1/2}.
\end{align}
\end{proposition}

To prove Proposition \ref{PropErrorEstimateAuxiliarySystem}, we first establish the following bound.
\begin{lemma}
\label{LemmaErrorEstimateAuxiliarySystem}
Suppose that the assumptions of Proposition \ref{PropErrorEstimateAuxiliarySystem} hold.
Additionally, set $\tilde\mu_{\alpha,t}^{\rI,N}:=\frac{1}{N}\sum_{i=1}^N \delta_{\tilde X^{\rI}_{\alpha,i}(t)}$.
Then the estimate
\begin{align*}
	\sup_{t\in [0,T]} \big|X^{\rI,N}_{\alpha,i}-\tilde X^{\rI,N}_{\alpha,i}\big|(t)
	& \leq 
	C(T+1)\exp\left(C T r_I^{-(d+2)}\right) \\
	& \quad\quad \times
	\int_0^T \big\| {\tilde \mu}_{t}^{\rI,N} - \rhobr(t) \big\|_{H^{-d/2-2}} \m t.
\end{align*}
holds for all all $\alpha\in \{1,\ldots,n_S\}$ and all $i\in \{1,\ldots,N\}$.
\end{lemma}
\begin{proof}
By subtracting \eqref{eq_CDPartSys} and \eqref{eq_CDauxPartSys}, we obtain
\begin{align*}
&\m \big(X^{\rI,N}_{\alpha,i}-\tilde X^{\rI,N}_{\alpha,i}\big)(t)
\\&
=-\sum_{\beta=1}^\nS N^{-1} \sum_{j=1}^N \bigg( \nabla\V[\alpha\beta] \left(X^{\rI,N}_{\alpha,i}(t)-X^{\rI,N}_{\beta,j}(t)\right) - (\nabla\V[\alpha\beta] \ast \rhobr[\beta]) (\tilde X^{\rI,N}_{\alpha,i}(t)) \bigg) \m t
\\&
=-\sum_{\beta=1}^\nS N^{-1} \sum_{j=1}^N \bigg( \nabla\V[\alpha\beta] \left(X^{\rI,N}_{\alpha,i}(t)-X^{\rI,N}_{\beta,j}(t)\right) - \nabla\V[\alpha\beta] \left(X^{\rI,N}_{\alpha,i}(t)-\tilde X^{\rI,N}_{\beta,j}(t)\right) \bigg) \m t
\\&~~~\,
-\sum_{\beta=1}^\nS N^{-1} \sum_{j=1}^N \bigg( \nabla\V[\alpha\beta] \left(X^{\rI,N}_{\alpha,i}(t)-\tilde X^{\rI,N}_{\beta,j}(t)\right) - \nabla\V[\alpha\beta] \left(\tilde X^{\rI,N}_{\alpha,i}(t)-\tilde X^{\rI,N}_{\beta,j}(t)\right) \bigg) \m t
\\&~~~\,
-\sum_{\beta=1}^\nS \bigg( N^{-1} \sum_{j=1}^N \nabla\V[\alpha\beta] \left(\tilde X^{\rI,N}_{\alpha,i}(t)-\tilde X^{\rI,N}_{\beta,j}(t)\right) - (\nabla\V[\alpha\beta] \ast \rhobr[\beta]) (\tilde X^{\rI,N}_{\alpha,i}(t)) \bigg) \m t
\end{align*}
which yields
\begin{align*}
\big|X^{\rI,N}_{\alpha,i}-\tilde X^{\rI,N}_{\alpha,i}\big|(T) 
& \leq C \left\| \nabla^2 \V \right\|_{L^\infty}  \int_0^T \sum_ {\beta=1}^\nS \frac{1}{N}\sum_{j=1}^N \big|X^{\rI,N}_{\beta,j}-\tilde X^{\rI,N}_{\beta,j}\big|(t) \m t
\\ & \quad 
+C \int_0^T \sum_ {\beta=1}^\nS \left\| \nabla V^\rI \right\|_{H^{d/2+2}} \left\| {\tilde \mu}_{\beta,t}^{\rI,N} - \rhobr[\beta] \right\|_{H^{-d/2-2}} \m t.
\end{align*}
The Gronwall inequality implies an estimate of the desired order for the averaged and time-integrated error $\int_0^T \sum_ {\beta=1}^\nS \frac{1}{N}\sum_{j=1}^N \big|X^{\rI,N}_{\beta,j}-\tilde X^{\rI,N}_{\beta,j}\big|(t) \m t$; plugging this back into the previous inequality, we arrive at our claim using $\|\nabla^2 \V\|_{L^\infty}\leq C \rI^{-(d+2)}$ and absorbing $\rI^{-1}$ factors (coming from the potential $\V$) into the exponential.
\end{proof}

\begin{rem}[Initial distribution of the particles]
We only prove Proposition \ref{PropErrorEstimateAuxiliarySystem} for the case that the particle positions initially satisfy the spectral gap inequality \ref{eq_ass_sepctralGap}.
The proof for the case of i.i.d.\ positions is analogous, replacing the spectral gap inequality with Hoeffdings inequality, and using that $\{\tilde X^{\rI,N}_{\alpha,i}(t)\}_{i=1}^{N}$ stay independent also for $t>0$.
\end{rem}

\begin{proof}[Proof of Proposition~\ref{PropErrorEstimateAuxiliarySystem}]
In view of the bound from Lemma~\ref{LemmaErrorEstimateAuxiliarySystem}, we need to bound the term $\int_0^T \| {\tilde\mu}_{t}^{\rI,N} - \rhobr \|_{H^{-d/2-2}} \m t$. Note that $w_\alpha\coloneqq \mathbb{E}[{\tilde\mu}_{\alpha,t}^{\rI,N}] - \rhobr[\alpha]$ solve the PDEs
\begin{align*}
\partial_t w_\alpha &=
\sigma_\alpha\Delta w_\alpha
+\nabla \cdot \left(\sum_{\beta=1}^\nS w_\alpha \nabla\V[\alpha\beta]\ast\rhobr[\beta]\right).
\end{align*}
Testing this PDE with $w$ with respect to the $H^{-d/2-2}$-inner product, with a standard energy estimate, and using the smoothness of $\V$ and $\rhobr$ (cfr. Assumptions \ref{ass_CDsystem+MFL} and \ref{ass_RegularityMFL}), we get
\begin{align}\label{ExpectationBound}
&\big\| \mathbb{E}[{\tilde \mu}_{\alpha,t}^{\rI,N}] - \rhobr[\alpha] \big\|_{H^{-\frac{d}{2}-2}}
\lesssim \exp\left(C r_I^{-(d+1)} t\right) \big\| \mathbb{E}[{\tilde \mu}_{\alpha,0}^{\rI,N}] - \rhobr[\alpha](0) \big\|_{H^{-\frac{d}{2}-2}}.
\end{align}
With \eqref{eq_ass_rhobz vs Ev} from Assumption \ref{ass_CDsystem+MFL}, it only remains to bound $\| {\tilde\mu}_{\alpha,t}^{\rI,N} -\mathbb{E}[{\bar\mu}_{\alpha,t}^{\rI,N}] \|_{H^{-d/2-2}}$. 
Testing ${\tilde \mu}_{\alpha,t}^{\rI,N}$ with elements of the Fourier basis $\exp(i \xi \cdot x)$ with $\xi\in\Z^d$, we have
\begin{align}\label{defF}
\int_{\domain}{\exp(i \xi\cdot x) \m \tilde \mu}_{\alpha,t}^{\rI,N}(x)
= N^{-1} \sum_{j=1}^N \exp(i \xi \cdot \tilde X^{\rI,N}_{\alpha,j}(t))=:F_t = F_t\left((X^{\rI,N}_{\alpha,j}(0))_j\right).
\end{align}
Note that by \eqref{eq_CDauxPartSys} we have 
$
\partial  X^{\rI,N}_{\alpha,j}(t)/\partial  X^{\rI,N}_{\alpha,j}(0) \leq C \rI^{-(d+1)}.
$
Applying the spectral gap inequality \eqref{eq_ass_sepctralGap} (with the usual extension to general exponents $\geq 2$) we obtain for any $p\geq 1$
\begin{align*}
&\mathbb{E}\bigg[
\bigg|
\int_{\domain} ({\tilde \mu}_{\alpha,t}^{\rI,N} - \mathbb{E}[{\tilde \mu}_{\alpha,t}^{\rI,N}]) \exp(i \xi\cdot x) \,dx \bigg|^{2p} \bigg]^{1/p}
\\&
\quad \leq C p \mathbb{E}\bigg[\bigg(\sum_{j=1}^N \bigg|\frac{\partial F}{\partial X^{\rI,N}_{\alpha,j}(0)}\bigg|^2 \bigg)^p \bigg]^{1/p} \leq C p |\xi|^2 \rI^{-2(d+1)} N^{-1},
\end{align*}
where $F$ is defined in \eqref{defF}, which entails
\begin{align*}
\bigg|
\int_{\domain} ({\tilde \mu}_{\alpha,t}^{\rI,N} - \mathbb{E}[{\tilde \mu}_{\alpha,t}^{\rI,N}]) \exp(i \xi\cdot x) \m x \bigg|
\leq \mathcal{C} |\xi| \rI^{-(d+1)} N^{-1/2}
\end{align*}
for a random variable $\mathcal{C}$ subject to a Gaussian moment bound. We thus deduce
\begin{align}\label{TildeMinus_exp}
&\left\| {\tilde \mu}_{\alpha,t}^{\rI,N} -\mathbb{E}[{\tilde \mu}_{\alpha,t}^{\rI,N}] \right\|_{H^{-d/2-2}}^2
\\&
\quad =\sum_{\xi \in \mathbb{Z}^d} (1+|\xi|^2)^{-d/2-2} \bigg|
\int_{\domain} ({\bar \mu}_{\alpha,t}^{\rI,N} - \mathbb{E}[{\tilde \mu}_{\alpha,t}^{\rI,N}]) \exp(i \xi\cdot x) \m x \bigg|^2
\leq \mathcal{C} \rI^{-2(d+1)} N^{-1}\nonumber.
\end{align}
Plugging this estimate as well as \eqref{ExpectationBound} into Lemma~\ref{LemmaErrorEstimateAuxiliarySystem}, we arrive at \eqref{ErrorEstimateAuxiliarySystem}. Combining \eqref{TildeMinus_exp} with \eqref{ExpectationBound}
and using the fact that the left-hand side of \eqref{ErrorEstimateAuxiliarySystem} controls the $1$-Wasserstein distance between $\empmeas$ and $\tilde \mu^{\rI,N}$ -- which is precisely the $W^{-1,\infty}$ distance -- and a Sobolev embedding, we arrive at \eqref{ErrorEstimateMeanFieldLimit_2}. 
Then \eqref{ErrorEstimateMeanFieldLimit} follows via another application of the spectral gap inequality, and the bound on $\Ev\big[\empmeas[0]\big]-\rhobz$ (cfr. Assumption \ref{ass_CDsystem+MFL}). 
\end{proof}

\section{Exponentially decaying estimate for the probability that $T < \Ts$}\label{SubsecExpSmallProb}
The stopping time $\Ts$ acts as a discrete analogue to the quantitative convergence of the empirical measure to the mean field limit proven in Section \ref{SecQuantConvMeanF}.
That is, it guarantees that also in the discrete case the fluctuations are roughly of order $N^{-1/2}$.
Additionally, together with Remark \ref{rem_rhominh_rhomaxh} it ensures that the solution $\rhor[h]$ of the discretized Dean--Kawasaki equation \eqref{eq_discretizedDK} stays positive.
The stopping time is defined as follows. 
First, for this section set
\begin{equation}\label{ChoiceOfL}
	l\coloneqq \floor{ d/2}+1.
\end{equation}
Given $\rhor[h]$, the discrete mean field limit $\rhobr[h]$, and small parameters $\ep,\delta>0$, we define
\begin{equation}
	\label{defn_stopping}
	\Ts = \Ts(\delta, \epsilon) \coloneqq T_\oslash^{\infty}(\delta,\ep) \wedge T_\oslash^{l}(\delta,\ep) \wedge T,
\end{equation}
where, with discrete $L^p$-norms as fixed in Definition \ref{def_discrLp} and negative Sobolev norms as fixed in Definition \ref{def_HshNeg},
\begin{align*}
	T_\oslash^{\infty}(\delta,\ep)&\coloneqq 
	\begin{cases}
 		&\inf \{ 0<t\leq T:\, \|(\rhor[h] -\rhobr[h])(t)\|_{L^\infty_h} \geq  N^{-\ep} \},\\
 		&\hspace{3.3cm}\textrm{if }  \|(\rhor[h] -\rhobr[h])(0)\|_{L^\infty_h} \leq  N^{-\ep-\delta},\\
		&0,   
		\hspace{2.85cm}\textrm{ otherwise},
	\end{cases}\\
	T_\oslash^{l}(\delta,\ep)&\coloneqq 
	\begin{cases}
 		&\inf \{ 0<t\leq T:\, \|(\rhor[h] -\rhobr[h])(t)\|_{H^{-2\ell}_h} \geq  N^{-1/2+\ep} \},\\
 		&\hspace{3.3cm}\textrm{if }  \|(\rhor[h] -\rhobr[h])(0)\|_{H^{-\ell}_h} \leq  N^{-1/2+\ep-\delta},\\
		&0,   
		\hspace{2.85cm}\textrm{ otherwise}.
	\end{cases}
\end{align*}
However, for the main result and hence in Sections \ref{TheKeyStepSection} and \ref{sec_ProofMT}, for given $\ep>0$ we always choose $\delta = \ep/2$.
In any way, this stopping time is only meaningful if we can bound the probability of the stopping time triggering, that is $\Ts<T$, which results from the corresponding bounds being broken.
Such control of $\Pm(\Ts<T)$ via stretched exponential bounds is established in the following proposition. 

\begin{rem} 
\begin{enumerate}[noitemsep, topsep=0pt, left= 0pt, label=\arabic*.)]
\item	Note that the conditional definition of the stopping time ensures that if $\Ts>0$, the size of the fluctuations at $t=0$ is guaranteed to be smaller than the stopping condition by a factor of $N^{-\delta}$ (in a stronger norm too for $T_\oslash^{l}(\delta,\ep)$).
		If $\rI$ does not scale but remains constant, then replacing $N^{-\delta}$ with a small enough constant would be sufficient.
\item	While the bounds satisfy $N^{-\ep}N^{-1/2+\ep}=N^{-1/2}$, the parameters for the two auxiliary stopping times are synced just for convenience: 
		The arguments below can be adapted for $\Ts\coloneqq T_\oslash^{\infty}(\delta,\ep) \wedge T_\oslash^{l}(\tilde{\delta},\tilde{\ep})$.
		However, we can not hope to get good probability estimates if we choose stricter bounds in the definition of $\Ts$.
		The range for which we get good probability bounds is limited by Assumption \ref{ass_InitCond} for the size of the fluctuations at $t=0$.
		This assumption is justified by the dicussion in Section \ref{AppConstructionInitialData}. 
\item	With our method, we can not obtain probability estimates for the auxiliary stopping times separately due to the nonlinearity.
		In the proof, the nonlinearity appears in form of the linearization compensation $\tilQ[h,t](\phith)$, for which to effectively bound we need an $L^\infty$-bound as well as an order $\backsim N^{-1/2}$ bound.
\end{enumerate}
\end{rem}

\begin{proposition}\label{prop_StopTBound}
Let $\rhor[h]$ solve \eqref{eq_discretizedDK} and $\rhobr[h]$ solve \eqref{eq_discrIntMeanFiLim}.
Suppose Assumption \ref{ass_CDsystem+MFL} (regularity of the potential), Assumption \ref{ass_RegularityMFL} (regularity of the continuous mean field limit),  Assumption \ref{ass_DiscrDefOps} (discrete differential operators), Assumption \ref{ass_InitCond} (discrete initial conditions), and Assumption \ref{ass_Scaling} (scaling regime) hold. 

Let the stopping time $\Ts$ be defined as in \eqref{defn_stopping} with $\ep, \delta \in (0,\delZero/4)$. 
Then there exists $N_0=N_0(\delta,\ep,data)$ such that for all $N>N_0$ there holds
\begin{align}
 \label{eq_TsBound}
 \Pm\left[\Ts  < T\right] &\leq C\exp\big(-CN^{\ep-\delta}\big).
\end{align}
\end{proposition}

\begin{rem}\label{rem_Ts_epInfluence}
In contrast to the definition of $T_\oslash^{\infty}$, larger $\ep$ yields better bounds.
This is since contributions from $T_\oslash^{\infty}$ are given and -- with Assumption \ref{ass_Scaling} and $\ep<\delZero/4$ -- estimated by 
\begin{equation*}
	\exp\left(-CN^{1/2-\ep-\delta}h^{d/2}\right)
	\leq
	\exp\big(-CN^{\delZero/2-\ep-\delta}\big)
	\leq
	\exp\big(-CN^{\ep-\delta}\big).
\end{equation*}
\end{rem}

\begin{rem}[Strategy for the proof of Proposition \ref{prop_StopTBound} in Subsection \ref{subsec_proofTsBound}]\label{rem_proofstrategyTs}
	We will show that, with high probability, for all $0< t\leq\Ts$ we continuously have stricter bounds than guaranteed just by $t\leq \Ts$.
	However, by definition this can only be the case if $\Ts=T$. 
	Hence, as long as $\Ts>0$, with high probability we obtain $\Ts=T$.
	Since Assumption \ref{ass_InitCond} yields $\Ts>0$ with high probability, this proves the proposition 
	(here, `with high probability' means that we obtain stretched exponential bounds as in \eqref{eq_TsBound} for the complement).\\
	More precisely, we will proceed in seven steps as follows:
	\begin{itemize}[noitemsep, topsep=0pt, left= 0pt]
	\item	For a fixed time, $\|\cdot\|_{L^\infty_h}$ and $\|\cdot\|_{H^{-2l}_h}$ can be represented via suitable test functions. 
			Hence, we can calculate them with the evolutional scheme from Section \ref{subsec_BackEvos} (Steps 1 and 2).
	\item	Via moment bounds for tested fluctuations (see Lemma \ref{lem_DiscrMomentBounds} below) we obtain stricter bounds at a fixed time $T_0$ with high probability assuming $T_0\leq\Ts$ (Steps 3 and 4).
	\item	Given a discretization $0\leq t_1<t_2<\ldots\leq T$ of $[0,T]$ with step size $\backsim h^\et$ for arbitrary $\et>0$, we have stricter bounds for all $t_k\leq\Ts$ with high probability (Step 5).
	\item	If $\et$ is large enough, then with high probability we have strict bounds on the difference between the fluctuations at any time $t<\Ts$ and at the closest previous time step (Step 6).
	\item	We now cover all $t\leq\Ts$ and conclude thanks to the bound on $\Pm[\Ts=0]$ (Step 7).
	\end{itemize}
\end{rem}

\subsection{Moment bounds for the discrete fluctuations}
Here,  given a stopping time with the same structure as $\Ts$, we provide a moment bound crucial for the proof of Proposition \ref{prop_StopTBound}.

\begin{lemma}[Moment bounds for the discrete fluctuations]\label{lem_DiscrMomentBounds}
Let $\rhor[h]$ solve \eqref{eq_discretizedDK} and $\rhobr[h]$ solve \eqref{eq_discrIntMeanFiLim}.
Suppose Assumptions \ref{ass_CDsystem+MFL}, \ref{ass_RegularityMFL}, and \ref{ass_DiscrDefOps} hold. 
Let $\rhobr[h](0)=\Ih[\rhobr(0)]$ and with respect to the scaling regime of $\rI$ assume that \eqref{ScalingRegimeRadius} from Assumption \ref{ass_Scaling} holds.
Define the stopping time $\tTs\coloneqq \tTs[\infty]\wedge\tTs[l]\wedge T$ for some $B_\infty,B_l, B_\infty^0, B_l^0\in\R$ as
\begin{align*}
	\tTs[\infty]&\coloneqq 
	\begin{cases}
 		\inf \{ t >0:\, \|(\rhor[h] -\rhobr[h])(t)\|_{L^\infty_h} \geq  B_\infty \},
 		&\textrm{if }  \|(\rhor[h] -\rhobr[h])(0)\|_{L^\infty_h} \leq  B_\infty^0,\\
		 0,   
		&\textrm{ otherwise},
	\end{cases}\\
	\tTs[l]&\coloneqq 
	\begin{cases}
 		\inf \{ t >0:\, \|(\rhor[h] -\rhobr[h])(t)\|_{H^{-2l}_h} \geq  B_l \},
 		&\textrm{if }  \|(\rhor[h] -\rhobr[h])(0)\|_{H^{-l}_h} \leq  B_l^0,\\
		 0,   
		&\textrm{ otherwise}.
	\end{cases}
\end{align*}
Let $\vphi_h\in L^2(\Ghd,\R^{\nS})$ with backwards evolution $\phit[h]\in C^1(L^2(\Ghd,\R^\nS))$ as in \eqref{eq_discrBackEvoTest}.

Then for each $\delflex>0$ there exists $N_0=N_0(data,\delflex)\in\N$ independent of $j$ such that for all $N\geq N_0$ and each even $j\in\N$ there holds
\begin{multline}\label{eq_DiscrMomentBounds}
	\Ev\left[\chi(\tTs>0)\sup_{0\leq t\leq \tTs}\left(\phit[h],(\rhor[h]-\rhobr[h])(t)\right)_h^j\right]\\
	\leq
	\shoveright{
	j^j\Big(\max\Big\{	CN^{-1/2+\delflex}\|\vphi_h\|_{L^2_h},\,
						j^{-1}CN^{\delflex}\min\{\|\vphi_h\|_{L^1_h}B_\infty^0,\|\vphi_h\|_{H^l_h}B_{l}^0\},
	\qquad}\\
						j^{-1}CN^{\delflex}\min\{h^{-1}\|\vphi_h\|_{L^1_h},\|\vphi_h\|_{L^2_h}\}B_\infty B_l						
				\Big\}\Big)^j.	
\end{multline}
\end{lemma}

\begin{proof}
Lemma \ref{lem_discrBackEvoTest} implies that
\begin{equation*}
	\M_t\coloneqq \chi(\tTs>0)\left(\phi^{t\wedge\tTs}_h,(\rhor[h]-\rhobr[h])(t\wedge\tTs)\right)_h - \int_0^{t\wedge\tTs}\tilQ[h,s](\phi^s_h)\m s
\end{equation*}
is a martingale with Dean--Kawasaki-type noise as in \eqref{eq_discrDynamics of fluctuations for discrBackEvoTest}.
Hence, Doob's inequality yields 
\begin{align}\label{eq_DMB_Doob}
	& \Ev\left[\chi(\tTs>0)\sup_{0\leq t\leq \tTs}\left(\left(\phit[h],(\rhor[h]-\rhobr[h])(t)\right)_h- \int_0^{t}\tilQ[h,s](\phi^s_h)\m s\right)^j\right]\nonumber\\
	& \leq \left(\frac{j}{j-1}\right)^j\Ev\left[\chi(\tTs>0)\left(\left(\phi^{\tTs}_h,(\rhor[h]-\rhobr[h])(\tTs)\right)_h- \int_0^{\tTs}\tilQ[h,s](\phi^s_h)\m s\right)^j\right].
\end{align}
Note that $\left(j/(j-1)\right)^j\leq 4$. 
For the right hand side, up to this constant,  via the It\^o rule with \eqref{eq_discrDynamics of fluctuations for discrBackEvoTest} we obtain for large enough $N$ that
\begin{align*}
	\Ev\left[(\M_{\tTs})^j\right]
	&=
	\Ev\left[\chi(\tTs>0)\left(\phi^{0}_h,(\rhor[h]-\rhobr[h])(0)\right)_h^j\right]\\
	&\quad
	+\Ev\left[\int_0^{\tTs}j(j-1)\left(\M_t\right)^{j-2}
				\sum_{\alpha=1}^{\nS}N^{-1}\sigma_{\alpha}\left(|\nabla_h\phit[h,\alpha]|^2,(\rhor[h,\alpha])^+(t)\right)_h\dt
			\right]\\
	&\leq
	\min\left\{\|\phit[h]\|_{L^1_h}B_\infty^0,\|\phit[h]\|_{H^l_h}B_l^0\right\}^j\\
	&\quad
	+j^2N^{-1}|\sigma|\rho_{max,h}\Ev\left[\sup_{0\leq t\leq \tTs}(\M_t)^{j-2}\right]
	\int_0^T\|\nabla_h\phit[h]\|_{L^2_h}^2\dt\\
	&\leq
	\left(CN^{\delflex}\min\left\{\|\vphi_h\|_{L^1_h}B_\infty^0,\|\vphi_h\|_{H^l_h}B_l^0\right\}\right)^j\\
	&\quad
	+\left(CjN^{-1/2+\delflex}\|\vphi_h\|_{L^2_h}\right)^2\Ev\left[\sup_{0\leq t\leq \tTs}(\M_t)^{j}\right]^{(j-2)/j}.
\end{align*}
For the first inequality, we applied H\"older's inequality (Corollary \ref{cor_DiscrHoelder}), used the definitions of  $\|\cdot\|_{H^s_h}$ and the stopping time, and that we can bound $\|(\rhor(t\wedge\tTs))^+\|_{L^\infty_h}\leq \rho_{max,h}$ as in Remark \ref{rem_rhominh_rhomaxh}.
The second inequality follows from bounds for the backwards evolution of test functions in Lemma \ref{HsBoundDiscrete} and Lemma \ref{L1_bound_discrete}, as well as an application of Jensen's inequality.

Plugging this inequality back into \eqref{eq_DMB_Doob}, and applying Young's inequality to absorb the moment-term (this is necessary only for $j>2$), we obtain
\begin{multline*}
	\Ev\left[\chi(\tTs>0)\sup_{0\leq t\leq \tTs}\left(\left(\phit[h],(\rhor[h]-\rhobr[h])(t)\right)_h- \int_0^{t}\tilQ[h,s](\phi^s_h)\m s\right)^j\right]\\
	\leq 
	\left(CN^{\delflex}\min\left\{\|\vphi_h\|_{L^1_h}B_\infty^0,\|\vphi_h\|_{H^l_h}B_l^0\right\}\right)^j
	+j^j\left(CN^{-1/2+\delflex}\|\vphi_h\|_{L^2_h}\right)^j.
\end{multline*}
Taking the $j$-th root, applying the reverse triangle inequality on the left, moving the $Q$-term to the other side, taking the $j$-th power, and applying $(a+b)^j\leq 2^ja^j + 2^jb^j$ yields 
\begin{multline*}
	\Ev\left[\chi(\tTs>0)\sup_{0\leq t\leq \tTs}\left(\phit[h],(\rhor[h]-\rhobr[h])(t)\right)_h^j\right]
	\leq
	j^j\left(CN^{-1/2+\delflex}\|\vphi_h\|_{L^2_h}\right)^j\\
	+\left(CN^{\delflex}\min\left\{\|\vphi_h\|_{L^1_h}B_\infty^0,\|\vphi_h\|_{H^l_h}B_l^0\right\}\right)^j
	+2^j\Ev\left[\left(\int_0^{\tTs}\left|\tilQ[h,t](\phi^t_h)\right|\dt\right)^j\right].
\end{multline*}
Hence, to finish the proof it is enough to show that
\begin{equation}\label{eq_DMB_Q-bound}
	\mathbb{I}_Q\coloneqq
	\int_0^{\tTs}\left|\tilQ[h,t](\phi^t_h)\right|\dt 
	\leq
	CN^{\delflex}\min\{\|\vphi_h\|_{L^2_h},h^{-1}\|\vphi_h\|_{L^1_h}\}B_\infty B_l.
\end{equation}
This follows immediately from the two following estimates by applying the test function estimates from Lemma \ref{HsBoundDiscrete} and Lemma \ref{L1_bound_discrete}, the definition of the stopping time, and the observation $\|\Ih[\nabla\V\|_{H^{2l}_h}\lesssim \|\V\|_{C^{2l+1}}\lesssim \rI^{-(d+2l+1)}$:
H\"older's inequality yields
\begin{align*}
	\mathbb{I}_Q
	&=
	\int_0^{\tTs}\left|\sum_{\alpha,\beta=1}^\nS
	\big(\nabla_h \phith[\alpha]\cdot (\Ih[\nabla\V[\alpha\beta]]\ast_h(\rhor[h,\beta]-\rhobr[h,\beta])(t)) ,(\rhor[h,\alpha]-\rhobr[h,\alpha])(t)\big)_h\right|\dt\\
	&\leq 
	\begin{cases}
		\left(\int_0^T\|\nabla_h\phit[h]\|_{L^2_h}^2\dt\right)^{\frac{1}{2}}
		\left(\int_0^{\tTs}\|\Ih[\nabla\V\|_{H^{2l}_h}^2 \|\rhor[h]-\rhobr[h]\|_{H^{-2l}_h}^2 \|\rhor[h]-\rhobr[h]\|_{L^\infty_h}^2 \dt\right)^{\frac{1}{2}},
		\vspace{1.5ex}\\
		h^{-1}\int_0^{\tTs}
			\|\phit[h]\|_{L^1_h} \|\Ih[\nabla\V\|_{H^{2l}_h} \|\rhor[h]-\rhobr[h]\|_{H^{-2l}_h} \|\rhor[h]-\rhobr[h]\|_{L^\infty_h} \dt.
	\end{cases}	
\end{align*}
Here, for the first option we used that $\|\cdot\|_{L^2_h}\leq C\|\cdot\|_{L^\infty_h}$ before applying Young's convolution inequality (Corollary \ref{cor_DiscrYoung}).
For the second option, we used that $\|\nabla_h f\|_{L^1_h}\leq Ch^{-1}\|f\|_{L^1_h}$ for any $f\in\LzGhd$.
\end{proof}

\subsection{Proof of the bound for $\Pm(\Ts<T)$}\label{subsec_proofTsBound}
\begin{proof}[Proof of Proposition \ref{prop_StopTBound}]
The proof consists of seven steps, outlined in Remark \ref{rem_proofstrategyTs}. 

\noindent \textbf{Step 1: $\|\cdot\|_{L^{\infty}_h}$ via test functions.}
We define $\vphi^\gamma_{h,x_0}\in[\LzGhd]^\nS$ for any $\gamma\in\{1,\ldots,\nS\}$ and $x_0\in\Ghd$ as $\vphi^\gamma_{h,x_0,\alpha}\coloneqq h^{-d/2}\delta_{\gamma,\alpha}\hbase{x_0}$ where $\delta_{\gamma,\alpha}$ is the Dirac delta and $\hbase{x_0}$ as defined in Section \ref{FiniteDiffDiscr}.
Hence, for any $\eta_h\in[\LzGhd]^\nS$ we have $\big(\vphi^\gamma_{h,x_0},\eta_h\big)_h = \eta_{h,\gamma}(x_0)$ and therefore
\begin{equation}\label{eq_Ts_LinftyTest}
	\|(\rhor[h] -\rhobr[h])(t)\|_{L^\infty_h} 
	= \max_{x_0\in\Ghd}\max_{\gamma=1,\ldots,\nS} \left|\big(\vphi^\gamma_{h,x_0},(\rhor[h] -\rhobr[h])(t)\big)_h\right|.
\end{equation}
Note that by the definition of $\vphi^\gamma_{h,x_0}$ we have
\begin{align}\label{eq_Ts_normsLinftyTest}
	\|\vphi^\gamma_{h,x_0}\|_{L^1_h} &= 1, \qquad
	\|\vphi^\gamma_{h,x_0}\|_{L^2_h} = h^{-d/2}.
\end{align}

\smallskip
\noindent \textbf{Step 2: $\|\cdot\|_{H^{-2l}_h}$ via test functions.}
We define $\vphi^\gamma_{h,m}\in[\LzGhd]^\nS$ for any $\gamma\in\{1,\ldots,\nS\}$ and $m\in\Z^d\cap\left[-\frac{\pi}{h},\frac{\pi}{h}\right)^d$ as $\vphi^\gamma_{h,m,\alpha}\coloneqq\delta_{\gamma,\alpha}\vtheta_m$ where $(\vtheta_m)_m$ is the discrete Fourier basis as introduced in Remark \ref{rem_DefHsh}.
Hence, with Definition \ref{def_HshNeg} we have
\begin{equation}\label{eq_Ts_HlTest}
	\|(\rhor[h] -\rhobr[h])(t)\|_{H^{-2l}_h}^2 
	= \sum_{\gamma=1}^\nS  \sum_{m}\frac{1}{(1+|m|^2)^{2l}}\big(\vphi^\gamma_{h,m},(\rhor[h] -\rhobr[h])(t)\big)_h^2,
\end{equation}
where the sum is over $m\in\Z^d\cap\left[-\frac{\pi}{h},\frac{\pi}{h}\right)^d$.
For $\vphi^\gamma_{h,m}$ we have
\begin{align}\label{eq_Ts_normsHlTest}
	\|\vphi^\gamma_{h,m}\|_{L^1_h} &= (2\pi)^d,
	&
	\|\vphi^\gamma_{h,m}\|_{L^2_h} &= (2\pi)^{d/2},
	&
	\|\vphi^\gamma_{h,m}\|_{H^{l}_h} &\leq (1+|m|^2)^{l/2},
\end{align}
where the last equation is an immediate consequence of Remark \ref{rem_DefHsh}. 

\smallskip
\noindent \textbf{Step 3: Stricter $L^\infty_h$-bounds for a single time point $0<T_0\leq\Ts$.} 
Let $0<\delflex<\delta$.
In this step, we will show that, if $N$ is large enough, for any fixed $T_0\in[0,T]$ we have 
\begin{equation}\label{eq_Ts_stricterLinftyT0}
	\Pm\left[T_0\leq \Ts \text{ and } \|(\rhor[h] -\rhobr[h])(T_0)\|_{L^\infty_h}\geq \frac{N^{-\ep}}{4} \right]
	\leq 
	\nS h^{-d}\exp\left(-CN^{1/2-\ep-\delflex}h^{d/2}\right)
\end{equation}
In order to achieve this, we will apply Lemma \ref{lem_DiscrMomentBounds} to the backwards evolution of the test functions $\vphi^\gamma_{h,x_0}$ from Step 1 with $j=2\left\lfloor\frac{1}{4}C_0^{-1}\frac{N^{-\ep}}{4}N^{1/2-\delflex}\|\vphi^\gamma_{h,x_0}\|_{L^2_h}^{-1}\right\rfloor$ where $C_0$ is the constant in Lemma \ref{lem_DiscrMomentBounds}.
The reasoning for this choice will be explained at the end of this step.
Note that we have $j\in 2\N$ for large enough $N$ since with \eqref{eq_Ts_normsLinftyTest} and \eqref{ScalingRegime} from Assumption \ref{ass_Scaling}
\begin{equation*}
	N^{1/2-\delflex-\ep}\|\vphi^\gamma_{h,x_0}\|_{L^2_h}^{-1} = N^{1/2-\delflex-\ep}h^{d/2} \geq N^{\delZero/2-\delflex-\ep}.
\end{equation*}
We will now show that in our setting the maximum on the right hand side of \eqref{eq_DiscrMomentBounds} is dominated by $C_0N^{-1/2+\delflex}\|\vphi^\gamma_{h,x_0}\|_{L^2_h}$.
Note that for our choice of $j$
\begin{equation}\label{eq_Ts_Step3j}
	N^{-\ep}/16  
	\leq 
	j \big(C_0N^{-1/2+\delflex}\|\vphi^\gamma_{h,x_0}\|_{L^2_h}\big) 
	\leq 
	N^{-\ep}/8.
\end{equation}
Regarding the middle term on the right hand side of \eqref{eq_DiscrMomentBounds}, with \eqref{eq_Ts_normsLinftyTest} and $\delflex<\delta$,  we have 
\begin{equation*}
	C_0N^\delflex \|\vphi^\gamma_{h,x_0}\|_{L^1_h}N^{-\ep-\delta}
	=  C_0 N^{-\ep} N^{-{(\delta-\delflex)}}
	\leq N^{-\ep}/16
\end{equation*}
for large enough $N$.
With respect to the last term, with \eqref{ScalingRegime} from Assumption \ref{ass_Scaling} and large enough $N$
\begin{equation*}
	C_0N^{\delflex}\|\vphi^\gamma_{h,x_0}\|_{L^2_h}N^{-1/2} 
	= C_0h^{-d/2}N^{-1/2}N^{\delflex} 
	\leq C_0 N^{-(\delZero/2-\delflex)} \leq N^{-\ep}/16.
\end{equation*}
Now we apply Lemma \ref{lem_DiscrMomentBounds}.
Denoting with $\phi^{\gamma,t}_{h,x_0}$ the backwards evolution according to \eqref{eq_discrBackEvoTest} with $\phi^{\gamma,t}_{h,x_0}=\vphi^\gamma_{h,x_0}$ for $t\geq T_0$, we obtain
\begin{equation*}
	\Ev\left[\chi(\Ts>0)\sup_{0\leq t\leq \Ts\wedge T_0}\big(\phi^{\gamma,t}_{h,x_0},(\rhor[h]-\rhobr[h])(t)\big)_h^j\right]
	\leq
	j^j\left(	C_0N^{-1/2+\delflex}\|\vphi^{\gamma}_{h,x_0}\|_{L^2_h}\right)^j.
\end{equation*}
With Chebyshev's inequality, then plugging in the choice of $j$ and applying \eqref{eq_Ts_Step3j}, this yields
\begin{align*}
	&\Pm\left[\Ts>0 \text{ and }T_0\leq \Ts \text{ and } \left|\big(\vphi^\gamma_{h,x_0},(\rhor[h] -\rhobr[h])(T_0)\big)_h\right|\geq \frac{N^{-\ep}}{4}\right]\\
	&\qquad\leq 
	\Pm\left[\chi(\Ts>0)\sup_{0\leq t\leq \Ts\wedge T_0}\left|\big(\phi^{\gamma,t}_{h,x_0},(\rhor[h]-\rhobr[h])(t)\big)_h\right|\geq \frac{N^{-\ep}}{4}\right]\\
	&\qquad\leq
	\exp\left(j\left[\log(j)+\log\left(4N^{\ep}C_0N^{-1/2+\delflex}\|\vphi^{\gamma}_{h,x_0}\|_{L^2_h}\right)\right]\right)\\
	&\qquad\leq
	\exp\left(-CN^{1/2-\ep-\delflex}h^{d/2}\right).
\end{align*}
This last inequality dictates the choice of $j$: Up to constants, it is the largest $j$ such that the logarithms combine to a negative constant.
Via \eqref{eq_Ts_LinftyTest} from Step 1, repeating the same argument for all different $\gamma=1,\ldots,\nS$ and the $h^{-d}$ grid points $x_0\in\Ghd$ yields \eqref{eq_Ts_stricterLinftyT0}.

\smallskip
\noindent \textbf{Step 4: Stricter $H^{-2l}_h$-bounds for a single time point $T_0\leq\Ts$.} 
Again, let $0<\delflex<\delta$.
In this step, we show that for any fixed $0<T_0\leq T$, if $N$ is large enough, it holds that 
\begin{equation}\label{eq_Ts_stricterHlT0}
	\Pm\left[T_0\leq \Ts \text{ and }\|(\rhor[h] -\rhobr[h])(T_0)\|_{H^{-2l}_h}\geq \frac{N^{-1/2+\ep}}{4}\right]
	\leq
	\nS\exp\left(-CN^{\ep-\delflex}\right).
\end{equation}
As in Step 3, the main part of the argument is an application of Lemma \ref{lem_DiscrMomentBounds}, here to the functions $\vphi^\gamma_{h,m}$ from Step 2.
We choose $j=2\left\lfloor\frac{1}{4}C_d^{-1}C_0^{-1} N^{\ep-\delflex} \right\rfloor $ independent of $m$, where $C_0$ denotes the constant in Lemma \ref{lem_DiscrMomentBounds} and $C_d$ is a constant, which we will fix later in this step.
Clearly, $j\in 2\N$ for $N$ large enough.

With these choices, the maximum in \eqref{eq_DiscrMomentBounds} is dominated by $C_0(1+|m|^2)^{l/2}N^{-1/2+\delflex}$ for large enough $N$.
Indeed, with our choice of $j$ and \eqref{eq_Ts_HlTest} we have 
\begin{align*}
	j^{-1}C_0 N^{\delflex}\|\vphi^\gamma_{h,m}\|_{L^2_h}N^{-1/2} 
	&\leq
	C_0N^{-1/2+\delflex}\|\vphi^\gamma_{h,m}\|_{L^2_h}
	=
	C_0N^{-1/2+\delflex},\\
	j^{-1}C_0 N^{\delflex}\|\vphi^\gamma_{h,m}\|_{H^l_h}N^{-1/2+\ep-\delta}
	&\leq 
	4C_dC_0^2 (1+|m|^2)^{l/2}N^{-1/2+\delflex} N^{-(\delta-\delflex)}.
\end{align*}
Thus, applying Lemma \ref{lem_DiscrMomentBounds} with $\phi^{\gamma,t}_{h,m}$ denoting the backwards evolution according to \eqref{eq_discrBackEvoTest} with $\phi^{\gamma,t}_{h,m}=\vphi^\gamma_{h,m}$ for $t\geq T_0$, we obtain
\begin{equation*}
	\Ev\left[\chi(\Ts>0)\sup_{0\leq t\leq \Ts\wedge T_0}\big(\phi^{\gamma,t}_{h,m},(\rhor[h]-\rhobr[h])(t)\big)_h^j\right]
	\leq
	j^j\left(C_0(1+|m|^2)^{l/2}N^{-1/2+\delflex}\right)^j.
\end{equation*}
Thus, already \eqref{eq_Ts_HlTest} for $\|\cdot\|_{H^{-2l}_h}$ in mind, after applying the triangle inequality we obtain
\begin{align*}
	&\Ev\left[\chi(\Ts>0)\sup_{0\leq t\leq \Ts\wedge T_0}\bigg(\sum_{m}\frac{1}{(1+|m|^2)^{2l}}\big(\phi^{\gamma,t}_{h,m},(\rhor[h] -\rhobr[h])(t)\big)_h^2\bigg)^{j/2}\right]^{2/j}\\
	&\qquad\leq
	\sum_{m}\frac{C}{(1+|m|^2)^{2l}}\Ev\left[\chi(\Ts>0)\sup_{0\leq t\leq \Ts\wedge T_0}\big(\phi^{\gamma,t}_{h,m},(\rhor[h] -\rhobr[h])(t)\big)_h^{j}\right]^{2/j}\\
	&\qquad\leq
	\sum_{m}\frac{1}{(1+|m|^2)^{l}}\left(jC_0N^{-1/2+\delflex}\right)^2.
\end{align*}
With Chebyshev's inequality, choosing $C_d$ only depending on the dimension and $\nS$ such that $\sum_{m}\frac{\nS}{(1+|m|^2)^{l}}\leq C_d^2$ (recall that $l\coloneqq \floor{d/2}+1$), with the above choice of $j$ we obtain
\begin{multline*}
	\Pm\left[T_0\leq \Ts \text{ and } 
		\sum_{m}\frac{1}{(1+|m|^2)^{2l}}\big(\vphi^\gamma_{h,m},(\rhor[h] -\rhobr[h])(t)\big)_h^2\geq \frac{(N^{-1/2+\ep})^2}{\nS4^2}\right]\\
	\leq \exp\left(j\left[\log(j)+\log\left(4C_dC_0N^{-\ep+\delflex} \right)\right]\right)
	\leq \exp\left(-C N^{\ep-\delflex}\right)
\end{multline*}
as in Step 3. 
Iterating through all species yields \eqref{eq_Ts_stricterHlT0} via \eqref{eq_Ts_HlTest}.

\smallskip
\noindent \textbf{Step 5: Extension to finitely many time points in $[0,\Ts]$ with step size $h^\et$.} 
We discretize the time interval $[0,T]$ with step size $h^\et$ for some large $\et>1$, which we will determine in the next step.
Adding up the bounds in \eqref{eq_Ts_stricterLinftyT0} from Step 3 as well as \eqref{eq_Ts_stricterHlT0} from Step 4 for $T_0=h^\et,2h^\et,\ldots,\floor{T/h^\et}h^\et$, for any $0<\delflex<\delta$ and large enough $N$ we obtain
\begin{align}
	&\Pm\left[ \Ts>0 \text{ and } \exists i\in \mathbb{N}, \, ih^{\et}\leq \Ts\colon 
									\| (\rhor[h] -\rhobr[h] )(ih^{\et}) \|_{L^\infty_h} \geq \frac{N^{-\epsilon}}{4}\right] \label{eq_Ts_stricterLinftyDiscr}\\
	&\qquad \leq  
	C h^{-d-\et}  \exp\left( -CN^{1/2-\ep-\delflex}h^{d/2}\right),\nonumber\\
	&\Pm\left[ \Ts>0 \text{ and } \exists i\in \mathbb{N}, \, ih^{\et}\leq \Ts\colon 
									\| (\rhor[h] -\rhobr[h] )(ih^{\et}) \|_{H^{-2l}_h} \geq \frac{N^{-1/2+\epsilon}}{4}\right] \label{eq_Ts_stricterHlDiscr}\\
	&\qquad \leq  
	C h^{-\et}  \exp\left( -CN^{\ep-\delflex}\right).\nonumber
\end{align}

\smallskip
\noindent \textbf{Step 6: Strict bounds for changes of $\rhor[h]-\rhobr[h]$ in small time increments.}
In this step, we show that, given any $\delflex>0$ and $j>1$, for large enough $N$ and any $i=1,\ldots,\floor{T/h^\et}$ we have 
\begin{align}
\label{eq_Ts_TimeIncrMom}
	&\Ev\left[\chi(\Ts>0) \!\sup_{ t \in [ih^{\et}, (i+1) h^{\et} \wedge \Ts]}  \left\| (\rhor[h]-\rhobr[h])(t)- (\rhor[h]-\rhobr[h])(ih^{\et})\right\|_{\infty}^j\right]^{1/j} \\
	&\qquad \leq  
	jC   h^{\et/2-5d/2-4}N^{-1/2+\ep+\delflex }.\nonumber 
\end{align}
This implies (via Chebyshev's inequality and optimization in $j$ as in Steps 3 and 4) that
\begin{align*}
	&\Pm \left[  \Ts>0 \text{ and } \sup_{t \in [ih^{\et}, (i+1)h^{\et} \wedge \Ts]}
					\left\| (\rhor[h]-\rhobr[h])(t)- (\rhor[h]-\rhobr[h])(ih^{\et})\right\|_{\infty} \geq \frac{N^{-\frac{1}{2}+\epsilon}}{4} \right]\\
	&\qquad \leq 
	\exp \left( - C h^{-\et/2+5d/2+4}N^{-\delflex}\right)\nonumber.
\end{align*}
With \eqref{ScalingRegimeNhReverse} from Assumption \ref{ass_Scaling}, choosing $\delflex=\delZero$,
\begin{equation*}
	h^{-\et/2+5d/2-4}N^{-\delflex} \geq h^{-\et/2+5d/2+5}
\end{equation*}
and hence, choosing $\et=5d+12$ and adding up the bounds for all $i$,
\begin{align}
\label{eq_Ts_TimeIncrStricter}
	&\Pm \bigg[  \Ts>0 \text{ and }\exists i\in \mathbb{N}, \, ih^{\et}\leq \Ts\colon \\
	&\qquad\qquad			 \sup_{t \in [ih^{\et}, (i+1)h^{\et} \wedge \Ts]}
					\left\| (\rhor[h]-\rhobr[h])(t)- (\rhor[h]-\rhobr[h])(ih^{\et})\right\|_{\infty} \geq \frac{N^{-1/2+\epsilon}}{4} \bigg]\nonumber\\
	&\quad \leq 
	Ch^{-\et} \exp \left( - C h^{-1} \right)\nonumber.
\end{align}
The exact estimate to \eqref{eq_Ts_TimeIncrStricter} holds when replacing $L^\infty_h$ with $H^{-2l}_h$, since $\|\cdot\|_{H^{-2l}_h} \leq C\|\cdot\|_{L^\infty_h}$ for a constant $C$ independent of $h$ (for $L^\infty_h$ we can replace $N^{-1/2+\ep}$ with $N^{-\ep}$ since $\ep<1/2$).

\smallskip
We now prove \eqref{eq_Ts_TimeIncrMom} by bounding the increments for each species $\gamma=1,\ldots,\nS$ and $x_0\in\Ghd$. 
With the basis functions $\hbase{x_0}$ as in Section \ref{subsec_spaces}, following the calculations in Section \ref{subsec_BackEvos}, the discrete version of \eqref{eq_FlucEvoSingleSpec} yields
\begin{align}\label{eq_Ts_TimeIncr_Evo}
	&(\rhor[h,\gamma]-\rhobr[h,\alpha])(t,x_0)- (\rhor[h,\alpha]-\rhobr[h,\alpha])(ih^{\et},x_0)\\
	&\quad=
	h^{-d/2}\left(\big(\hbase{x_0},(\rhor[h,\alpha]-\rhobr[h,\alpha])(t)\big)_h - \big(\hbase{x_0},(\rhor[h,\alpha]-\rhobr[h,\alpha])(ih^{\et})\big)_h\right)
	\nonumber\\&\quad=
	h^{-d/2}\int_{ih^{\et}}^t 
		\left(	\sigma_\alpha\Delta_h\hbase{x_0} -\U[h,\alpha](s)\cdot\nabla_h\hbase{x_0}				
						, (\rhor[h,\alpha]-\rhobr[h,\alpha])(s)\right)\m s
	\nonumber\\&\qquad	
	-h^{-d/2}\sum_{\beta=1}^\nS\int_{ih^{\et}}^t  
		\left(	\Ih[\nabla\V[\alpha\beta]]\ast_{h,c}\left(\rhobr[h,\alpha](s)\nabla_h\hbase{x_0} \right)
						, (\rhor[h,\beta]-\rhobr[h,\beta])(s)\right)\m s
	\nonumber\\&\qquad	
	+h^{-d/2}\int_{ih^{\et}}^t  \m \mathcal{W}_h ( (\rhor[h,\alpha])^+,\hbase{x_0})(s),\nonumber
\end{align}
where $\U[h,\alpha]$ is defined as in $\eqref{eq_defUh}$ and
where $\mathcal{W}_h ( (\rhor[h,\alpha])^+,\hbase{x_0})$ is a real-valued martingale corresponding to the Dean--Kawasaki noise with quadratic variation satisfying
\begin{equation*}
	\m \left[\mathcal{W}((\rhor[h,\alpha])^+,\hbase{x_0}) , \mathcal{W}((\rhor[h,\alpha])^+,\hbase{x_0}) \right] (t)
	= 
	2 \sigma_{\alpha} N^{-1}\big(( \rhor[h,\alpha])^+(t), \nabla_h \hbase{x_0} \cdot \nabla_h \hbase{x_0}\big)_h. 
\end{equation*}
Applying the Burkholder-Davis-Gundy inequality (for $j>1$) then gives
\begin{align}
\label{eq_Ts_TimeIncr_martingale}
	&\Ev\left[ \sup_{t \in [ih^{\et}, (i+1) h^{\et} \wedge \Ts]} 
				\Big|\int_{ih^{\et}}^t\m \mathcal{W}_h ( (\rhor[h,\alpha])^+,\hbase{x_0})(s)\Big|^j \right]^{1/j}\\
	&\quad\leq
	Cj  \Ev\left[ \left(\int_{ih^{\et}}^{(i+1) h^{\et}\wedge\Ts} 2 \sigma_{\alpha} N^{-1}\big(( \rhor[h,\alpha])^+(t), \nabla_h \hbase{x_0} \cdot \nabla_h \hbase{x_0}\big)_h\dt \right)^{j/2} \right]^{1/j}\nonumber\\
	&\quad\leq
	Cj h^{\et/2} N^{-1/2}\sqrt{\rho_{max,h}}h^{-1},\nonumber
\end{align}
where we used that $\|(\rhor[h,\alpha])^+(t)\|_{L^\infty_h}\leq \rho_{max,h}$ as in Remark \ref{rem_rhominh_rhomaxh} and $\|\nabla_h \hbase{x_0}\|_{L^2_h}\leq Ch^{-1}$.

With respect to the first term in \eqref{eq_Ts_TimeIncr_Evo}, with $\|\cdot\|_{H^{2l}_h}\leq Ch^{-2l}\|\cdot\|_{L^2_h}$ we have 
\begin{align}\label{eq_Ts_TimeIncr_term1}
	&\left|\int_{ih^{\et}}^t \left(\sigma_\alpha\Delta_h\hbase{x_0}-\U[h,\alpha](s)\cdot\nabla_h\hbase{x_0}	
																		, (\rhor[h,\alpha]-\rhobr[h,\alpha])(s)\right)\m s\right|
	\\&\quad\leq
	Ch^{-2l}\int_{ih^{\et}}^{(i+1) h^{\et}\wedge\Ts}\left(\|\Delta_h\hbase{x_0}\|_{L^2_h}+\|\U[h]\|_{L^\infty_h}\|\nabla_h\hbase{x_0}\|_{L^2_h} \right) \|\rhor[h]-\rhobr[h]\|_{H^{-2l}_h}\m s
	\nonumber\\&\quad\leq
	Ch^{-2l+\et}\left(h^{-2}+\rI^{-d-1}h^{-1}\right)N^{-1/2+\ep}
	\nonumber
\end{align}
where we used the definition of $\Ts$ and that $\|\U\|_{L^{\infty}_h}\leq \|\Ih[\V]\|_{L^\infty}\|\rhobr\|_{L^1_h}\leq C\rI^{-d-1}$ due to Young's inequality, see Corollary \ref{cor_DiscrYoung}.
Note that if $\rhobr(t)\geq 0$, in particular for $t\leq\Ts$ with Remark \ref{rem_rhominh_rhomaxh}, it holds that $\|\rhobr(t)\|_{L^1_h}=\|\rhobr(0)\|_{L^1_h}$. 

For the second term in \eqref{eq_Ts_TimeIncr_Evo}, an analogous argument yields
\begin{align}\label{eq_Ts_TimeIncr_term2}
	&\left|\int_{ih^{\et}}^t  \left(\Ih[\nabla\V[\alpha\beta]]\ast_{h,c}\left(\rhobr[h,\alpha](s)\nabla_h\hbase{x_0} \right), (\rhor[h,\beta]-\rhobr[h,\beta])(s)\right)\m s\right|
	\\&\quad\leq
	Ch^{-2l}\int_{ih^{\et}}^{(i+1) h^{\et}\wedge\Ts}\|\Ih[\V]\|_{L^2_h}\|\rhobr[h]\|_{L^{2}_h}\|\nabla_h\hbase{x_0}\|_{L^{2}_h} \|\rhor[h]-\rhobr[h]\|_{H^{-2l}_h}\m s
	\nonumber\\&\quad\leq
	Ch^{-2l+\et}\rI^{-d-1}\rho_{max,h}h^{-1} N^{-1/2+\ep}.
	\nonumber
\end{align}
Plugging the sum of \eqref{eq_Ts_TimeIncr_Evo} over all $\gamma=1,\ldots,\nS$ and the $h^{-d}$ different $x_0\in\Ghd$ into the left hand side of \eqref{eq_Ts_TimeIncrMom}, applying the triangle inequality and then \eqref{eq_Ts_TimeIncr_term1}, \eqref{eq_Ts_TimeIncr_term2}, and \eqref{eq_Ts_TimeIncr_martingale}, we obtain the right hand side of \eqref{eq_Ts_TimeIncrMom} with $l$ chosen as in \eqref{ChoiceOfL} for large enough $N$ with \eqref{ScalingRegimeRadius} from Assumption \ref{ass_Scaling}.

\smallskip
\noindent \textbf{Step 7: Conclusion from Steps 5 and 6.}
Combining \eqref{eq_Ts_stricterLinftyDiscr} and \eqref{eq_Ts_stricterHlDiscr} for the discrete time points with \eqref{eq_Ts_TimeIncrStricter} for the increments inbetween with respect to $\|\cdot\|_{L^{\infty}_h}$ and $\|\cdot\|_{H^{-2l}_h}$  yields
\begin{align*}
	& \Pm\left[\Ts=T\right] \\
	&\quad\geq 
	\Pm\bigg[\Ts>0 \text{ and }\forall 0\leq t\leq\Ts: \|(\rhor[h] -\rhobr[h])(t)\|_{L^\infty_h} \leq \frac{N^{-\ep}}{2}\\
	&\quad\qquad\qquad\qquad\qquad\qquad\qquad\qquad			\text{and }\|(\rhor[h] -\rhobr[h])(t)\|_{H^{-2\ell}_h} \leq \frac{N^{-1/2+\ep}}{2}\bigg]\\
	&\quad\geq 
	1-\Pm\left[\Ts=0\right]-\Pm\left[\Ts>0\text{ and }\exists 0\leq t\leq\Ts: \|(\rhor[h] -\rhobr[h])(t)\|_{L^\infty_h} > \frac{N^{-\ep}}{2} \right]\\
	&\quad
	-\Pm\left[\Ts>0\text{ and }\exists 0\leq t\leq\Ts: \|(\rhor[h] -\rhobr[h])(t)\|_{H^{-2\ell}_h} > \frac{N^{-1/2+\ep}}{2}\right]\\
	&\quad\geq 
	1 - C\exp\big(\!\!-\! CN^{1/2-\ep-\delta}h^{d/2} \big) - C\exp\big(\!\!-\! CN^{\ep-\delta}\big)\\
	&\quad
	- Ch^{-\et}\!\left(h^{-d}
		\exp\big(\!\!-\! CN^{1/2-\ep-\delflex}h^{d/2}\big)
		+ \exp\big(\!\!-\! CN^{\ep-\delflex}\big)
		+ \exp\big(\!\!-\! C h^{-1} \big)\!\right),
\end{align*}
where we used \eqref{eq_ass_InitialLinftyh} and \eqref{eq_ass_InitialHlh} from Assumption \ref{ass_InitCond} to bound $\Pm\left[\Ts=0\right]$.
After absorbing the $h^{-1}$ prefactors and applying Assumption \ref{ass_Scaling} (e.g., as in Remark \ref{rem_Ts_epInfluence}), we obtain \eqref{eq_TsBound}.
\end{proof}

\begin{appendix}
\section{Regularity of continuous test functions}\label{AppContReg}


\begin{lemma}[Bounds on Sobolev norms of continuous test functions]\label{lem_RegularityContinuousTestFunctions}
Let $\vphi_\alpha\in H^{s}(\domain)$ for all $\alpha\in 1,\dots, n_S$, for some $s\in \mathbb{N}_0$. Assume $\rhobr\in C^{s}$.
Then there exists a unique $H^{s}$-valued strong solution $\phi^t_{\alpha}$ to the backwards evolution equation \eqref{eq_backEvoTest}. Additionally, given an arbitrary $\delflex>0$, there exists $N_0=N_0(\delflex,data)$, such that for all $N>N_0$ we have the bound 
\begin{align}\label{ContinuousTestFuncBound}
\|\phi^t\|_{H^{s}} \leq N^{\delflex (T-t)}\|\vphi\|_{H^{s}}.
\end{align}
\end{lemma}
\begin{proof}
Existence and uniqueness for $H^s$-valued strong solutions to \eqref{eq_backEvoTest} is a straightforward matter. As for proving \eqref{ContinuousTestFuncBound}, we follow a standard energy estimate. For the case $s=0$, upon time reversal $t \leftrightarrow T - t$, we may test \eqref{eq_backEvoTest} with $\phi_{\alpha}^{T-t}$, thus getting 
\begin{align*}
\partial_t \|\phi^{T-t}_{\alpha}\|^2 & = -\sigma_\alpha \|\nabla \phi^{T-t}_{\alpha}\|^2 + (\U[\alpha](T-t) \cdot \nabla \phi^{T-t}_{\alpha}, \phi^{T-t}_{\alpha}) \\
& \quad + \sum_{\beta=1}^\nS \left(\nabla\V[\beta\alpha] \ast_c \big(\rhobr[\beta](T-t)\nabla\phi^{T-t}_{\beta}\big), \phi^{T-t}_{\alpha}\right) \\
& \leq -\sigma_\alpha \|\nabla \phi^{T-t}_{\alpha}\|^2 + \|\U[\alpha](T-t)\|_\infty  \|\nabla \phi^{T-t}_{\alpha}\| \|\phi^{T-t}_{\alpha}\| \\
& \quad + \sum_{\beta=1}^\nS \|\nabla\V[\beta\alpha]\|_{1}  \left\|\big(\rhobr[\beta](T-t)\nabla\phi^{T-t}_{\beta}\big)\right\| \|\phi^{T-t}_{\alpha}\|
\end{align*}
Using the definition of $\U[\alpha]$, we get 
\begin{align*}
\partial_t \|\phi^{T-t}_{\alpha}\|^2 & \leq -\sigma_\alpha \|\nabla \phi^{T-t}_{\alpha}\|^2 + K\left(\sum_{\beta=1}^{\nS}{\|\nabla \phi^{T-t}_{\beta}\|^2}\right)\|\phi^{T-t}_{\alpha}\|^2,
\end{align*}
where we have set
\begin{align*}
K := C(\nS)\left(\sum_{\beta=1}^{\nS}{\|\nabla \V[\alpha\beta]\|_{L^2}\|\rhobr[\beta](T-t)\|_{C^{0}}}\right)
\end{align*}
Summing over all species $\alpha = 1, \dots, \nS$, using Young to absorb the gradient contributions, applying Gronwall Lemma, and relying on the regularity of $\rhobr$ and $V$ (see \eqref{ScalingRegimeRadius}) settles the claim for the case $s=0$. The case $s>0$ is settled analogously using induction over $s$.
\end{proof}

\section{Relevant estimates in the discrete setting}\label{AppDiscrete}

We prove several auxiliary results relating to the discretised setting, namely: error bounds for the difference of continuous mean-field limit $\rhobr[]$ and discrete mean-field limit $\rhobr[h]$ (Subsection \ref{SubsecDiffContDiscreteMFL}); error bounds for the difference of continuous test functions \eqref{eq_backEvoTest} and discrete test functions \eqref{eq_discrBackEvoTest}, and $L^1$-bound for the discrete test functions \eqref{eq_discrBackEvoTest} (Subsection \ref{SubsecRegDiscreteTest}).

In what follows, we will use a discrete version of H\"older's inequality and Young's convolution inequality.
To state these, we need a discrete notion of $L^p$-norms.
\begin{definition}[Discrete $L^p$-norms]\label{def_discrLp}
For all $p\in[1,\infty)$ and $f\in\LzGhd$ we write
\begin{align*}
	\|f\|_{L^p(\Ghd)}=\|f\|_{L^p_h}=\Big(\sum_{x\in\Ghd}h^d|f(x)|^p\Big)^{1/p},\qquad \|f\|_{L^\infty(\Ghd)}=\max_{x\in\Ghd}|f(x)|.
\end{align*}
In particular, this is consistent with the previous definition of $\|\cdot\|_\LzGhd$.
\end{definition}

\begin{corollary}[Discrete H\"older inequality]\label{cor_DiscrHoelder}
Let $q,r\in[1,\infty]$ with $\frac{1}{q}+\frac{1}{r}=1$. 
Then for all $f,g\in\LzGhd$ it holds that
\begin{align*}
	\big(f,g\big)_h \leq \|f\|_{L^q(\Ghd)}\|g\|_{L^r(\Ghd)}.  
\end{align*}
\end{corollary}

With the validity of a H\"older inequality, Young's convolution inequality automatically holds in the same framework.

\begin{corollary}[Discrete Young's convolution inequality]\label{cor_DiscrYoung}
Assume $q,\widetilde{q},r\in[1,\infty]$ satisfy $\frac{1}{q}+\frac{1}{\widetilde{q}}=\frac{1}{r}+1$ and $f,g\in\LzGhd$.
Then
\begin{align*}
	\|f\ast_h g\|_{L^r(\Ghd)}\leq \|f\|_{L^q(\Ghd)}\|g\|_{L^{\widetilde{q}}(\Ghd)}.  
\end{align*}
\end{corollary}

We obtain a quantitative error bound for comparing continuous and discrete convolutions based on the classical Euler-Maclaurin formula for numerical integration.

\begin{lemma}[A multidimensional Euler-Maclaurin formula]\label{lem_EulerMaclaurin}
Let $s\in\Nz$.  Then there exists $C_s>0$, such that for all $f\in C^{2s+2}(\domain)$
\begin{align*}
	\bigg|\int_\domain f(x)\dx - \sum_{x\in\Ghd} h^d f(x)\bigg| \leq  C_s h^{2s+2} \|f\|_{C^{2s+2}(\domain)}.
\end{align*}
\end{lemma}

\begin{proof}
For $d=1$ see, for instance, \cite[Chapter 1]{cheng2007advanced}.
The lemma then follows with Fubini's theorem by induction. 
\end{proof}

These facts imply the following approximation property of the discrete convolution.

\begin{corollary}[Approximation order of convolutions on periodic grids]\label{lem_ApprOrderConv}
Let $s\in\Nz$.  Then there exists $C_s>0$, such that for all $f,g\in C^{2s+2}(\domain)$
\begin{align*}
	\big\| \Ih[f\ast g] - \Ih[f]\ast_h\Ih[g] \big\|_{L^\infty(\Ghd)} 
	\leq
	C_s h^{2s+2} \|f\|_{C^{2s+2}(\domain)}\|g\|_{C^{2s+2}(\domain)}.
\end{align*}
\end{corollary}

\subsection{Error bounds for difference of continuous and discrete mean field limit}\label{SubsecDiffContDiscreteMFL}
In this section we compare the continuous and discretised mean field limits. 
First, we obtain an $L^2$-type bound of order $h^{p+1}$ via an energy estimate. 
Since we need to assume the difference to be small to deal with the quadratic nonlinearity, the estimate only holds up to a finite positive time.
Second, building upon the $L^2$-estimate we will derive higher order estimates difference.

\begin{proposition}[$L^2$-Estimate for the continuous-discrete mean field limit difference]\label{prop_MFLconsistency}
Let Assumption \ref{ass_CDsystem+MFL} hold.
Let $\rhobr$ be the continuous mean-field limit as in \eqref{eq_CD-IMFL} satisfying Assumption \ref{ass_RegularityMFL} with discrete counterpart $\rhobr[h]$ as in \eqref{eq_discrIntMeanFiLim}.
Assume that the initial conditions satisfy Assumption \ref{ass_InitCond} while the parameters $\rI,N,h$ are subject to Assumption \ref{ass_Scaling}.
Set
\begin{align*}
	\Tbh
	\coloneqq \inf\left\lbrace t\in[0,\infty);\, \|\Ih[\rhobr](t)-\rhobr[h](t)\|_{\LzGhd}>1\right\rbrace
	\overset{Assumption\,\ref{ass_InitCond}}{>}0.
\end{align*}
Then for all $\delflex>0$ there exists $N_0=N_0(p,\delflex,\text{data})$ such that for all $t\in[0,\Tbh]$ and $N>N_0$
\begin{align}\label{eq_MFLconsistency:main estimate}
	\|\Ih[\rhobr]-\rhobr[h]\|_{L^2(\Ghd)}^2(t)
	\leq Ch^{2(p+1)} N^{\delflex t}
\end{align}
In particular, $\Tbh\geq T$, since for small $\delflex>0$ and $h>0$ with Assumption \ref{ass_Scaling} it holds that
\begin{align*}
	Ch^{2(p+1)} N^{\delflex t}
	\leq 
	Ch^{2(p+1)} N^{\frac{\delflex}{\delZero}\delZero T}
	\leq
	Ch^{2(p+1)}h^{-\delflex/\delZero}
	\leq 
	1.
\end{align*}
\end{proposition}

\begin{proof}
Testing \eqref{eq_discrIntMeanFiLim}, the definition of the discrete mean field limit, with $\eta_h\in\LzGhd$ and applying integration by parts for $\nabla_h$ we obtain its adjoint/reflected version $\gradhR$
\begin{align*}
	\big(\partial_t\rhobr[h,\alpha],\eta_h\big)_h
	=
	\sigma_\alpha \big(\Delta_h\rhobr[h,\alpha], \eta_h \big)_h
	-\sum_{\beta=1}^\nS \big(\rhobr[h,\alpha](\Ih[\nabla\V[\alpha\beta]]\ast_h\rhobr[h,\beta]), \gradhR\eta_h\big)_h
\end{align*}
for all species $\alpha=1,\ldots,\nS$.
After applying $\Ih$ to \eqref{eq_CD-IMFL} and testing with $\eta_h$, for the continuous mean field limit we have
\begin{align}\label{eq_MFLconsistency:contMFLtested}
	\big(\Ih[\partial_t\rhobr[\alpha]],\eta_h\big)_h
	=
	\sigma_\alpha \big(\Ih[\Delta\rhobr[\alpha]], \eta_h \big)_h
	+\sum_{\beta=1}^\nS \big(\Ih\big[\nabla\cdot\big(\rhobr[\alpha](\nabla\V[\alpha\beta]\ast\rhobr[\beta])\big)\big], \eta_h\big)_h.
\end{align}
In order to make both evolutions comparable, we want to rephrase \eqref{eq_MFLconsistency:contMFLtested} in terms of discrete differential operators.
For the diffusion term, with Assumption \ref{ass_DiscrDefOps} -- the order of the differential operators -- and the discrete H\"older inequality we have 
\begin{align*}
	\left|\big(\Ih[\Delta\rhobr[\alpha]], \eta_h \big)_h - \big(\Delta_h\rhobr[\alpha], \eta_h \big)_h\right|
	\leq
	C\|\eta_h\|_{L^1(\Ghd)}\|\rhobr[\alpha]\|_{C^{p+3}(\domain)}h^{p+1}
\end{align*}
Using integration by parts, we rewrite (for all $l=1,\ldots,d$ and $\alpha,\beta=1,\ldots,\nS$) 
\begin{align*}
	& \left(\Ih\big[\partial_{x_\ell}\big(\rhobr[\alpha](\partial_{x_\ell}\V[\alpha\beta]\ast\rhobr[\beta])\big)\big] ,\eta_h\right)_h \\
	& \quad = 
	\left(\partial_{h,x_\ell}\Ih\big[\rhobr[\alpha](\partial_{x_\ell}\V[\alpha\beta]\ast\rhobr[\beta])\big] ,\eta_h\right)_h 
	+ \widetilde{R}_\nabla\\
	& \quad
	= 
	-\left(\Ih\big[\rhobr[\alpha](\partial_{x_\ell}\V[\alpha\beta]\ast\rhobr[\beta])\big],\dhRx[\ell]\eta_h\right)_h 
	+ \widetilde{R}_\nabla\\
	& \quad =
	-\left(\Ih[\rhobr[\alpha]]\big(\Ih[\partial_{x_\ell}\V[\alpha\beta]]\ast_h\Ih[\rhobr[\beta]]\big) ,\dhRx[\ell]\eta_h\right)_h 
	+ \widetilde{R}_\nabla + \widetilde{R}_\ast,
\end{align*}
where for $\nabla \rightarrow \nabla_h$ with Assumption \ref{ass_DiscrDefOps} and the discrete H\"older inequality we pay
\begin{align*}
	|\widetilde{R}_\nabla| 
	&\leq 
	C\|\eta_h\|_{L^1(\Ghd)}\|\rhobr[\alpha](\partial_{x_\ell}\V[\alpha\beta]\ast\rhobr[\beta])\|_{C^{p+2}(\domain)}h^{p+1}\\
	&\leq 
	C\|\eta_h\|_{L^1(\Ghd)}\|\rhobr\|_{C^{p+2}(\domain)}^2\|\nabla\V[\alpha\beta]\|_{L^1(\domain)}h^{p+1},
\end{align*}
while for $\ast\rightarrow \ast_h$ due to Lemma \ref{lem_ApprOrderConv} with the discrete H\"older inequality we have
\begin{align*}
	|\widetilde{R}_\ast|
	\leq
	C\|D^h_\ell\eta_h\|_{L^1(\Ghd)}\|\rhobr[\alpha]\|_{L^\infty}\|\nabla\V[\alpha\beta]\|_{C^{p+2}} \|\rhobr[\beta]\|_{C^{p+2}}h^{p+1}.
\end{align*}
Thus, combining these two observations we obtain 
\begin{align}\label{eq_MFLconsistency:testedDiff}
	& \big(\partial_t(\Ih[\rhobr[\alpha]]-\rhobr[h,\alpha]),\eta_h\big)_h \nonumber\\
	& \quad =
	\sigma_\alpha \big(\Delta_h(\Ih[\rhobr[\alpha]]-\rhobr[h,\alpha]), \eta_h \big)_h \nonumber\\
	& \quad \quad \quad - \sum_{\beta=1}^\nS \big((\Ih[\rhobr[\alpha]]-\rhobr[h,\alpha])\big(\Ih[\nabla\V[\alpha\beta]]\ast_h\Ih[\rhobr[\beta]]\big), \gradhR\eta_h\big)_h \nonumber
	\\
	& \quad\quad\quad
	-\sum_{\beta=1}^\nS \big(\Ih[\rhobr[\alpha]]\big(\Ih[\nabla\V[\alpha\beta]]\ast_h(\Ih[\rhobr[\beta]]-\rhobr[h,\beta])\big), \gradhR\eta_h\big)_h \nonumber
	\\
	&\quad\quad\quad
	+\sum_{\beta=1}^\nS \big((\Ih[\rhobr[\alpha]]-\rhobr[h,\alpha])\big(\Ih[\nabla\V[\alpha\beta]]\ast_h(\Ih[\rhobr[\beta]]-\rhobr[h,\beta])\big), \gradhR\eta_h\big)_h \nonumber
	\\
	& \quad\quad\quad + R_{\Delta,\nabla}\|\eta_h\|_{L^1(\Ghd)} + R_\ast\|\gradhR\eta_h\|_{L^1(\Ghd)},
\end{align}
with
\begin{align*}
	|R_{\Delta,\nabla}|
	&\leq
	C\left(\|\rhobr\|_{C^{p+3}} + \|\rhobr\|_{C^{p+2}}^2\|\nabla\V\|_{L^1}\right)h^{p+1},
	\\
	|R_\ast|
	&\leq
	C\|\rhobr\|_{L^\infty}\|\rhobr\|_{C^{p+2}}\|\nabla\V\|_{C^{p+2}} h^{p+1}.
\end{align*}
The remainder of the proof will follow the steps of a standard energy estimate. 
Thus, we will choose $\eta_h=\Ih[\rhobr[\alpha]]-\rhobr[h,\alpha]$ as a test function.
Note that with H\"older
\begin{align*}
	\|\Ih[\rhobr[\alpha]]-\rhobr[h,\alpha]\|_{L^1(\Ghd)}
	&\leq |\domain|^{1/2}\|\Ih[\rhobr[\alpha]]-\rhobr[h,\alpha]\|_{L^2(\Ghd)},\\
	\|\gradhR(\Ih[\rhobr[\alpha]]-\rhobr[h,\alpha])\|_{L^1(\Ghd)}
	&\leq |\domain|^{1/2}\|\gradhR(\Ih[\rhobr[\alpha]]-\rhobr[h,\alpha])\|_{L^2(\Ghd)},
\end{align*}
Applying \eqref{IntByParts}, comparing the first-order operators via \eqref{1stDerBound} from Assumption \ref{ass_DiscrDefOps}, and integrating by parts, we obtain
\begin{align*}
	\sigma_\alpha \big(\Delta_h(\Ih[\rhobr[\alpha]]-\rhobr[h,\alpha]), \Ih[\rhobr[\alpha]]-\rhobr[h,\alpha] \big)_h
	& = -\sigma_\alpha \|\nabla_{D,h}(\Ih[\rhobr[\alpha]]-\rhobr[h,\alpha])\|_{\LzGhd}^2\\
	& \leq -\sigma_\alpha C_{D} \|\gradhR(\Ih[\rhobr[\alpha]]-\rhobr[h,\alpha])\|_{\LzGhd}^2.
\end{align*}
Now, applying a combination of the discrete H\"older inequality, Young's discrete convolution inequality and Young's inequality for products to \eqref{eq_MFLconsistency:testedDiff}, we obtain
\begin{align*}
	& \frac{1}{2}\partial_t \|\Ih[\rhobr[\alpha]]-\rhobr[h,\alpha]\|_{L^2(\Ghd)}^2
	+\frac{C_{D}\underline{\sigma}}{5}\|\gradhR(\Ih[\rhobr[\alpha]]-\rhobr[h,\alpha])\|_{\LzGhd}^2\\
	& \quad \leq
	2\frac{5}{C_{D}\underline{\sigma}}\|\Ih[\nabla V]\|_{\LzGhd}^2 \|\Ih[\rhobr]\|_{\LzGhd}^2 \|\Ih[\rhobr]-\rhobr[h]\|_{\LzGhd}^2
	\\
	& \quad \quad + \frac{5}{C_{D}\underline{\sigma}}\|\Ih[\nabla V]\|_{L^2(\Ghd)}^2 \|\Ih[\rhobr]-\rhobr[h]\|_{\LzGhd}^4 
	\\
	& \quad \quad+\frac{5}{C_{D}\underline{\sigma}}|\domain| |R_\ast|^2
	+|\domain||R_{\Delta,\nabla}|^2 + \|\Ih[\rhobr]-\rhobr[h]\|_{\LzGhd}^2,
\end{align*}
where $\underline{\sigma}\coloneqq \min_\alpha \sigma_\alpha$.
For $t\in[0,\Tbh]$, which implies 
\begin{align*}
	\|\Ih[\rhobr]-\rhobr[h]\|_{\LzGhd}^4(t)\leq \|\Ih[\rhobr]-\rhobr[h]\|_{\LzGhd}^2(t),
\end{align*}
summing over all species results in 
\begin{align*}
	& \partial_t \|\Ih[\rhobr]-\rhobr[h]\|_{L^2(\Ghd)}^2
	+\|\gradhR(\Ih[\rhobr]-\rhobr[h])\|_{\LzGhd}^2\\
	& \quad \leq C \left(1+\big(1+\|\Ih[\rhobr]\|_\LzGhd^2\big) \|\Ih[\nabla \V]\|_\LzGhd^2 \right)\|\Ih[\rhobr]-\rhobr[h]\|_{\LzGhd}^2 \\
	& \quad \quad+ |R|^2,
\end{align*}
with 
\begin{align*}
	|R|\leq C \big(1+\|\nabla\V\|_{C^{p+2}}\|\rhobr\|_{C^{p+2}}\big)\|\rhobr\|_{C^{p+3}} h^{p+1}.
\end{align*}
Now, Assumption \ref{ass_InitCond} yields
\begin{align*}
	\|\Ih[\rhobz]-\rhobr[h](0)\|_{L^2(\Ghd)} \leq C h^{p+1}.
\end{align*}
Further using 
$\|\Ih[f]\|_\LzGhd\leq C\|f\|_{C(\domain)}$ and that
$\|\nabla\V\|_{C}\leq C \rI^{-(d+1)}$ as well as
$\|\nabla\V\|_{C^{p+2}}\leq C \rI^{-(d+p+3)}$,
the Gronwall lemma then yields 
\begin{align}\label{eq_MFLconsistency:estimate with rI}
	\|\Ih[\rhobr]-\rhobr[h]\|_{L^2(\Ghd)}^2(t)
	& \leq \left(C h^{2(p+1)} + tC\big(\|\rhobr\|_{L^\infty(C^{p+3})}\big)h^{2(p+1)} \rI^{-2(d+p+3)}\right)\nonumber \\
	& \quad \quad \times\exp\left(t C(\rho_{\max}) \rI^{-2(d+1)}\right).
\end{align}
In order to simplify the structure, we estimate
\begin{align*}
	tC(\|\rhobr\|_{L^\infty(C^{p+3})}) \rI^{-2(d+p+3)}
	\leq
	\exp\left(t C\big(\|\rhobr\|_{L^\infty(C^{p+3})},p\big) \rI^{-2(d+1)}\right).
\end{align*}
With the scaling from Assumption \ref{ass_Scaling} we have, for any chosen $\delflex>0$
\begin{align*}
	\rI^{-2(d+1)}\leq \log(N)^{\frac{2d+2}{2d+4}} 
	\leq 
	\delflex \left(C\big(\|\rhobr\|_{L^\infty(C^{p+3})},p\big)+C(\rho_{\max})\right)^{-1}\log(N)
\end{align*}
for $N$ large enough.
Combining these observations with \eqref{eq_MFLconsistency:estimate with rI} yields \eqref{eq_MFLconsistency:main estimate}.
\end{proof}

\subsubsection{Higher Order bounds for the mean field limit difference}
We will obtain higher order bounds for the mean field limit difference via induction. 
Opposed to Proposition \ref{prop_MFLconsistency} in this proof we do not need additional smallness assumptions to deal with the quadratic non-linearity because -- when taking derivatives -- at most one factor is of the highest order.
The other one is controlled by the induction assumption.

The first order one-sided finite differences given by
\begin{equation}\label{eq_MFLhO_defPartialh1}
	\dOne{e_\ell}f(x) = \frac{f(x+he_\ell)-f(x)}{h}
\end{equation}
for $\ell=1,\ldots,d$, satisfy the product rule
\begin{equation}\label{eq_MFLhO_productRule1}
	\dOne{e_\ell}(fg)=(\tau_{he_\ell}f) (\dOne{e_\ell}g) + (\dOne{e_\ell}f) g
\end{equation}
for any $f,g\in\LzGhd$ with the shift operators $\tau_{h\nu}$, $\nu\in\Z^d$, given by $\tau_{h\nu}f(x)=f(x+h\nu)$.

\begin{proposition}[Higher order bound for the discrete mean field limit]\label{prop_MFLhigherOrder}
Let $s\in\N$.
In the setting of Proposition \ref{prop_MFLconsistency} with the additional assumption $\rhobr\in L^\infty(0,T;C^{s+p+3}(\domain))$,
for every $\delflex>0$ there exists $N_0=N_0(\delflex,s,\text{data})$ such that for all $N>N_0$
\begin{equation}\label{eq_MFLhO_estimateO1}
	\|\Ih[\rhobr]-\rhobr[h]\|_{\HGhd[s]}^2(t) 
	\leq 
	C h^{2(p+1)}N^{\delflex t},\qquad \forall t\in[0,T].
\end{equation}
\end{proposition}

\begin{proof}
We use induction over $s\in\Nz$. 
The base case is settled by Proposition \ref{prop_MFLconsistency}.
For the induction step assume that \eqref{eq_MFLhO_estimateO1} holds for $s-1$, $s\geq 1$, and large enough $N$.
Let $\nu\in\Nz^d$ with $|\nu|=s$.
With \eqref{eq_MFLhO_productRule1}, similarly to the proof of Proposition \ref{prop_MFLconsistency}, for $\eta_h\in\LzGhd$, we get 
\begin{align}\label{eq_MFLhO_rhoh-dynamics}
\big(\partial_t(\dOne{\nu}\rhobr[h,\alpha]),\eta_h\big)_h
	& =
	-\sigma_\alpha \big(\nabla_{D,h}\dOne{\nu}\rhobr[h,\alpha], \nabla_{D,h}\eta_h \big)_h \nonumber\\
	& \quad -\sum_{\gamma\leq\nu}\sum_{\beta=1}^\nS b_{\nu,\gamma}\big(\tau_{h\gamma}\dOne{\nu-\gamma}\rhobr[h,\alpha](\Ih[\nabla\V[\alpha\beta]]\ast_h\dOne{\gamma}\rhobr[h,\beta]), \gradhR\eta_h\big)_h,
\end{align}
where $\gradhR$ is the adjoint/reflected version of $\nabla_h$ while $(b_{\nu,\gamma})_\gamma$ are the typical product rule prefactors (hence binomial coefficients), and
\begin{multline}\label{eq_MFLhO_R-dynamics}
	\big(\partial_t\dOne{\nu}\Ih[\rhobr[\alpha]],\eta_h\big)_h
	=
	\shoveright{
	-\sigma_\alpha \big(\nabla_{D,h}\dOne{\nu}\Ih[\rhobr[\alpha]], \nabla_{D,h}\eta_h \big)_h\\
	-\sum_{\gamma\leq\nu}\sum_{\beta=1}^\nS b_{\nu,\gamma} \big(\tau_{h\gamma}\dOne{\nu-\gamma}\Ih[\rhobr[\alpha]](\Ih[\nabla\V[\alpha\beta]]\ast_h\dOne{\gamma}\Ih[\rhobr[\beta]]), \gradhR\eta_h\big)_h
	}\\
	+ R_{\Delta,\nabla}\|\eta_h\|_\LzGhd 
	+ R_*\|\gradhR\eta_h\|_\LzGhd 
\end{multline}
for all species $\alpha=1,\ldots,\nS$ with consistency errors
\begin{align}\label{eq_MFLhO_R-dynamicsEst}
	|R_{\Delta,\nabla}|,|R_*|\leq C\left(1+\|\nabla\V\|_{C^{p+2}}\|\rhobr\|_{C^{s+p+2}}\right)\|\rhobr\|_{C^{s+p+3}}h^{p+1}.
\end{align}
To obtain this bound for $R_{\Delta,\nabla}$ we use that
\begin{equation*}
	\|\dOne{\nu}(\Delta_h-\Delta)\rhobr[\alpha]\|_{L^\infty}
	\leq 
	Ch^{p+1}\|\dOne{\nu}\rhobr[\alpha]\|_{C^{p+3}}
	\leq 
	Ch^{p+1}\|\rhobr[\alpha]\|_{C^{p+s+3}}
\end{equation*}
and the analogous estimate for $\nabla_h$, which hold since $\dOne{\nu}$ is an order $|\nu|$ finite difference operator and with the mean value theorem.
To bound $R_*$ note that with Lemma \ref{lem_ApprOrderConv}
\begin{align*}
	& \|\dOne{\nu}(\rhobr[\alpha](\nabla\V\ast\rhobr[\beta]))-\dOne{\nu}(\rhobr[\alpha](\Ih[\nabla\V[\alpha\beta]]\ast_h\rhobr[\beta]))\|_{L^\infty}\\
	& \quad =
	\Big\|\sum_{\gamma\leq\nu} b_{\nu,\gamma} 
		\tau_{h\gamma}\dOne{\nu-\gamma}\rhobr[\alpha](\nabla\V[\alpha\beta]\ast\dOne{\gamma}\rhobr[\beta]
				-\Ih[\nabla\V[\alpha\beta]]\ast_h\dOne{\gamma}\rhobr[\beta])\Big\|_{L^\infty}\\
	& \quad \leq
	Ch^{p+1}\sum_{\gamma\leq\nu}\|\dOne{\nu-\gamma}\rhobr\|_{L^\infty}\|\nabla\V\|_{C^{p+2}}\|\dOne{\gamma}\rhobr\|_{C^{p+2}}.
\end{align*}

Taking the difference of \eqref{eq_MFLhO_rhoh-dynamics} and \eqref{eq_MFLhO_R-dynamics}, we obtain a structure corresponding to \eqref{eq_MFLconsistency:testedDiff}.
Further proceeding as in the previous proof, we choose $\eta_h =\dOne{\nu}\Ih[\rhobr[\alpha]]-\dOne{\nu}\rhobr[h,\alpha]$, apply a combination of Young's discrete convolution inequality, the discrete H\"older inequality, as well as Young's inequality for products, and then absorb the $\|\gradhR(\Ih[\partial^\nu\rhobr]-\dOne{\nu}\rhobr[h])\|_{L^2}$-terms.
Note that in the interaction terms always at least one of the differences only has derivatives of order at most $s-1$.
Evoking the induction assumption, we thus obtain
\begin{align*}
	&\partial_t\|\dOne{\nu}\Ih[\rhobr[\alpha]]-\dOne{\nu}\rhobr[h,\alpha]\|_{L^2} 
	\\
	& \quad \leq 
	C\|\Ih[\nabla\V]\|_{L^2_h}^2 
	\Big(\sup_{|\gamma|\leq s}\|\dOne{\gamma}\Ih[\rhobr]\|_{L^2}^2 + h^{2(p+1)}N^{\delflex t} \Big)\\
	& \qquad\qquad
	\times
	\Big(\|\dOne{\nu}\Ih[\rhobr]-\dOne{\nu}\rhobr[h])\|_{L^2}^2 + h^{2(p+1)}N^{\delflex t}\Big)\\
	& \quad\quad
	+ \|\dOne{\nu}\Ih[\rhobr]-\dOne{\nu}\rhobr[h]\|_{L^2}^2 + |R_{\Delta,\nabla}|^2+C|R_*|^2.
\end{align*}
Assumption \ref{ass_Scaling} implies $h^{2(p+1)}N^{\delflex T}\leq 1$ for $\delflex\leq\delZero$, while as for the $R_{\Delta,\nabla}$-bound
\begin{equation*}
	\sup_{|\gamma|\leq s}\|\dOne{\gamma}\Ih[\rhobr]\|_{L^2}^2
	\leq 
	C\|\rhobr\|_{L^\infty(0,T;C^s(\domain))}^2.
\end{equation*}
Thus, after summing over all species and replacing the $\V$-norms with the respective $\rI$-scaling, since $\dOne{\nu}(\Ih[\rhobr]-\rhobr[h])(0)=0$ via the Gronwall inequality we obtain
\begin{align*}
	\|\Ih[\partial^\nu\rhobr]-\dOne{\nu}\rhobr[h]\|_\LzGhd^2(t) & \leq
	Ct \big(1+\|\rhobr\|_{L^\infty(C^{s+p+3})}^4\big) h^{2(p+1)}N^{\delflex t}\rI^{-2(d+p+2)}\\
	& \quad \quad \quad \quad \times\exp\left(Ct (1+\|\rhobr\|_{L^\infty(C^{s})}^2)\rI^{-2(d+1)}\right).
\end{align*}
Now we plug in the scaling of $\rI$ and simplify as we did at the end of the proof for Proposition \ref{prop_MFLconsistency}.
Looping over all $|\nu|=s$ leads to \eqref{eq_MFLhO_estimateO1} for $N$ large enough (and a slightly larger $\delflex$ than chosen for the induction assumption), thus finishing the proof by induction.
\end{proof}

\subsection{Regularity of discrete test functions}\label{SubsecRegDiscreteTest}

 \begin{lemma}[Bounds on Sobolev norms of discrete test functions]\label{HsBoundDiscrete}
 Let $\phi_h^t$ be the solution to \eqref{eq_discrBackEvoTest}, and let the assumptions of Proposition \ref{prop_MFLhigherOrder} be satisfied.
 Let $\delflex>0$.
Then, for every $s\in\mathbb{N}_0$, for every $t\in[0,T]$, and for $N$ large enough, we have the bound 
\begin{align}\label{eq_DiscreteTestFuncBound}
\|\phi_h^t\|_{H^{s}_h}^2+\int_t^T \|\nabla_h \phi^{\tau}_h\|_{H^{s}_h}^2\m\tau  \leq N^{\delflex (T-t)}\|\Ih\vphi\|_{H^{s}_h}^2 \leq N^{\delflex (T-t)}\|\vphi\|_{C^s}^2.
\end{align}
 \end{lemma}
 We omit the proof, a straightforward adaption of Lemma \ref{ContinuousTestFuncBound} combined with \eqref{eq_MFLhO_estimateO1}.

\begin{lemma}[High-order bounds on Sobolev norms for difference of continuous and discrete test functions]\label{lemma_H1ErrorTestFunctions} 
Let $\phi_{\alpha}^t$ be the solution to \eqref{eq_backEvoTest} with final datum $\phi^T=\varphi$. 
Let $\phi^t_{h,\alpha}$ be the solution to \eqref{eq_discrBackEvoTest}. 
Furthermore, let all the assumptions in Proposition \ref{prop_MFLconsistency} be satisfied. 
Let $\delflex>0$.
Then the following estimates hold for large enough $N$.
\begin{itemize}
\item {\bfseries Part I: bound for $L^2$-difference (analogue of Lemma \ref{prop_MFLconsistency})}.
If the regularity requirements
\begin{align}\label{RegularityConvolutionL2}
V_{\beta\alpha} \in C^{p+3}(\mathbb{T}^d), \qquad \phi_\alpha\in C^{p+3}(\domain), \qquad \rhobr\in C^{p+3}
\end{align}
are satisfied, we have the bound 
\begin{align}\label{L2ErrorTestFunctions}
\|\phi^t_{h,\alpha} - \Ih\phi_{\alpha}^t\|_h^2 & \leq Ch^{2(p+1)}N^{\delflex (T-t)}\|\varphi\|_{C^{p+3}}.
\end{align}

\item {\bfseries Part II: bound for discrete Sobolev norm of difference (analogue of Proposition \ref{prop_MFLhigherOrder})}. 
If the following stricter regularity requirement
\begin{align}\label{RegularityConvolutionH1}
V_{\beta\alpha} \in C^{s+(p+3)}(\mathbb{T}^d), \qquad \phi_\alpha\in C^{s+(p+3)}(\domain), \qquad \rhobr\in C^{s+(p+3)}
\end{align}
is satisfied, we have the bound 
\begin{align}\label{OpErrorTestFunctions}
\|\phi^t_{h,\alpha} - \Ih\phi_{\alpha}^t\|_{\HGhd[s]}^2(t) & \leq Ch^{2(p+1)}N^{\delflex (T-t)}\|\varphi\|_{C^{s+p+3}}.
\end{align}

\item {\bfseries Part III: bound for gradient difference}. The assumptions of Part II) being satisfied for $s=1$, we have the bound
\begin{align}\label{Diff_Gradients}
\| \nabla_h\phi^t_{h,\alpha} - \Ih \nabla\phi^t_{\alpha}\|_h  & \leq Ch^{2(p+1)}N^{\delflex (T-t)}\|\varphi\|_{C^{p+4}}.
\end{align}

\item {\bfseries Part IV: bound for difference of product of gradients}.
Assuming the validity of the hypotheses in Parts I), II), III), we obtain, as a special case, that
\begin{align}\label{H1ErrorTestFunctionsProduct}
& \|\nabla_h\phi^t_{h,\alpha_1}\cdot \nabla_h\phi^t_{h,\alpha_2} - \Ih[\nabla\phi_{\alpha_1}^t \cdot \nabla\phi^t_{\alpha_2}]\|_h^2 \leq Ch^{2(p+1)}N^{\delflex(T-t)}\|\varphi\|^2_{C^{\lceil d/2 \rceil + p + 5}}.
\end{align}
\end{itemize}
\end{lemma}


\begin{proof}
\emph{Part I)}
The Euler MacLaurin summation formula of Lemma \ref{lem_EulerMaclaurin}, the regularity of the continuous mean field limit $\overline{\rho}^{r_I}$, and Lemma \ref{prop_MFLconsistency}, entail the bound
 \begin{align}\label{BoundReplaceU}
 \left| \U[\alpha] - \U[h,\alpha] \right| \lesssim (\max_{\alpha,\beta}\|\nabla\V[\alpha\beta]\|_{C^{p+2}})\|\overline{\rho}^{r_I}\|_{C^{p+3}}N^{\delta t}h^{p+1}.
 \end{align}
Performing the replacement $\U[\alpha] \rightarrow \U[h,\alpha]$, as well as the replacements $\Delta \rightarrow \Delta_h$ and $\nabla \rightarrow \nabla_h$ of order $p+2$ (these being possible since $\phi_\alpha\in C^{p+1}(\domain)$) we rewrite \eqref{eq_backEvoTest} as 
\begin{align}\label{rewrite_equn_phi}
-\partial_t \Ih \phi_{\alpha}^t & =  \sigma_\alpha \Delta_h \Ih\phi_{\alpha}^t - \U[h,\alpha](t)  \cdot \nabla_h \Ih\phi_{\alpha}^t \nonumber\\
& \quad \quad \quad + \Ih \left[ \sum_{\beta=1}^\nS V_{\beta\alpha}^{r_I} \ast \nabla \cdot  (\overline{\rho}^{r_I}_{\beta} \nabla \phi_{\beta}^t)\right] + R_{\nabla,\Delta} + R_{\U},
\end{align}
where 
\begin{align}\label{Boundf1}
\left|R_{\nabla,\Delta}\right| + \left| R_{\U}\right| \lesssim \|\phi^t\|_{C^{p+3}}(\max_{\alpha,\beta}\|\nabla\V[\alpha\beta]\|_{C^{p+2}})\|\overline{\rho}^{r_I}\|_{C^{p+3}}N^{\delta t}h^{p+1}.
\end{align} 

By adding and subtracting zero, performing the replacement $\rhobr[] \rightarrow \rhobr[h]$ (with the associated residual bounded using Lemma \ref{prop_MFLconsistency}), and performing the replacement $\ast \rightarrow \ast_h$ (with the associated residual bounded by the Euler-Maclaurin summation formula in Lemma \ref{lem_ApprOrderConv}, which we can use due to \eqref{RegularityConvolutionL2}) 
we can rewrite the difference of \eqref{eq_discrBackEvoTest} and \eqref{rewrite_equn_phi} as 
\begin{align}\label{eq100}
& -\partial_t ( \phi_{h,\alpha}^t - \Ih \phi_{\alpha}^t )\nonumber \\
& \quad  = \sigma_\alpha \Delta_h ( \phi^t_{h,\alpha} - \Ih \phi_{\alpha}^t ) - \U[h,\alpha](t)  \cdot \nabla_h (\phi_{h,\alpha}^t - \Ih\phi_{\alpha}^t) \nonumber\\
& \quad \quad + \sum_{\beta=1}^\nS \Ih \nabla V_{\beta\alpha}^{r_I} \ast_{h} \left[\Ih (\rhobr[h,\beta]\nabla \phi_{\beta}^t) \right] - \sum_{\beta=1}^\nS \Ih \nabla V_{\beta\alpha}^{r_I} \ast_{h} \left(\rhobr[h,\beta] \nabla_h \phi_{h,\beta}^t\right) \nonumber\\
& \quad \quad  + R_{\ast} + R_{\nabla,\Delta} + R_{\U} + R_{\rhobr[]} \nonumber \\
& =: \text{Diff}_1 + \text{React}_1 + \text{Conv}_1 - \text{Conv}_2 + R_{\ast} + R_{\nabla,\Delta} + R_{\U} + R_{\rhobr[]}, 
\end{align}
where 
\begin{align}\label{Boundf2}
\left| R_{\rhobr[]} \right| & \lesssim (\max_{\alpha,\beta}\|\V[\alpha\beta]\|_{W^{1,2}})\|\phi^t\|_{C^1}h^{p+1}N^{\delta (T-t)},\\
\left| R_{\ast} \right| & \lesssim \|\phi^t\|_{C^{p+2}}\|(\max_{\alpha,\beta}\|\nabla\V[\alpha\beta]\|_{C^{p+1}})\|\overline{\rho}^{r_I}\|_{C^{p+1}}h^{p+1}.\label{Boundf3}
 \end{align}
Testing \eqref{eq100} with $\phi_{h,\alpha}^t - \Ih \phi_{\alpha}^t$, using Young's inequality, integration by parts, \eqref{Boundf1}--\eqref{Boundf3}, and the comparison between the first-order operators (\eqref{IntByParts} and \ref{ass_DiscrDefOps}) gives
\begin{align}\label{L2Estimate}
& -\partial_t \| \phi_{h,\alpha}^t - \Ih \phi_{\alpha}^t  \|^2_h \nonumber\\ 
& \quad = -\sigma_\alpha \left( \nabla_{D,h} ( \phi^t_{h,\alpha} - \Ih \phi_{\alpha}^t ),  \nabla_{D,h}( \phi_{h,\alpha}^t - \Ih \phi_{\alpha}^t) \right)_h \nonumber\\ 
& \quad\quad \quad - \left( \U[h,\alpha](t)  \cdot \nabla_h (\phi_{h,\alpha}^t - \Ih\phi_{\alpha}^t) , \phi_{h,\alpha}^t - \Ih \phi_{\alpha}^t \right)_h \nonumber\\
& \quad\quad \quad + \left(\text{Conv}_1 - \text{Conv}_2, \phi_{h,\alpha}^t - \Ih \phi_{\alpha}^t \right)_h \nonumber\\
& \quad\quad \quad + \left(R_{\ast} + R_{\nabla,\Delta} + R_{\U} + R_{\rhobr[]}, \phi_{h,\alpha}^t - \Ih \phi_{\alpha}^t \right)_h \nonumber\\
& \quad\lesssim - (\sigma_\alpha/2) \|\nabla_h(\phi_{h,\alpha}^t - \Ih \phi_{\alpha}^t) \|_h^2 \nonumber\\
& \quad\quad \quad +  (\| \U[h,\alpha](t)\|^2_{\infty} + 1) \| \phi_{h,\alpha}^t - \Ih \phi_{\alpha}^t \|_h^2 \nonumber\\
& \quad\quad \quad + \left(\text{Conv}_1 - \text{Conv}_2, \phi_{h,\alpha}^t - \Ih \phi_{\alpha}^t \right)_h + \left| R_{\ast} + R_{\nabla,\Delta} + R_{\U} + R_{\rhobr[]} \right|^2.
\end{align}

The regularity assumption \eqref{RegularityConvolutionL2} (enabling a further replacement $\nabla \rightarrow \nabla_h$)  entails
\begin{align*}
& \left(\text{Conv}_1 - \text{Conv}_2, \phi_{h,\alpha}^t - \Ih \phi_{\alpha}^t \right)_h \\
& \quad = \sum_{\beta=1}^\nS \left(\Ih \nabla V^{r_I}_{\beta\alpha} \ast_{h} \left[\Ih  (\rhobr[h,\beta] \nabla \phi_{\beta}^t) -  \left(\rhobr[h,\beta] \nabla_h \phi_{h,\beta}^t\right)\right] , \phi_{h,\alpha}^t - \Ih \phi_{\alpha}^t\right)_h \\
& \quad = \sum_{\beta=1}^\nS \left(\Ih \nabla V^{r_I}_{\beta\alpha} \ast_{h} \left[\left(\rhobr[h,\beta] \nabla_h(\Ih\phi_{\beta}^t -\phi^t_{h,\beta})\right)\right], \phi_{h,\alpha}^t - \Ih \phi_{\alpha}^t \right)_h \\
& \quad \quad \quad + \left(R_{\nabla},  \phi_{h,\alpha}^t - \Ih \phi_{\alpha}^t \right)_h,
\end{align*}
where $\left|R_{\nabla}\right| \leq \|\rhobr\|_{C^{1}}\|\phi^t\|_{C^{p+2}}\|\V[\alpha\beta]\|_{W^{1,1}}h^{p+1}$ uses \eqref{eq_MFLhO_estimateO1}. Using the Young inequality for convolutions and the symmetry of the kernel $ V_{\beta\alpha}^{r_I}$, we carry on and deduce 
\begin{align}\label{Conv1_2}
& \left(\text{Conv}_1 - \text{Conv}_2, \phi_{h,\alpha}^t - \Ih \phi_{\alpha}^t \right)_h \nonumber\\
& \quad = - \sum_{\beta=1}^\nS \left(\nabla_h (\Ih V_{\beta\alpha}^{r_I}) \ast_{h} \left[\phi_{h,\beta}^t - \Ih \phi_{\beta}^t\right], \rhobr[h,\beta]  \nabla_h(\Ih\phi_{\alpha}^t -\phi^t_{h,\alpha}) \right)_h \nonumber\\
& \quad \quad \quad + \left(R_{\nabla},  \phi_{h,\alpha}^t - \Ih \phi_{\alpha}^t \right)_h \nonumber\\
& \quad \lesssim  \sum_{\beta=1}^{\nS} \|\rhobr\|_{C^{1}}\|(\max_{\alpha,\beta}\|\V[\alpha\beta]\|_{W^{1,1}}) \|\phi_{h,\beta}^t - \Ih \phi_{\beta}^t\|_h \| \nabla_h(\Ih\phi_{\alpha}^t -\phi^t_{h,\alpha})  \|_h \nonumber\\
& \quad \quad \quad + \left(R_{\nabla},  \phi_{h,\alpha}^t - \Ih \phi_{\alpha}^t \right)_h.
\end{align}
Summing over all species $\alpha\in 1,\dots,\nS$ in \eqref{Conv1_2}, using the Young inequality  in \eqref{L2Estimate} (with weights suitable for an absorbtion argument), the inequality \eqref{1stDerBound}, the norm equivalence $\|\nabla^R_h \cdot\|=\|\nabla_h \cdot\|$, \eqref{eq_MFLconsistency:main estimate}, and the simple bound $\|\U[h,\alpha](t)\|_\infty \leq \|\V[\alpha\beta]\|_{H^1}\|\rhobr[h]\|_{L^2}$, we obtain
\begin{align*}
 & -\partial_t \| \phi_{h,\alpha}^t - \Ih \phi_{\alpha}^t  \|^2_h \\
 & \quad \lesssim (\max_{\alpha,\beta}{[\|\V[\alpha\beta]\|_{W^{1,1}} \vee \|\V[\alpha\beta]\|_{H^1}}] + 1)^2\left(\|\rhobr\|^2_{C^{p+1}} + \|\rhobr[h]\|^2_{L^2}\right)\|\phi_{h,\alpha}^t - \Ih \phi_{\alpha}^t \|_h^2 \\
& \quad \quad + N^{\delta t}\|\phi^t\|^2_{C^{p+2}}\|\rhobr\|^2_{C^{p+1}} (\max_{\alpha,\beta}{\|\V[\alpha\beta]\|_{C^{p+1}}})^2 h^{2(p+1)}\\
& \quad \lesssim \left[ (\max_{\alpha,\beta}{[\|\V[\alpha\beta]\|_{W^{1,1}} \vee \|\V[\alpha\beta]\|_{H^1}}] + 1)^2\left(\|\rhobr\|^2_{C^{p+1}} + h^{2(p+1)}N^{\delta t} \right) \right] \\
 & \quad \quad \quad \quad \times \|\phi_{h,\alpha}^t - \Ih \phi_{\alpha}^t \|_h^2 \\
& \quad \quad + N^{\delta t}\|\phi^t\|^2_{C^{p+2}}\|\rhobr\|^2_{C^{p+1}} (\max_{\alpha,\beta}{\|\V[\alpha\beta]\|_{C^{p+1}}})^2 h^{2(p+1)}.
\end{align*}
Using now Assumption \ref{ass_Scaling} to bound the term $h^{2(p+1)}N^{\delta t}$, \eqref{L2ErrorTestFunctions} follows promptly using Gronwall Lemma, Lemma \ref{lem_RegularityContinuousTestFunctions}, and the scaling assumption \eqref{ScalingRegimeRadius} (see analogous Gronwall argument in Proposition \ref{prop_MFLconsistency}). 

\emph{Part II)} We argue by induction over $s$. The case $s=0$ is settled by Part I). Now assume that the validity of \eqref{OpErrorTestFunctions} for some $s-1$. Let $\nu\in\Nz^d$ with $|\nu|=s$. Then we can write, for some test function $\eta_h$
\begin{align*}
\big(\partial_t(\dOne{\nu}\phi^t_{h,\alpha}),\eta_h\big)_h & = -\sigma_\alpha \big( \nabla_{D,h} \dOne{\nu}\phi_{h,\alpha},  \nabla_{D,h}\eta_h \big)_h \\
& \quad - \left(\sum_{\gamma \leq |\nu|}{b_{\nu,\gamma}\tau_{h\gamma}\dOne{\nu-\gamma}\U[h,\alpha](t) \cdot \dOne{\gamma}\nabla_h \phi_{h,\alpha}}, \eta_h \right)_h \\
& \quad + \left( \sum_{\beta}\sum_{\gamma \leq \nu}b_{\nu,\gamma}\tau_{h\gamma} \Ih[\nabla\V[\beta\alpha]] \ast_h \dOne{\nu-\gamma}\rhobr[h,\beta] \dOne{\gamma}\nabla_h \phi_{h,\beta}, \eta_h \right)_h,
\end{align*}
where $(b_{\nu,\gamma})_\gamma$ are the standard binomial factors from the product rule.
Analogously to what we have done in Part I), we perform relevant residual substitutions (this time we additionally rely on the mean value theorem, see \eqref{eq_MFLhO_R-dynamicsEst} for an identical discussion) and obtain 
\begin{align*}
\big(\partial_t(\dOne{\nu}\Ih \phi^t_{\alpha}),\eta_h\big)_h & = -\sigma_\alpha \big( \nabla_{D,h} \dOne{\nu}\Ih \phi_{\alpha},  \nabla_{D,h}\eta_h \big)_h \\
& \quad - \left(\sum_{\gamma \leq |\nu|}{b_{\nu,\gamma}\tau_{h\gamma}\dOne{\nu-\gamma}\U[h,\alpha](t) \cdot \dOne{\gamma}\nabla_h \Ih\phi_{\alpha}}, \eta_h \right)_h \\
& \quad + \left( \sum_{\beta}\sum_{\gamma \leq \nu}b_{\nu,\gamma}\tau_{h\gamma} \Ih[\nabla\V[\beta\alpha]] \ast_h \dOne{\nu-\gamma}\rhobr[h,\beta] \dOne{\gamma}\nabla_h \Ih \phi_{\beta}, \eta_h \right)_h \\
& \quad + \left( R_{\ast} + R_{\nabla,\Delta} + R_{\U} + R_{\rhobr[]} + R_{\nabla} , \eta_h \right)_h , 
\end{align*}
where we have
\begin{align*}
\left|R_{\nabla,\Delta}\right| + \left| R_{\U}\right| & \leq \|\phi^t\|_{C^{s+p+3}}(\max_{\alpha,\beta}\|\nabla\V[\alpha\beta]\|_{C^{s+p+2}})\|\overline{\rho}^{r_I}\|_{C^{s+p+3}}N^{\delta t}h^{p+1},\\
\left| R_{\rhobr[]} \right| & \lesssim (\max_{\alpha,\beta}\|\V[\alpha\beta]\|_{W^{1,2}})\|\phi^t\|_{C^{s+1}}h^{p+1}N^{\delta (T-t)},\\
\left| R_{\ast} \right| & \lesssim \|\phi^t\|_{C^{s+p+2}}\|(\max_{\alpha,\beta}\|\nabla\V[\alpha\beta]\|_{C^{s+p+1}})\|\overline{\rho}^{r_I}\|_{C^{s+p+1}}h^{p+1},\\
\left|R_{\nabla}\right| & \lesssim \|\rhobr\|_{C^{s+1}}\|\phi^t\|_{C^{s+p+2}}\|\V[\alpha\beta]\|_{W^{1,1}}h^{p+1}.
\end{align*}
We therefore deduce that
\begin{align}\label{eq100_b}
& -\partial_t ( \partial^{\nu}_{h,1}\phi_{h,\alpha}^t - \dOne{\nu}\Ih  \phi_{\alpha}^t, \eta_h )_h\nonumber \\
& \quad  = -\sigma_\alpha \left(\nabla_{D,h} ( \partial^{\nu}_{h,1}\phi^t_{h,\alpha} - \dOne{\nu}\Ih \phi_{\alpha}), \nabla_{D,h}\eta_h\right)_h \nonumber\\
& \quad \quad - \left(\sum_{\gamma \leq |\nu|}{b_{\nu,\gamma}\tau_{h\gamma}\dOne{\nu-\gamma}\U[h,\alpha](t) \cdot \left[ \dOne{\gamma}\nabla_h \phi_{h,\alpha} - \dOne{\gamma}\nabla_h \Ih\phi_{\alpha}\right]}, \eta_h \right)_h  \nonumber\\
& \quad \quad + \left( \sum_{\beta}\sum_{\gamma \leq \nu}b_{\nu,\gamma}\tau_{h\gamma} \Ih[\nabla\V[\beta\alpha]] \ast_h \dOne{\nu-\gamma}\rhobr[h,\beta] \left[ \dOne{\gamma}\nabla_h \Ih \phi_{\beta} - \dOne{\gamma}\nabla_h \phi_{h,\beta} \right], \eta_h \right)_h \nonumber\\
& \quad \quad  + \left( R_{\ast} + R_{\nabla,\Delta} + R_{\U} + R_{\rhobr[]} +R_{\nabla}, \eta_h \right)_h.  
\end{align}
Then \eqref{OpErrorTestFunctions} follows testing \eqref{eq100_b} with $\eta_h = \partial^{\nu}_{h,1}\phi_{h,\alpha}^t - \dOne{\nu}\Ih  \phi_{\alpha}^t$, using the induction hypothesis and the same Gronwall type argument used in Part I).

\emph{Part III)} 
The fact that the replacement $\nabla \rightarrow \nabla_h$ is of order $p+1$, and the fact that $\dOne{1}$ bounds any discrete first derivative (in terms of $L^p$ norms), imply
\begin{align*}
\left\| \nabla_h \phi_{h,\alpha} - \Ih \nabla\phi_{\alpha}\right\|_h & \lesssim \left\| \nabla_h \phi_{h,\alpha} - \nabla_h\phi_{\alpha}\right\|_h + \left\| \nabla_h\phi_{\alpha} - \nabla\phi_{\alpha}\right\|_h \\
& \lesssim \| \phi_{h,\alpha} - \phi_{\alpha} \|_{H^1_h} + h^{p+1}\|\phi_{\alpha}\|_{C^{p+2}},
\end{align*}
and \eqref{Diff_Gradients} follows from \eqref{OpErrorTestFunctions}.

\emph{Part IV)} This follows by combining \eqref{OpErrorTestFunctions}, Lemma \ref{lem_RegularityContinuousTestFunctions}, Lemma \ref{HsBoundDiscrete}, and the Sobolev embedding $H^{d/2+}\subset L^{\infty}$.

\end{proof}

\begin{lemma}[$L^1$ bound for discrete test functions]\label{L1_bound_discrete}
Let $\phi_h^t$ be the solution to \eqref{eq_discrBackEvoTest}. 
Let $\delflex>0$.
Then, for large enough $N$, $\phi_h^t$ satisfies the $L^1$ bound
\begin{align}\label{eq_L1BoundDiscrTestFunc}
\sup_{t\in[0,T]}{\|\phi^t_h\|_{L^1(\Ghd)}} \leq N^{\delflex T}\|\Ih \varphi\|_{L^1(\Ghd)}.
\end{align}
\end{lemma}

\begin{proof}
We expand the solution $\phi_h$ using the discrete Green's function $\mathcal{G}_h$ associated with the discrete Laplace $\Delta_h$, obtaining
\begin{align*}
  \phi_{h,\alpha}^{T-t}(x) & = (\mathcal{G}_h^t \ast_h \Ih \varphi)(x) + \int_{T-t}^{T}{\left(\mathcal{G}_h^{T-s}(x-\cdot) , \left[ - \U[h,\alpha](s)\cdot\nabla_h\phi^{s}_{h,\alpha} \right]\right)_h\m s} \\
& \quad \quad + \int_{T-t}^{T}{\left(\mathcal{G}_h^{T-s}(x-\cdot), \sum_{\beta=1}^\nS \Ih[\nabla\V[\beta\alpha]]\ast_{h,c}(\rhobr[h,\beta]\nabla_h\phi^s_{h,\beta})  \right)_h\m s}
\end{align*}
Using integration by parts so as to remove all first derivatives from $\phi_h$, we get 
\begin{align*}
 \phi_{h,\alpha}^{T-t}(x) &= (\mathcal{G}_h^t \ast_h \Ih \varphi)(x) + \int_{T-t}^{T}{\left( \phi^{s}_{h,\alpha},\nabla_h \cdot \left[ \mathcal{G}_h^{T-s}(x-\cdot) \U[h,\alpha](s) \right] \right)_h\m s} \\
& \quad \quad  
- \int_{T-t}^{T}{\left(\nabla_h \mathcal{G}_h^{T-s}(x-\cdot), \sum_{\beta=1}^\nS \Ih[\nabla\V[\beta\alpha]]\ast_{h,c}(\rhobr[h,\beta] \phi_{h,\beta}^s) \right)_h\m s} \\
& \quad \quad  - \int_{T-t}^{T}{\left(\mathcal{G}_h^{T-s}(x-\cdot), \sum_{\beta=1}^\nS \Ih[\nabla\V[\beta\alpha]]\ast_{h,c}(\nabla_h[\rhobr[h,\beta]] \phi^s_{h,\beta}) \right)_h\m s}.
\end{align*}
Using the inequality 
\begin{align*}
\|\nabla_h \mathcal{G}^{T-s}_h(\cdot-y)\|_{L^1(\Ghd)} \lesssim (T-s)^{-1/2}, 
\end{align*}
the bound $\| \U[h,\alpha] \|_{W^{1,\infty}} \leq \| \V[\beta\alpha] \|_{H^2}\|\rhobr[h]\|_{W^{1,\infty}(\Ghd)}$,
the discrete Young's convolution inequality given in Corollary \ref{cor_DiscrYoung}), and summing up over all species, we obtain 
\begin{align*}
\sum_{\alpha=1}^{\nS}\|\phi_{h,\alpha}^{T-t}\|_{L^1(\Ghd)} & \leq \|\Ih \varphi\|_{L^1(\Ghd)} + \| \V[\beta\alpha] \|_{H^2}\|\rhobr[h]\|_{W^{1,\infty}(\Ghd)} \\
& \quad \quad \quad \times \int_{T-t}^{T}{(1+(T-s)^{-1/2})\sum_{\alpha=1}^{\nS} \|\phi^s_{h,\alpha}\|_{L^1(\Ghd)}\m s}.
\end{align*}
Taking the supremum for $t\in [0,\tilde{t}]$, we deduce
\begin{align}\label{L1_before_absorbtion}
& \sup_{t\in [0,\tilde{t}]}\sum_{\alpha=1}^{\nS} \|\phi^{T-t}_{h,\alpha}\|_{L^1(\Ghd)}\nonumber \\
 & \quad \leq \|\Ih \varphi\|_{L^1(\Ghd)} + \| \V[\beta\alpha] \|_{H^2}\|\rhobr[h]\|_{W^{1,\infty}(\Ghd)} \nonumber\\
&  \quad \quad \quad \times \left[\int_{T-\tilde{t}}^{T}{(1+(T-s)^{-\frac{1}{2}})\m s}\right] \sup_{t\in [0,\tilde{t}]}\sum_{\alpha=1}^{\nS} \|\phi^{T-t}_{h,\alpha}\|_{L^1(\Ghd)}.
\end{align}
Choosing 
$
\tilde{t} \leq C \| \V[\beta\alpha] \|^{-2}_{H^2}\|\rhobr[h]\|^{-2}_{W^{1,\infty}(\Ghd)}
$
with small enough $C$ allows to perform an absorbtion argument in \eqref{L1_before_absorbtion}, leading to 
\begin{align*}
& \sup_{t\in [0,\tilde{t}]}\sum_{\alpha=1}^{\nS} \|\phi^{T-t}_{h,\alpha}\|_{L^1(\Ghd)}  \leq 2\|\Ih \varphi\|_{L^1(\Ghd)}.
\end{align*}
Repeating the analysis over all time intervals $[j\tilde{t},(j+1)\tilde{t}]$ up to saturation of $[0,T]$ (there are $\lceil T/\tilde{t} \rceil \lesssim \| \V[\beta\alpha] \|^2_{H^2}\|\rhobr[h]\|^2_{W^{1,\infty}(\Ghd)}T$ of such intervals) we get 
\begin{align*}
\sup_{t\in[0,T]}{\|\phi_h(t)\|_{L^1(\Ghd)}} \leq 2^{T\|\V[\beta\alpha] \|^2_{H^2}\|\rhobr[h]\|^2_{W^{1,\infty}(\Ghd)}} \|\Ih \varphi\|_{L^1(\Ghd)},
\end{align*}
and the proof is now concluded bounding $\|\rhobr[h,\beta]\|_{W^{1,\infty}}$ using \eqref{eq_MFLhO_estimateO1}, and bounding $\|\V[\beta\alpha] \|^2_{H^2}$ using the scaling \eqref{ScalingRegimeRadius}.
\end{proof}

\section{Constructing discrete initial data via interpolation}\label{AppConstructionInitialData}
We substantiate Remark \ref{rem_DiscrInitData}: specifically, we show how to construct discrete initial data satisfying Assumption \ref{ass_InitCond} from given continuous initial data (i.i.d.) by redistributing mass to an appropriate set of surrounding grid points according to interpolation weights.

\subsection{The multivariate interpolation scheme}
We briefly discuss the polynomial multivariate interpolation scheme of order $p+1$.
First, for each $y\in\domain$, we have to fix the set of surrounding grid points $B_h^p(y)$ onto which to distribute mass from the continuous initial data.
For schemes of order $p+1$ scheme (interpolating polynomials of order at most $p$ exactly) we need $\binom{p+d}{d}$ points in a suitable formation. 
As a subset of our rectangular grid we choose, e.g., the ``lower left triangular cone'' defined as
\begin{equation*}
	B_h^p := \big\lbrace x\in\Ghd;~ \forall =1,\ldots,d: 0\leq x_i,~ x_1+\ldots+x_d \leq ph  \big\rbrace,
\end{equation*}
see \cite[Section 2.2]{gasca2000polynomial}.
Thus, for $y\in[0,h)^d$ we choose $B_h^p(y)=B_h^p$.
In case $y$ is in another grid section, we appropriately shift the set of interpolation points: $B_h^p(y)=B_h^p+h\floor{y/h}$.
With this choice, for $f\in C^{p+1}$ and $y\in[0,h)^d$ we have the error estimate 
\begin{equation}\label{eq_interpolationLagrange1section}
	f(y) 
	=
	\sum_{x\in B_h^p} f(x)\Lambda_x(y) + \mathcal{O}\big(\|\vphi\|_{C^{p+1}}h^{p+1}\big),
\end{equation}
where the $\Lambda_x$ are the Lagrange fundamental polynomials with $\Lambda_x(\tilde{x})=\delta_{x\tilde{x}}$ for $\tilde{x}\in B_h^p$, see \cite[Section 4.1]{gasca2000polynomial}.
Note though that we can rephrase $\Lambda_x(y)$ as samples $\Lambda(x-y)$ from a continuously defined convolution kernel 
\begin{equation*}
	\Lambda\colon\domain\rightarrow\R,\quad
	\Lambda(z)\coloneqq 	\left\lbrace\begin{split}
								\Lambda_x(x-z),	&\qquad\text{for }x-z\in [0,h)^d,~ x\in B_h^p,\\
								0, 				&\qquad\text{otherwise}.			
							\end{split}\right.
\end{equation*}
Note that in the interior of each grid section the kernel $\Lambda$ is a polynomial and thus continuous. 
Due to the definition of the Lagrange polynomials, $\Lambda$ is further continuous at all grid points, but in general discontinuous on the remaining contact areas between grid sections. 

With this definition we can extend \eqref{eq_interpolationLagrange1section} to all $y\in\domain$ as
\begin{align}\label{eq_interpolationLagrange}
	f(y) 
	=
	\sum_{x\in\Ghd} f(x)\Lambda(x-y) + \mathcal{O}\big(\|\vphi\|_{C^{p+1}}h^{p+1}\big).
\end{align}

\subsection{Constructing the discrete initial data}
To construct discrete initial data from the continuous one, for $\alpha=1,\ldots,\nS$ we would like to straight-up distribute the mass of a particle starting at $X^{\rI,N}_{\alpha,i}(0)$ according to the interpolation weights, that is for $x\in\Ghd$
\begin{equation*}
	\rho^{0,\ast}_{h,\alpha}(x) 
	\coloneqq 
	h^{-d}\int_\domain \Lambda(x-y) \m\empmeas[0,\alpha](y)
	= h^{-d}\big(\Lambda\ast\empmeas[0,\alpha]\big)(x),
\end{equation*}
where the volume factor $h^{-d}$ comes from the interpretation as a density. 
In order to abide by the wanted properties of the fluctuations (Assumption \ref{ass_InitCond}), we have to adjust for the fact that the discrete mean field limit is initialized as $\rhobz[h]=\Ih[\rhobz]$.
Thus, for $x\in\Ghd$ we set
\begin{equation*}
	\rhoz[h,\alpha](x) 
	\coloneqq 
	h^{-d}\big(\Lambda\ast\empmeas[0,\alpha]\big)(x)
	-h^{-d}\big(\Lambda\ast\rhobz[\alpha]\big)(x)
	+\Ih[\rhobz[\alpha]](x).
\end{equation*}
From \eqref{eq_interpolationLagrange} it follows that this definition satisfies \eqref{InitDataFluctInequality}..

With respect to \eqref{eq_ass_InitialLinftyh} and \eqref{eq_ass_InitialHlh}, note that since the particles are initially i.i.d.,
Hoeffding's inequality implies for any $m \in \Z^d\cap\left[-\frac{h}{\pi},\frac{h}{\pi}\right)^d$
\begin{align*}
	\bigg|h^{d}\sum_{x\in\Ghd} (\rhoz[h]-\rhobz[h])(x) \exp(i m\cdot x)  \bigg|^2
	\leq \mathcal{C}^2 N^{-1}
\end{align*}
where $\mathcal{C}$ is subject to a Gaussian moment bound. 
Multiplying with $(1+|m|)^{-d-1}$ and summing in $m$, we deduce
$
	\|\rhoz[h]-\rhobz[h]\|_{H^{-(\floor{d/2}+1)}_h} \leq \mathcal{C} N^{-1/2},
$
which immediately implies that \eqref{eq_ass_InitialHlh} is satisfied.
Similarly, by Bernstein's inequality  we obtain
\begin{align*}
	\Pm\left[|\rhoz[h](x)-\rhobz[h](x)|\geq \sqrt{\frac{C N \operatorname{Var}\rhoz[h](x)}{N}} R + \frac{C |N \max_\Omega \rhoz[h](x)| R^2}{N}\right]
	\leq \exp(-R^2),
\end{align*}
for every $x\in \Ghd$ and any $R\geq 1$, and thus by bounding the local mollifier $h^{-d} \Lambda$
\begin{align*}
\Pm\bigg[|\rhoz[h](x)-\rhobz[h](x)|\geq C N^{-1/2} h^{-d/2} R + C h^{-d} N^{-1} R^2 \bigg]
\leq \exp(-R^2).
\end{align*}
This implies that this set of initial conditions satisfies \eqref{eq_ass_InitialLinftyh}.
The only parts of Assumption \ref{ass_InitCond} left are positivity and the mass restriction. 
Due to $\rhobz\geq \rho_{min}$ and \eqref{eq_ass_InitialLinftyh}, the probability of negative initial data decays exponentially -- hence, for large enough $N$, we can restrict to positive data at the cost of only an additional insignificant error in \eqref{InitDataFluctInequality}.
Since the interpolation weights preserve mass, for positive initial data we have $\|\rhoz[h,\alpha]\|_{L^1_h}=\|\Ih[\rhobz[\alpha]]\|_{L^1_h}\approx 1$ for $h$ small enough.
Thus, restricting to $\rhoz[h]\geq 0$, this construction satisfies Assumption \ref{ass_InitCond}.

\end{appendix}


{\bfseries Acknowledgements}. All authors gratefully acknowledge funding from the Austrian Science Fund (FWF) through the project F65. 
CR gratefully acknowledges support from the Austrian Science Fund (FWF), grants P30000, P33010, W1245.
FC gratefully acknowledges funding from the European Union’s Horizon 2020 research and innovation programme under the Marie Sk\l{}odowska-Curie grant agreement No. 754411.
\bibliographystyle{abbrvnat}
\bibliography{fluctuations}

\end{document}